\pdfoutput=1
\documentclass[11pt, reqno]{amsart}
\usepackage{amsmath,amsthm,amsfonts,amssymb,mathrsfs,bm,graphicx,stmaryrd}
\usepackage{epstopdf}
\usepackage{mathtools}
\usepackage{dsfont}
\usepackage{multicol}
\usepackage[colorlinks=true,linkcolor=blue,citecolor=blue,breaklinks]{hyperref}
\usepackage{url}

\usepackage{breakurl}
\usepackage{bbm}

\newcommand{\QQ}{\mathbb{Q}}
\newcommand{\HH}{\mathbb{H}}
\newcommand{\wmu}{\widetilde{\mu}}
\newcommand{\cX}{\mathcal{X}}
\newcommand{\RR}{\mathbb{R}}
\newcommand{\CC}{\mathbb{C}}
\newcommand{\PP}{\mathbb{P}}
\newcommand{\DD}{\mathbb{D}}
\newcommand{\fp}{\mathfrak{p}}
\newcommand{\TT}{\mathbb{T}}
\newcommand{\fB}{\mathfrak{B}}
\newcommand{\cB}{\mathcal{B}}

\newcommand{\rr}{\mathbbm{r}}
\newcommand{\EE}{\mathbb{E}}
\newcommand{\NN}{\mathbb{N}}
\newcommand{\tth}{\mathtt{h}}

\newcommand{\fh}{\mathfrak{h}}
\newcommand{\coal}{\mathtt{Coal}}
\newcommand{\good}{\mathtt{G}}

\newcommand{\dist}{\mathrm{dist}}

\newcommand{\A}{\mathbb{A}}
\newcommand{\ZZ}{\mathbb{Z}}
\newcommand{\cV}{\mathcal{V}}
\newcommand{\cI}{\mathcal{I}}
\newcommand{\fq}{\mathfrak{q}}

\usepackage[letterpaper,hmargin=1.0in,vmargin=1.0in]{geometry}
\parskip 	\smallskipamount

\newtheorem{theorem}{Theorem}
\newtheorem{lemma}[theorem]{Lemma}
\newtheorem{definition}[theorem]{Definition}
\newtheorem{corollary}[theorem]{Corollary}
\newtheorem{proposition}[theorem]{Proposition}

\newtheorem{remark}[theorem]{Remark}

\newcommand{\ind}{\mathbbm{1}}
\def\Var{{\rm Var}}

\newcommand{\cD}{\mathcal{D}}
\newcommand{\cE}{\mathcal{E}}

\newcommand{\cM}{\mathcal{M}}

\newcommand{\cS}{\mathcal{S}}

\newcommand{\cU}{\mathcal{U}}
\newcommand{\diam}{\mathrm{diam}}

\usepackage[backend=biber,style=alphabetic,doi=false,maxalphanames=10,maxnames=50]{biblatex}
\AtBeginBibliography{\small}

\begin{document}
\title[]{The $\boldsymbol{d_\gamma/2}$-variation of distance profiles in {\large$\boldsymbol{\gamma}$}-Liouville quantum gravity}
\author[]{Manan Bhatia}
\address{Manan Bhatia, Department of Mathematics, Massachusetts Institute of Technology, Cambridge, MA, USA}
\email{mananb@mit.edu}
\date{}
\maketitle
\begin{abstract}
  For Brownian surfaces with boundary and an interior marked point, a natural observable to consider is the distance profile, defined as the process of distances from the marked point to a variable point $x$ lying on the boundary. When the boundary is parametrized by the natural length measure on it, this distance profile turns out to be locally absolutely continuous to Brownian motion, and as a result, the boundary length measure itself has a natural interpretation as the quadratic variation process of the distance profile. In this paper, we extend this interpretation to $\gamma$-Liouville quantum gravity ($\gamma$-LQG), a one-parameter family of models of random geometry which is known to specialize to the case of Brownian geometry for the case $\gamma=\sqrt{8/3}$. With $d_\gamma$ denoting the Hausdorff dimension of $\gamma$-LQG, we show that for a $\gamma$-LQG surface with boundary, the natural boundary length measure can be interpreted (up to a constant factor) as the $d_\gamma/2$-variation process of the distance profile from an interior point.
\end{abstract}

\tableofcontents
\section{Introduction}
\label{sec:intro}
Expected to arise as the scaling limits of a variety of natural planar map models, $\gamma$-Liouville quantum gravity ($\gamma$-LQG) \cite{She22} is a canonical one parameter family of continuum random geometries indexed by the parameter $\gamma\in (0,2)$. In fact, predating the construction \cite{DDDF20,GM21} and investigation of $\gamma$-LQG, the special case of \emph{uniform} planar maps and their scaling limits \cite{LeGal11,Mie13} has been deeply studied in itself, as part of the field known as Brownian geometry \cite{LeGal19}. As was shown in the works \cite{MS15,MS20,MS16,MS21}, Brownian geometry can be realized a special case of the theory of $\gamma$-LQG, by specializing to the value $\gamma=\sqrt{8/3}$. However, even though there are a number of interesting results on the $\gamma$-LQG metric for general values of $\gamma$ (e.g.\ \cite{Gwy19,Gwy21,GS22}), the understanding is significantly limited when compared to the Brownian case, with the primary barrier being the lack of any integrability of the metric structure when $\gamma\neq \sqrt{8/3}$. Among the handful of phenomena observed in Brownian geometry which can be proven to transcend to the general case, the chief example is that of ``geodesic coalescence'' (or confluence) \cite{LeGal10,AKM17,GM20,GPS20}, which refers to the tendency of geodesics to merge with nearby geodesics.

In this paper, we present another phenomenon which does generalize from the case of Brownian surfaces, and this concerns an object which we call the distance profile. For a natural Brownian surface with the disk topology (Brownian disk) \cite{BM17} having an interior marked point, it is known that the distance profile, or the process of distances from the fixed marked point to a variable point $x$ on the boundary of the disk, yields a Brownian bridge \cite{BM17}, provided that we parametrize the boundary by the natural length measure associated to it. As a result of this Brownianity, the above-mentioned natural boundary length measures can be interpreted as the quadratic variation process associated to the distance profile. We note that by an absolute continuity argument, it can be shown that this phenomenon is true for general Brownian surfaces with boundary and a marked point, since the associated distance profile is still locally absolutely continuous to Brownian motion.

For a $\gamma$-LQG surface with boundary, one can similarly define a distance profile, by first fixing a bulk point and then considering distances to the boundary. While the local behaviour of this distance profile does not appear to be explicit as in the case of the Brownian surfaces, one might wonder if it is still true that, up to a deterministic constant factor, the natural $\gamma$-LQG boundary length measure can be obtained as a variation process of the distance profile, and the goal of this paper is to prove a version of this statement. Specifically, we show that the above boundary length measure is equal to the $d_\gamma/2$-variation of the distance profile, where $d_\gamma$ refers to the Hausdorff dimension of $\gamma$-LQG as a metric space. We note that since the Hausdorff dimension $d_{\sqrt{8/3}}$ of Brownian surfaces is known \cite{LeGal07} to be $4$, the above is consistent with the appearance of the quadratic variation in Brownian geometry.

Though the above (Corollary \ref{cor:main}) is the central result in this paper, we go further and show that (Theorem \ref{thm:2}) for any $\gamma'\in (0,2)$, after proper renormalization in the case $\gamma'\neq \gamma$, the $\gamma'd_\gamma/(2\gamma)$-variation process of the distance profile can be defined and is equal to the boundary $\gamma'$-LQG measure. We emphasize here that for the case $\gamma'\neq \gamma$, a renormalization process is required to define this ``variation'', and thus, in some sense, this case does not yield a true variation process. Nevertheless, this provides an interesting construction of the $\gamma'$-LQG boundary length measures using the $\gamma$-LQG metric. We now pause to introduce the basic important objects studied in LQG and then subsequently return to the discussion. %

\subsection{The primary observables in $\boldsymbol{\gamma}$-LQG}%
Expected to describe the scaling limit of the distances, areas and boundary lengths corresponding to random planar maps decorated by statistical physics models, $\gamma$-LQG comes equipped with a rich class of observables involving a Gaussian free field (GFF) $h$, a metric $D^\gamma_h$, a bulk measure $\widetilde{\nu}^\gamma_h$, and a boundary measure $\nu^\gamma_h$. Throughout this section, we work with a simple connected domain $U\subseteq \CC$ whose boundary $\partial U$ is a simple and smooth curve, and let $h$ be a variant of a Gaussian free field on $U$, and we note that $h$ is $\log$-correlated and is thus not pointwise defined, but instead defined as a generalized function. The simplest LQG observable to construct is the bulk measure $\widetilde{\nu}^\gamma_h$, and can be defined as the a.s.\ \cite{DS09,SW17} weak limit
\begin{equation}
  \label{eq:193}
  \widetilde{\nu}^\gamma_h=\lim_{\varepsilon \rightarrow 0}\varepsilon^{\gamma^2/2}e^{\gamma h_\varepsilon(z)}d^2z,
\end{equation}
where $h_\varepsilon(z)$ denotes the average of $h$ on a circle of radius $\varepsilon$ around $z$. We note that the existence of the above limit can be seen as a special case of the general construction of Kahane's Gaussian multiplicative chaos measures \cite{Kah85, RV13}. In fact, as long as $h$ is locally absolutely continuous to a free boundary GFF around points of $\partial U$, a similar construction allows one to define a natural boundary length measure $\nu^\gamma_h$ on the set $\partial U$. For the case $U=\HH$, with $dx$ denoting the Lebesgue measure on $\partial \RR$, $\nu^\gamma_h$ is defined as the a.s.\ weak limit
\begin{equation}
  \label{eq:195}
  \nu^\gamma_h=\lim_{\varepsilon\rightarrow 0} \varepsilon^{\gamma^2/4}e^{\gamma h_\varepsilon(x)/2} dx,
\end{equation}
and for general domains $U$, the measure $\nu_h^\gamma$ can be defined by conformally mapping to $\HH$ and using conformal covariance (see e.g.\ \cite[Definition 6.40]{BP23}).

Finally, owing to the non-local dependence of distances with respect to the background noise, the construction of the LQG metric $D^\gamma_h$ was more difficult and was achieved in the sequence of works \cite{DDDF20,DFGPS20,GM19,GM19+,GM20,GM21}. For the case when $U=\CC$ and $h$ is a whole plane GFF and $d_\gamma$ is a constant defined in \cite{DG20} (which later turns out to be the Hausdorff dimension of $\gamma$-LQG), the above-mentioned works defined $D^\gamma_h$ to be the limit in probability of prelimiting metrics $D_{h,\varepsilon}^\gamma$, with the latter being defined by
\begin{equation}
  \label{eq:196}
  D_{h,\varepsilon}^\gamma(u,v)=(a^\gamma_\varepsilon)^{-1}\inf_{P\colon z\rightarrow w}\int_0^1e^{(\gamma/d_\gamma) h_{\varepsilon}^*( P(t) )} | P'(t)| dt
  \end{equation}
  for $u,v\in \CC$, and we note that $a^\gamma_\varepsilon$ is a normalizing constant that is chosen appropriately. In the above expression, the infimum is over all piecewise continuously differentiable paths $P$ from $u$ to $v$, and $h_\varepsilon^*$ is the regularization of the GFF $h$ by convolution with the heat kernel at time $\varepsilon^2/2$. Though the domain $U$ was taken to be $\CC$ in the above, one can define the LQG metric for a general domain $U$ by using local absolute continuity of GFFs with different boundary conditions (see \cite[Remark 1.5]{GM21}). In fact, as shown recently \cite{HM22}, for a domain $U$ and a GFF $h$ on $U$ with free boundary conditions, it is possible to continuously extend the metric $D^\gamma_h$ to the boundary (see Proposition \ref{prop:6}), and this will be important for us. We also note that the observables $(\widetilde{\nu}^\gamma_h,\nu^\gamma_h,D^\gamma_h)$ defined above are all  measurable with respect to the field $h$. %
We now move towards a precise statement of the main result of the paper.

\subsection{The main result}  As mentioned earlier, we consider a simply connected domain $U\subseteq \CC$ with a smooth and simple boundary, and let $h$ be a free boundary GFF in $U$. After fixing a point $z\in U$, we consider the distance profile $D_h^\gamma(z,x)$ as a one-parameter process in the variable $x\in \partial U$. That this process is $\chi-$ H\"older continuous in $x$, for some $\chi>0$ has been shown recently in \cite{HM22}. We study the variation regularity of the function $x\mapsto D_h^\gamma(z,x)$ by choosing a partition $P_n$ of any segment $I\subseteq \partial U$ and considering the sum $\sum_{x,y\in P_n} |D_h^\gamma(z,x)-D_h^\gamma(z,y)|^{\gamma'd_\gamma/(2\gamma)}$ where the sum is over adjacent points $x,y$ in the partition, and the goal of the paper is to show that, when normalized appropriately, the above sum converges in probability to the measure $\nu^{\gamma'}_h(I)$. We now note that though we have been working for general domains $U$ till now, for the formal statement of the main result, we specialize to the case $U=\HH$. By the conformal covariance of the LQG metric (see \cite[Theorem 1.3]{GM19+}), it is possible to translate the statement for the domain $\HH$ to the case of general domains $U$. Also, though we do not state this here, we note that if instead of the free boundary GFF, we worked with the so-called quantum wedges \cite{She16} or disks \cite{DMS14,HRV18}, which are the objects actually expected to appear as scaling limits of natural planar map models, then the same results can be shown hold due to the local absolute continuity of these objects with respect to the free boundary GFF. We now state the main result of this paper, and as mentioned above, we work solely with the free boundary GFF on the domain $\HH$, the upper half plane.

\begin{theorem}
  \label{thm:2}
Let $h$ be a free boundary GFF on $\HH$ normalized to have average $0$ on the semicircle $\TT_1(0)=\{z\in \overline{\HH}: |z|=1\}$. Define the partition $\Pi_n=2^{-n}\ZZ$ and for any $u\in \Pi_n$, use $u^+$ to denote the point $u+2^{-n}$. %
Use $\psi_\gamma(p)$ to denote the constant $\psi_\gamma(p)=p-p(p-1)\gamma^2/4$. There exists a constant $\kappa=\kappa_{\gamma,\gamma'}$ such that with $\delta_u$ denoting the unit atomic measure at a point $u\in \RR$, for any fixed $z\in \HH$ and any fixed $K>0$, the random measure
\begin{equation}
  \label{eq:main}
2^{-n(1-\psi_\gamma(\gamma'/\gamma))}\sum_{u\in \Pi_n\cap [-K,K]}|D^{\gamma}_h(z,u)-D^{\gamma}_h(z,u^+)|^{\gamma'd_\gamma/(2\gamma)}\delta_{u}
\end{equation}
converges weakly in probability to the measure $\kappa\nu^{\gamma'}_h\lvert_{[-K,K]}$. %
\end{theorem}

In fact, as will be clear from the proof, the constant $\kappa=\EE|\fB^\gamma_h(0,1)|^{\gamma'd_\gamma/(2\gamma)}$, where $\fB^\gamma_h$ denotes the Busemann function corresponding to $D_h^\gamma$, a notion which we shall discuss in Section \ref{sec:busem}. We also note that in the statement of Theorem \ref{thm:2}, we use the free boundary GFF, and for a precise definition of this, we refer the reader to Section \ref{sec:gff}. We caution that the point $u^+=u+2^{-n}$ depends on $n$ but this dependency is suppressed to avoid clutter; this notation will be in play throughout the paper.  We also note that the constant $\psi_\gamma(p)$ introduced above appears frequently in the analysis of the so-called multifractal spectrum of $\nu_h^\gamma$ as we shall recall in Proposition \ref{prop:8} later. Finally, we remark that the normalization constant $2^{-n(1-\psi_\gamma(\gamma'/\gamma))}$ in the above theorem is chosen such that the normalized quantity has expectation which is uniformly bounded away from $0$ and $\infty$ (see Lemma \ref{lem:8}); note that the quantity $\psi_\gamma(\gamma'/\gamma)>0$ for all $\gamma,\gamma'\in (0,2)$, and this can be checked by the explicit expression for $\psi_\gamma$. %

As mentioned earlier, the most important case of Theorem \ref{thm:2} is when $\gamma'$ and $\gamma$ are equal. Indeed, in this case, the normalization term disappears as $\psi_\gamma(1)=1$, and thus the boundary measure $\nu^\gamma_h$ is equal to the $d_\gamma/2$-variation process of the distance profile, and we now record this as a corollary.

\begin{corollary}
  \label{cor:main}
  There exists a constant $\kappa_{\gamma,\gamma}$ such that for any fixed $z\in \HH$, and interval $I\subseteq \RR$, we have the following convergence in probability as $n\rightarrow \infty$,
  \begin{displaymath}
    \sum_{u\in \Pi_n\cap I} |D_h^\gamma(z,u)-D_h^\gamma(z,u^+)|^{d_\gamma/2}\rightarrow \kappa_{\gamma,\gamma} \nu_h^\gamma(I).
  \end{displaymath}
\end{corollary}
\begin{proof}
  Since the measure $\nu_h^\gamma$ a.s.\ has no atoms (see e.g.\ \cite[Theorem 6.36]{BP23}), the result follows immediately by noting that $\psi_\gamma(1)=1$, applying Theorem \ref{thm:2} with $\gamma'=\gamma$ and using the Portmanteau lemma.
\end{proof}
For another perspective on the above result, we note that the investigation of the relations between the different $\gamma$-LQG observables for a fixed $\gamma$ has received significant interest recently, and the above can be seen in this regard as well. To give examples of recent work concerning the relations between different $\gamma$-LQG observables, a natural question is to wonder whether any of $\widetilde{\nu}^\gamma_h,D^\gamma_h$ individually determine $h$ and thus the entire random geometry, and these questions have now been answered affirmatively, the former in \cite{BSS23}, and the latter very recently in \cite{GS22}. The latter work in fact shows that the bulk measure $\widetilde{\nu}^\gamma_h$ agrees with the $d_\gamma$-Minkowski content calculated using $D^\gamma_h$. Another recent work is \cite{LeGal22+}, which shows that for Brownian surfaces, the bulk measure is a.s.\ equal to the Hausdorff measure of the metric with respect to the background gauge function $\phi(r)=r^4\log \log (1/r)$. %

Since the scale invariance of the LQG metric is only true modulo a random multiplicative factor, and since the distance profile $x\mapsto D^\gamma_h(z,x)$ does not have any obvious Markov property, it seems difficult to directly argue by using moment estimates, explicit independence, or zero-one laws, that the variation process exists and is equal to the boundary measure $\nu^{\gamma'}_h$. Thus, in order to bypass these issues, we take a different route, which we now give a whirlwind tour of.

First, by moment arguments, we argue the uniform integrability (Lemma \ref{lem:15}) of the approximants in \eqref{eq:main}, and then we subsequently obtain the existence of a non-trivial limiting measure along subsequences (Proposition \ref{prop:main:1}) via a tightness argument. Having done so, we show that the subsequential limits have the correct mean (Lemma \ref{lem:14}), are measurable with respect to $h$ (Proposition \ref{prop:1}), and additionally transform according to `$\gamma'$-Weyl scaling' when the field $h$ is transformed (Proposition \ref{lem::3}), and this is enough to characterize $\nu^{\gamma'}_h$ and complete the proof due to a general result characterizing Gaussian multiplicative chaos \cite{Sha16}. 

The most challenging and novel aspect of the paper is the proof of the $\gamma'$-Weyl scaling of the subsequential limiting measures, and this is stated as Proposition \ref{lem::3}. To deduce $\gamma'$-Weyl scaling, we need to argue that for most points $u\in \Pi_n=2^{-n}\ZZ$, (which we later call as good), the variation increment $|D^\gamma_h(z,u)-D^\gamma_h(z,u^+)|^{\gamma'd_\gamma/(2\gamma)}$ is determined ``locally'' by the field close to $u$ (Lemmas \ref{lem:29}, \ref{lem:7}), and further, we need to argue that points $u$ which are not good contribute only negligibly (Proposition \ref{lem:9}) to the variation. To obtain these, we take advantage of the confluence structure of geodesics in $\gamma$-LQG to obtain that for most points $u$, the geodesics from $u$ to $z$ and from $u^+$ to $z$ coalesce quickly, and then we crucially use (Lemmas \ref{lem:29}, \ref{lem:7}) the equality $D^\gamma_h(z,u)-D^\gamma_h(z,u^+)=D^\gamma_h(w_{u,u^+},u)- D^\gamma_h(w_{u,u^+},u^+)$, where $w_{u,u^+}$ denotes the first point at which the geodesics from $u$ to $z$ and from $u^+$ to $z$ coalesce. An interesting aspect of the above approach is that it only uses confluence, a fairly typical property of planar random geometries, and for instance, the above does not utilize the mating-of-trees integrability \cite{DMS14} present for LQG, which we note was as important tool in the characterization \cite{GS22} of the bulk measure $\widetilde{\nu}^\gamma_h$ as the Minkowski content of $D_h^\gamma$.

\paragraph{\textbf{Notational comments}} 
For $x<y\in \RR$, we will use $[\![x,y]\!]$ to denote the set $\ZZ\cap [x,y]$. For a point $x\in \RR$, we use $\delta_x$ to denote the unit atomic measure at $x$. Often, we shall work with sets $U\subseteq \overline{\HH}$ that are open as a subset of $\overline{\HH}$ as opposed to being open as a subset of $\CC$, and to emphasize this, we will say that $U$ is $\overline{\HH}$-open. Similarly, we will use the phrase $\overline{\HH}$-closed to denote sets which are closed in $\overline{\HH}$. For any interval $I\subseteq \RR$, we will use $m(I)$ to denote its midpoint and $|I|$ to denote its length. Throughout the paper, we will work \textrm{LQG} with two parameters-- $\gamma$ and $\gamma'$. Usually, if we are working with the former, then we suppress the superscript $\gamma$ and just write $D_h,\nu_h,\widetilde{\nu}_h$. The above convention shall be used extensively, even for Busemann functions and geodesics which will be introduced in the next section. %
For a point $u\in \Pi_n=2^{-n} \ZZ$, we will use the notation $u^+$ to denote the point $u+2^{-n}$. We now introduce some sets which shall be used frequently. For $z\in \CC$, we define $\mathbb{D}_r(z)=\{z'\in \overline{\HH}: |z'-z|<r\}$ and $\mathbb{T}_r(z)=\{z'\in \overline{\HH}:|z'-z|=r\}$, and For $z\in \CC$ and $r<s$, we define $\A_{r,s}(z)=\{z'\in \overline{\HH}: |z'-z|\in (r,s)\}$. %
\paragraph{\textbf{Acknowledgements}}
The author thanks Morris Ang, Riddhipratim Basu and Ewain Gwynne for the discussions and Ewain Gwynne for comments on a previous draft of the paper. The author acknowledges the partial support of the NSF grant DMS-2153742. 
\section{Outline of the argument}
\label{sec:outline}
We now give a broad outline of the strategy used to prove Theorem \ref{thm:2} and point out the role played by each section of the paper. 
\subsubsection*{\textbf{Section \ref{sec:no-atoms}: Uniform integrability of the prelimiting measures}}
The goal of this section is to show the uniform integrability of the total mass of the sequence of random measures indexed by $n$ defined in \eqref{eq:main}. In fact, we will show the above-mentioned uniform integrability for the more general measures $\mu^{z,U,I}_{n,h}$ defined for any closed interval $I\subseteq \RR$, any fixed $\overline{\HH}$-open set $U$ with $U\cap \RR\supseteq I$, and any fixed point $z\in U\cap \HH$ by
\begin{displaymath}
\mu^{z,U,I}_{n,h}\coloneqq 2^{-n(1-\psi_\gamma(\gamma'/\gamma))} \sum_{u\in \Pi_n\cap I}|D_h(z,u;U)-D_h(z,u^+;U)|^{\gamma' d_\gamma/(2\gamma)}\delta_{u}. 
\end{displaymath}
To obtain this uniform integrability, we first define another sequence of measures which we call $\widetilde{\mu}_{n,h}$ by
\begin{displaymath}
   \wmu_{n,h}=2^{-n(1-\psi_\gamma(\gamma'/\gamma))} \sum_{u\in \Pi_n\cap [-1/2,1/2]}D_h(u,u^+;\mathbb{D}_{2^{-n}}( (u+u^+)/2)) ^{\gamma'd_\gamma/(2\gamma)} \delta_u,
 \end{displaymath}
 and we note (Lemma \ref{lem:35}) that these measures dominate the ones from \eqref{eq:main} as a consequence of the triangle inequality for $D_h$. Further, the measures $\mu_{n,h}$ are also ``local'' in the sense that each coefficient $D_h(u,u^+;\mathbb{D}_{2^{-n}}( (u+u^+)/2)) ^{\gamma'd_\gamma/(2\gamma)}$ is measurable with respect to $h\lvert_{\mathbb{D}_{2^{-n}}( (u+u^+)/2)}$. Now, by using the above locality along with a Markov property argument, for any fixed $p\in [1,4/\gamma'^2)$, we can control $\EE[\widetilde{\mu}_{n,h}([-1/2,1/2])^p]$ for large values of $n$ in terms of smaller values of $n$, thereby showing that the sequence must in fact be bounded in $n$. This argument is similar to the one used for estimating the moments of the Gaussian multiplicative chaos measure $\nu_h^{\gamma'}$ (see \cite[Section 3.9]{BP23}), and as a result, it is not surprising that the same threshold $4/\gamma'^2$ appears here as well. In fact, just as for $\nu_h^{\gamma'}$ (see Proposition \ref{prop:8}), we obtain (Lemma \ref{lem:15}) that for all $p\in [1,4/\gamma'^2)$, for some constant $C_p$ and every interval $I\subseteq [-1/2,1/2]$, we have $\EE[\wmu_{n,h}(I)^p]\leq C_p|I|^{\psi_{\gamma'}(p)}$, where $|I|$ refers to the length of $I$.

\subsubsection*{\textbf{Section \ref{sec:interv}: Good points and their properties}} In this section, for any interval $I\subseteq [-1/2,1/2]$, we use the confluence properties (Proposition \ref{lem:main:20}) of geodesics to define the set of good points $\good_{n,I}\subseteq \Pi_n\cap I$, which is a random set of points such that for all $u\in \good_{n,I}$ and any $\overline{\HH}$-open set $U$ with $I\subseteq U\cap \RR$, the quantity $D_h(z,u;U)-D_h(z,u^+;U)$ does not depend (see Figure \ref{fig:coal} and Lemma \ref{lem:7}) on $U$ or the point $z\in U$ as long as $z$ is not too close to $u$, by which we mean that $|z-(u+u^+)/2|>2^{-n(1-\alpha_2)}$ for a constant $\alpha_2\in (0,1)$ that will stay fixed throughout the paper. In fact, the above quantity is equal to the Busemann function $\fB_h(u,u^+)$ which intuitively equals ``$D_h(\infty,u)-D_h(\infty,u^+)$'' and will be rigourously defined in Section \ref{sec:busem}. Further, we show that these good points asymptotically carry all the mass of $\mu_{n,h}^{z,U,I}$ in the sense that the measure $\mu_{n,h}^{z,U,I}\lvert_{(\Pi_n\cap I)\setminus \good_{n,I}}$ a.s.\ converges to the zero measure as $n\rightarrow \infty$ (Proposition \ref{lem:9}). 
 
\subsubsection*{\textbf{Section \ref{sec:tightness}: Subsequential limits of the prelimiting measures and their mean}}
As an immediate consequence of the results of Section \ref{sec:no-atoms}, we obtain that the sequence $(h,\mu_{n,h}^{z,U,I})$ is tight in $n$ and must admit subsequential limits, which we denote as $(h,\mu_h^{z,U,I})$. An important goal of this section is to compute the mean $\EE[\mu_h^{z,U,I}(J)]$ for any interval $J$ and show that it is equal to $\kappa \EE[\nu_h^{\gamma'}(J)]$, as it must be for Theorem \ref{thm:2} to be true. To do so, we first use the uniform integrability from Section \ref{sec:no-atoms} to obtain that $\EE[\mu_h^{z,U,I}(J)]=\lim_{n\rightarrow \infty}\EE[\mu_{n,h}^{z,U,I}(J)]$, and then we use the results from Section \ref{sec:interv} described in the above paragraph to show that the above is equal to $\lim_{n\rightarrow \infty}\EE[\mu_{n,h}(J)]$, where $\mu_{n,h}$ should intuitively be thought of as $\mu_{n,h}^{\infty,\HH,[-1/2,1/2]}$ and is rigourously defined in \eqref{eq:230} using Busemann functions. The symmetries (Proposition \ref{prop:19}) satisfied by Busemann functions can now be utilized to obtain that $\lim_{n\rightarrow \infty}\EE[\mu_{n,h}(J)]=\kappa \EE[\nu_h^{\gamma'}(J)]$ with $\kappa=\EE|\fB_h(1,0)|^{\gamma'd_\gamma/(2\gamma)}$. We note that by a reasoning similar to the above, we also obtain that the subsequential limits $\mu_{h}^{z,U,I}$ are all compatible (Lemma \ref{lem:11}) in the sense that if $(h,\mu_{h}^{z,U,I},\mu_{h})$ is a subsequential weak limit of $(h,\mu_{n,h}^{z,U,I},\mu_{n,h})$, then we must have the a.s.\ equality
\begin{equation}
  \label{eq:3}
  \mu_{h}^{z,U,I}=\mu_{h}\lvert_I.
\end{equation}
\subsubsection*{\textbf{Section \ref{sec:weyl}: Any subsequential limit $\boldsymbol{\mu_h}$ satisfies $\boldsymbol{\gamma'}$-Weyl scaling}} In order to use Shamov's characterization of Gaussian multiplicative chaos (Proposition \ref{prop:4}) to conclude the a.s.\ equality $\mu_h=\nu_h^{\gamma'}\lvert_{[-1/2,1/2]}$, a crucial step is to establish that $\mu_h$ satisfies $\gamma'$-Weyl scaling. This amounts to showing that (Lemma \ref{lem:38}) for any differentiable function $\phi\colon \overline{\HH}\rightarrow \RR$ with the property that $h+\phi$ is mutually absolutely continuous to $h$, and for any subsequential weak limit $(h,\mu_h,\mu_{h+\phi})$ of $(h,\mu_{n,h},\mu_{n,h+\phi})$, we a.s.\ have $d\mu_{h+\phi}=e^{\gamma'\phi/2}d\mu_h$. The proof of this involves several ideas, the first of which is to define, a set $\good_{n,[-1/2,1/2],\phi}\subseteq \Pi_n\cap [-1/2,1/2]$ of points, which are good, in the sense of Section \ref{sec:interv}, for both the fields $h$ and $h+\phi$. By using that these good points must asymptotically carry (Lemma \ref{lem:37}) all the mass of the measures $\mu_{n,h}$ and $\mu_{n,h+\phi}$, we can reduce the task of establishing the a.s.\ equality $d\mu_{h+\phi}=e^{\gamma'\phi/2}d\mu_h$ to showing that for any bounded continuous function $f$ on $[-1/2,1/2]$, we have
\begin{equation}
  \label{eq:2}
  \sum_{u\in \good_{n,[-1/2,1/2],\phi}}f(u)\mu_{n,h+\phi}(\{u\})\sim \sum_{u\in \good_{n,[-1/2,1/2],\phi}}f(u)e^{\gamma'\phi(u)/2}\mu_{n,h}(\{u\}),
\end{equation}
in the sense that their ratio a.s.\ approaches $1$ as $n\rightarrow \infty$. Intuitively, the above amounts to establishing that for all $u\in \good_{n,[-1/2,1/2],\phi}$, the term $|\fB_{h+\phi}(u,u^+)|^{\gamma'd_\gamma/(2\gamma)}$ is approximately equal to $e^{\gamma'\phi(u)/2}|\fB_{h}(u,u^+)|^{\gamma'd_\gamma/(2\gamma)}$.

Now, by using the definition of the good set $\good_{n,[-1/2,1/2],\phi}$, we can show that with $\cS_u$ denoting the semi-circle of radius $2^{-n(1-\alpha_2)}$ around $(u+u^+)/2$, the terms $\fB_{h+\phi}(u,u^+)$ and $\fB_h(u,u^+)$ are equal to $D_{h+\phi}(\cS_u,u)-D_{h+\phi}(\cS_u,u^+)$ and $D_{h}(\cS_u,u)-D_{h}(\cS_u,u^+)$ respectively, where we note that we are using distances measured from the semi-circle $\cS_u$ (see \eqref{eq:237}). The utility of the above is that all the geodesics $\Gamma^h_{\cS_u,u},\Gamma^h_{\cS_u,u^+}, \Gamma^{h+\phi}_{\cS_u,u},\Gamma^{h+\phi}_{\cS_u,u^+}$ must lie inside the set $\overline{\DD}_{2^{-n(1-\alpha_2)}}((u+u^+)/2)$, and by the differentiability of $\phi$, we know that for some constant $C$ and for all $v\in \overline{\DD}_{2^{-n(1-\alpha_2)}}((u+u^+)/2)$, we have $|\phi(v)-\phi(u)|\leq C2^{-n(1-\alpha_2)}$. As a result of this and the Weyl scaling satisfied by the LQG metric (see Proposition \ref{prop:5}), we obtain that the ratio
\begin{displaymath}
  \left(D_{h+\phi}(\cS_u,u)-D_{h+\phi}(\cS_u,u^+)\right)/\left(e^{\gamma \phi(u)/d_\gamma}(D_{h}(\cS_u,u)-D_{h}(\cS_u,u^+))\right) 
\end{displaymath}
lies in the interval $[1-C2^{-n(1-\alpha_2)},1+C2^{-n(1-\alpha_2)}]$. By choosing the parameter $\alpha_2$ appropriately (see \eqref{eq:206}) such that the associated error terms are small, and with some careful bookkeeping, the above can be used to obtain the needed relation \eqref{eq:2}.

\subsubsection*{\textbf{Section \ref{sec:meas}: Measurability of $\boldsymbol{\mu_h}$ with respect to $\boldsymbol{h}$}} In order to apply Shamov's characterization of Gaussian multiplicative chaos (Proposition \ref{prop:4}), it remains to verify that for any subsequential limit $(h,\mu_h)$, the random measure $\mu_h$ is a.s.\ determined by $h$. To obtain this, we use an Efron-Stein argument. The first step is to argue that for any disjoint intervals $I_i\subseteq [-1/2,1/2]$, the measures $\mu_h\lvert_{I_i}$ are mutually independent conditional on the field $h$. It is not difficult to see that we might as well show the above for just two disjoint intervals $I_1,I_2$ and in this case, we first choose disjoint $\overline{\HH}$-open sets $U_1,U_2$ such that $I_1\subseteq U_1\cap \RR$ and $I_2\subseteq U_2\cap \RR$, and fix points $z_1\in U_1\cap \HH$ and $z_2\in U_2\cap \HH$. Now, by the Markov property (Proposition \ref{prop:18}), with $\fh$ denoting the harmonic extension off $U_1$, the fields $h\lvert_{U_1}$ and $(h-\fh)\lvert_{U_1^c}$ are independent, and as a result of the locality of the LQG metric (Proposition \ref{prop:5}), the pairs $(h\lvert_{U_1},\mu^{z_1,U_1,I_1}_{n,h})$ and $((h-\fh)\lvert_{U_1^c},\mu^{z_2,U_2,I_2}_{n,h-\fh})$ are independent for each $n$. Since weak convergence preserves independence, this implies that $(h\lvert_{U_1},\mu_{h}^{z_1,U_1,I_1})$ and $((h-\fh)\lvert_{U_1^c},\mu_{h-\fh}^{z_2,U_2,I_2})$ must be independent, and by \eqref{eq:3}, this is the same as $(h\lvert_{U_1},\mu_h\lvert_{I_1})$ and $((h-\fh)\lvert_{U_1^c},\mu_{h-\fh}\lvert_{I_2})$ being independent. By a slightly stronger version (Proposition \ref{lem::3}) of the $\gamma'$-Weyl scaling discussed in the previous paragraph (Proposition \ref{lem::3}), we can write $d\mu_h\lvert_{I_2}=e^{\gamma'\fh/2}d\mu_{h-\fh}\lvert_{I_2}$, and since $\fh$ is measurable with respect to $\sigma(h)=\sigma(h\lvert_{U_1},(h-\fh)\lvert_{U_1^c})$, this implies that $\mu_h\lvert_{I_1}$ and $\mu_h\lvert_{I_2}$ are conditionally independent given $h$.

With the above at hand, for any $a<b\in [-1/2,1/2]$ and small $\varepsilon>0$, we write $\mu_h(a,b)=\sum_{i=1}^{(b-a)/\varepsilon}\mu_h(a_i,a_{i+1})$ with $a_i=a+i\varepsilon$, where we use that $\mu_h$ a.s.\ has no atoms, and this is a fact that is shown in Proposition \ref{prop:3} in Section \ref{sec:tightness}. By using the Efron-Stein inequality and the above, it is not difficult to show that the conditional variance $\Var(\mu_h(a,b)\lvert h)$ must a.s.\ be equal to $0$ and this is enough to establish that $\mu_h$ is a.s.\ determined by $h$. We note that the above argument is similar to the one used in (\cite[Section 2.6]{DFGPS20}, \cite[Section 5]{GM19}) to prove the measurability of ``weak LQG metrics'' with respect to the GFF.

\subsubsection*{\textbf{Section \ref{sec:endshamov}: Completion of the proof via Shamov's theorem}} In this short and final section, we combine the above ingredients obtained in the previous sections to justify the use of Shamov's characterization of Gaussian multiplicative chaos and thereby finish the proof. We first use scaling arguments to argue that it is sufficient to prove Theorem \ref{thm:2} with $K=1/2$, and then we use Shamov's theorem to show that $\mu_h=\kappa \nu_h^{\gamma'}\lvert_{[-1/2,1/2]}$ almost surely. The proof of Theorem \ref{thm:2} is then completed by combining the above with \eqref{eq:3}.

\section{Preliminaries}
\label{sec:prelim}

\subsection{The Gaussian free field}
\label{sec:gff}

We begin by discussing the Gaussian free field, which we usually abbreviate as GFF. Since the GFF is by now a standard object in the literature, we do not provide all the details, and we refer the reader to the text \cite{BP23} for a detailed introduction to Gaussian free fields.

In this paper, we will mostly work with the free boundary (or Neumann) GFF on the upper half plane $\HH$ normalized to have average zero on the semicircle $\mathbb{T}_1(0)$, and we now build up towards its definition. Using $|v|_+$ to denote $\max(|v|,1)$ for $v,w\in \overline{\HH}$, we define the kernel $G(v,w)$ for $v,w\in \overline{\HH}$ by
\begin{equation}
  \label{eq:161}
  G(v,w)=\log \frac{|v|^2_+|w|^2_+}{|v-w||v-\overline{w}|},
\end{equation}
which can be recognized as the Green's function corresponding to the Laplacian on $\overline{\HH}$ normalized with Neumann boundary conditions and such that $\int(\int G(v,w)\phi(v)dv)d\rho_{1,0}(w)=0$ for all bump functions $\phi$, where $\rho_{1,0}$ denotes the uniform probability measure on the semicircle $\TT_1(0)$. Let $\cM$ denote the set of signed measures $\rho$ on $\overline{\HH}$ with the property that
\begin{equation}
  \label{eq:166}
  \int G(v,w)d\rho(v)d\rho(w)<\infty.
\end{equation}
 The free boundary GFF $h$ is now defined to be Gaussian process $\{(h,\rho)\}_{\rho\in \cM}$ indexed by elements of $\cM$ with the covariance structure given by
\begin{equation}
  \label{eq:168}
  \EE[(h,\rho)(h,\rho')]=\int G(v,w)d\rho(v)d\rho'(w).
\end{equation}
for any $\rho,\rho'\in \cM$ and it is not difficult to see that we a.s.\ have $(h,\rho_{1,0})=0$, and the specific form of the kernel \eqref{eq:161} was chosen to ensure this property. Throughout the paper, for $z\in \overline{\HH}$, we will use $h_r(z)$ to denote the ``circle'' average $(h,\rho_{r,z})$, where $\rho_{r,z}$ denotes the probability measure on the set $\TT_r(z)$ which is uniform with respect to the arc-length measure on it. We note that for $x\in \RR\subseteq \overline{\HH}$, the set $\TT_r(x)$ is a semicircle of radius $r$ around $x$.

We will use $\cD(\overline{\HH})$ to denote the set of real valued functions on $\overline{\HH}$ which have bounded support and are smooth (at the boundary as well) and additionally have Neumann boundary conditions along $\RR$, in the sense that, along $\RR$, the vector $\nabla \phi$ is orthogonal to the normal vector to $\RR$. Now, for any function $\phi \in \cD(\overline{\HH})$, it can be seen that the measure $\phi(z)d^2z$ lies in $\cM$ and thus for each fixed $\phi$, we can define $(h,\phi)=(h,\phi(z)d^2z)$. In fact (see \cite[Definition 2.10]{HB18}), one can a.s.\ define $(h,\phi)$ simultaneously for all $\phi\in \cD(\overline{\HH})$ in a continuous manner, and thus $h$ can be interpreted as a random generalized function.

Further, as we sketch in the following lemma, by a slight adaptation of the Kolmogorov-Chenstov argument used in \cite[Lemma 3.1]{HM22} and \cite[Lemma C.1]{HMP10}, all the circle averages $h_r(z)$ for $z\in \overline{\HH}$ can be defined simultaneously such that the map $(r,z)\mapsto h_r(z)$ is a.s.\ continuous.

\begin{lemma}
  \label{lem:cty} There exists a version of the process $(z,r)\mapsto h_r(z)$ that is a.s.\ continuous on the set $\{(z,r)\in \overline{\HH}\times (0,\infty)\}$.
\end{lemma}
\begin{proof}[Proof sketch]
  By the Kolmogorov-Chenstov criterion for Gaussian processes (see e.g.\ \cite[Lemma 3.19]{WP21}), it suffices to show that for any fixed $r_0>0$, there exists an $\varepsilon>0$ and a positive constant $C$, such that for all all $z,z'\in \overline{\HH}$ and all $r,r'\geq r_0$, we have $\EE[(h_r(z)-h_{r'}(z'))^2]\leq C(|z-z'|^\varepsilon + |r-r'|^\varepsilon)$, and we show this for $\varepsilon=1$.

  Since the above expression only concerns the difference between two circle averages, defining the Green's function $\widetilde{G}(v,w)= -\log |v-w| -\log |v-\overline{w}|$, we have the equality
  \begin{equation}
    \label{eq:4}
    \EE[(h_r(z)-h_{r'}(z'))^2]=\int \widetilde{G}(v,w) d(\rho_{r,z}-\rho_{r',z'})(v)d(\rho_{r,z}-\rho_{r',z'})(w).
  \end{equation}
. As a further simplification, if we denote the complex conjugation map on $\overline{\HH}$ by $\pi$ and define $\rho^\dagger_{r,z}=\rho_{r,z}+\pi^*\rho_{r,z}$, where $\pi^*\rho_{r,z}$ denotes the pushforward measure, then \eqref{eq:4} yields that
  \begin{equation}
    \label{eq:8}
    \EE[(h_r(z)-h_{r'}(z'))^2]=\int ( -\log | v-w| ) d(\rho_{r,z}-\rho_{r',z'})(v)d(\rho^\dagger_{r,z}-\rho^\dagger_{r',z'})(w).
  \end{equation}
  It can be checked that the function $F$ defined by $F(r,z,v)= \int-\log |v-w|d\rho_{r,z}^{\dagger}(w)$ is Lipschitz on the set $\{(r,z,v)\in [r_0,\infty)\times \overline{\HH}\times \CC\}$. By using this, we obtain that for some constant $C$ and all $r,r'\geq r_0$,
  \begin{align}
    \label{eq:9}
    \EE[(h_r(z)-h_{r'}(z'))^2]&=\int ( F(r,z,v)-F(r',z',v) )d(\rho_{r,z}-\rho_{r',z'})(w)\nonumber\\
    &\leq C(|r-r'|+|z-z'|)\int  d(\rho_{r,z}+\rho_{r',z'})(w)= 2C(|r-r'|+|z-z'|),
  \end{align}
  and this completes the proof.
\end{proof}
Regarding the marginals of the circle average process, by analysing the covariance function \eqref{eq:161}, it is not difficult to obtain that that for any fixed $z\in \HH$, the process $t\mapsto h_{e^{-t}\mathrm{Im}(z)}(z)-h_{\mathrm{Im}(z)}(z)$ for $t\geq 0$ is a Brownian motion of diffusivity $1$. In contrast, the behaviour around boundary points is different, and for any fixed $x\in \RR$, the process $t\mapsto h_{e^{-t}}(x)-h_1(x)$ for $t\in \RR$ is a Brownian motion with diffusivity $2$ (see \cite[Theorem 6.35]{BP23}). We now state an easy lemma about circle averages that will be useful later.
\begin{lemma}
  \label{lem:30}
For any $p>0$, the mapping $x\mapsto \EE[e^{ph_1(x)}]$ defines a continuous function on $\RR$.
\end{lemma}
\begin{proof}
  With $\rho_{1,x}$ denoting the uniform probability measure on the semi-circle $\TT_1(x)$, we know that that any fixed $x\in \RR$, $h_1(x)$ is a Gaussian with mean zero and variance
  \begin{equation}
    \label{eq:243}
    V(x)=\int G(v,w) d\rho_{1,x}(v)d\rho_{1,x}(w).
  \end{equation}
  Due to the expression \eqref{eq:161}, it can be seen that $V(x)$ is continuous function in $x$. Also, by Gaussianity, we have $\EE[e^{ph_1(x)}]=e^{p^2V(x)/2}$ and this combined with the continuity of $x\mapsto V(x)$ completes the proof.
\end{proof}
In general, Gaussian free fields are invariant under conformal transformations, and in particular, they have translational and scaling symmetries. However, in the case of the free boundary GFF $h$ (as opposed to the Dirichlet GFF), some care has to be taken due to the arbitrariness of the normalization $h_1(0)=0$. We now state the precise translational and scale invariance properties enjoyed by the field $h$.
\begin{lemma}[{see e.g.\ \cite[Corollary 6.5]{BP23}, \cite[Example 3.6]{GHS19}}]
  \label{lem:33}
  For any $x\in \RR$, the field $h(x+\cdot)-h_1(x)\stackrel{d}{=}h(\cdot)$ and further, $h_1(x)$ is independent of the field $h(x+\cdot)-h_1(x)$. Similarly for any $r>0$, we have $h(r\cdot)-h_r(0)\stackrel{d}{=}h(\cdot)$ and $h_r(0)$ is independent of the field $h(r\cdot)-h_r(0)$.
\end{lemma}

Now, for any fixed $\overline{\HH}$-open set $U$, we define the sigma algebra $\sigma(h\lvert_U)$. Since, as mentioned earlier, $h$ can be interpreted as a linear functional on $\cD(\overline{\HH})$, we now define
\begin{equation}
  \label{eq:216}
  \sigma(h\lvert_U)=\sigma(\{ (h,\phi): \phi \in \cD(\overline{\HH}), \mathrm{supp} (\phi) \subseteq U\}).
\end{equation}
Similarly, for an $\overline{\HH}$-closed set $K$, we can define $\sigma(h\lvert_K)=\bigcap_{\varepsilon>0}\sigma(h\lvert_{\DD_\varepsilon(K)})$, where $\DD_\varepsilon(K)$ refers to the $\varepsilon$-neighbourhood of $K$ inside $\overline{\HH}$. Often, if a random variable is measurable with respect to $\sigma(h\lvert_U)$ (or $\sigma(h\lvert_K)$), we will simply say that the variable is determined by $h\lvert_U$ (or $h\lvert_K$).
We now state a standard result on the absolute continuity properties of the GFF. %
  \begin{proposition}%
    \label{prop:7}
    Let $\rho\in \cD(\overline{\HH})$ and define the function $\phi$ on $\overline{\HH}$ by $\phi(v)=\int G(v,w)\rho(w) dw$. Then the field $h+\phi$ is mutually absolutely continuous with respect to $h$, with the Radon Nikodym derivative of the former with respect to the latter being $\exp( (h, \rho) -(\rho,\phi)/2)$, where $(\rho,\phi)$ denotes the usual $L^2$ inner product.
  \end{proposition}
  For the whole plane GFF, the analogous statement to the above appears as \cite[Proposition 2.9]{MS13}, and the above can by obtained via the same proof with some superficial modifications. For a general discussion on the absolute continuity properties of the GFF, we refer the reader to \cite[Section 3.3.3]{WP21}.

 Another useful property of the GFF that we shall need is the Markov property, and we now give a statement for the Neumann GFF. In order to do so, we need to introduce the Dirichlet-Neumann GFF. For any point $x\in \RR$ and half-disk $\DD_r(x)$, it is possible to define a GFF with Dirichlet boundary conditions on $\TT_r(x)$ and Neumann boundary conditions on $(x-r,x+r)$ as a continuous linear functional on the set $\{\phi\in \cD(\overline{\HH}):\mathrm{supp}(\phi)\subseteq  \DD_r(x)\}$, and we refer the reader to \cite[Section 6.4.2]{BP23} for a discussion of this. We are now ready to state the Markov property.
\begin{proposition}
  \label{prop:18}
For any $x\in \RR$ and $r>0$ such that $\mathbb{D}_r(x)\subseteq \DD_{1}(0)$, we have the independent decomposition
  \begin{displaymath}
    h\lvert_{\DD_r(x)}=h^{\mathrm{DN}}+\fh,
  \end{displaymath}
  where $h^{\mathrm{DN}}$ is a Dirichlet-Neumann GFF in $\DD_r(x)$ and $\fh$ is a random harmonic function in $\DD_r(x)$, measurable with respect to $h\lvert_{\overline{\HH}\setminus \DD_r(x)}$, and smooth up to and including the boundary $(x-r,x+r)$, and having Neumann boundary conditions along $(x-r,x+r)$. %
\end{proposition}
Since $h$ and $h^{\mathrm{DN}}$ are Gaussian processes, it is not difficult (see \cite[Lemma 6.3, Lemma 6.4]{MS16+}) to obtain that the function $\fh$, usually called the harmonic part/extension, is a centered Gaussian process as well. In fact, due to the continuity of $\fh$, for any compact set $K\subseteq \DD_r(x)$, $\fh\lvert_K$ is a centered Gaussian process which is a.s.\ bounded, and later in the paper, we shall use the Borell-TIS inequality \cite[Theorem 2.1.1]{AT10} to get subgaussian tail estimates for $\sup_{z\in K}\fh(z)$ and $\inf_{z\in K}\fh(z)$. Also, we note that a version of Proposition \ref{prop:18} can be stated for general disks $\DD_r(x)$, but we just assume $\DD_r(x)\subseteq \DD_1(0)$ since the statement is more complicated if we have $\DD_r(x)\cap \TT_1(0)\neq \emptyset$, where we recall that $\TT_1(0)$ has the special property that the $h_1(0)$, the average of $h$ on $\TT_1(0)$, is almost surely equal to $0$.

At one point in the paper (Lemma \ref{lem:4}), we will need to argue that an event known to occur with positive probability at a given scale for $h$ in fact occurs with high probability at a large number of exponentially separated scales. The reason for this is the scale invariance and log correlated structure of the GFF, and the above is usually referred to as iterating events in annuli, and this approach has been used in a variety of works and is by now standard. However, most of the statements are in the setting of the whole plane GFF, and for this reason, we now state a slightly modified version adapted to the half plane setting. We remark that the proof is the same with only superficial differences.
  \begin{proposition}[{\cite[Lemma 3.1]{GM19}}]
    \label{prop:10}
    Fix $0<s_1<s_2<1$. Let $\{r_k\}_{k\in \NN}$ be a decreasing sequence of positive real numbers such that $r_{k+1}/r_k\leq s_1$ for each $k\in \NN$. Now, for $z\in \overline{\HH}$, consider events $E^z_{r_k}\in \sigma( (h-h_{r_k}(z))\lvert_{\A_{s_1r_k,s_2r_k}(z)})$. Using $N(K)$ to denote the number of $k\in [\![1,K]\!]$ for which $E^z_{r_k}$ occurs, we have
    \begin{itemize}
    \item For each $a>0$ and each $b\in (0,1)$, there exist constants $p\in (0,1)$ and $c>0$ depending only on $a,b,s_1,s_2$ (and not on $z$) such that if $\PP(E^z_{r_k})\geq p$ for all $k\in \NN$, then $\PP(N(K)<bK)\leq ce^{-aK}$ for all $K\in \NN$.
    \item For each $p\in (0,1)$, there exist constants $a,c>0,b\in (0,1)$ depending only on $p,s_1,s_2$ (and not on $z$) such that if $\PP(E^z_{r_k})\geq p$, then $\PP(N(K)<bK)\leq ce^{-aK}$ for all $K\in \NN$.
    \end{itemize}
  \end{proposition}

\subsection{The LQG boundary length measure $\boldsymbol{\nu_h}$}
\label{sec:bdry}

We now come to the definition of the $\gamma$-LQG boundary length measure. The measure $\nu_h$ is a random measure on $\RR$ which is defined as the a.s.\ \cite{DS09,SW17} weak limit
\begin{equation}
  \label{eq:169}
  \nu_h\coloneqq \lim_{\varepsilon\rightarrow 0}\varepsilon^{\gamma^2/4}e^{\gamma h_\varepsilon(x)/2} dx.
\end{equation}
This is a specific example of a Gaussian multiplicative chaos measure, which can in fact be defined for any dimension; we refer the reader to the survey \cite{RV13} for a detailed discussion. As mentioned in the introduction, one can also define the $\gamma$-LQG bulk measure $\widetilde{\nu}_h$ analogously to the boundary measure \eqref{eq:169}, but the bulk measure will not be used in this paper and thus we do not discuss it in detail.

 We will require some moment estimates for the mass of $\nu_h$, and to state these, we now recall the function $\psi_\gamma$ defined in Theorem \ref{thm:2}. The following result (see \cite[Theorem 3.25]{BP23}) captures the multifractal behaviour of $\nu_h$ and will be useful for us. %
  \begin{proposition}
    \label{prop:8}
  For any fixed $p\in (0,4/\gamma^2)$ and any interval $I\subseteq \RR$, $\int_I \varepsilon^{\gamma^2/4}e^{\gamma h_\varepsilon(x)/2}dx$ forms a uniformly integrable sequence in $\varepsilon$. Further, there exists a constant $C$ such that for all intervals $I\subseteq [-1,1]$, we have
    \begin{equation}
      \label{eq:172}
      \EE[\nu_h(I)^p]\leq C|I|^{\psi_\gamma(p)}.
    \end{equation}
  \end{proposition}
  We will also require the following simple expression for the mean of $\nu_h$.
  \begin{lemma}
    \label{lem:31}
    For any interval $I\subseteq \RR$, we have the equality
    \begin{displaymath}
      \EE[\nu_h(I)]=\int_I \EE[e^{\gamma h_1(x)/2}]dx.
    \end{displaymath}
  \end{lemma}
  \begin{proof}
    By using the fact that for any fixed $x\in \RR$, the process $t\mapsto h_{e^{-t}}(x)-h_1(x)$ is a Brownian motion with diffusivity $2$, it can be obtained by using Fubini's theorem that for each $\varepsilon>0$, the expected mass that the measure $\varepsilon^{\gamma^2/4}e^{\gamma h_\varepsilon(x)/2} dx$ gives to the interval $I$ is exactly $\int_I \EE[e^{\gamma h_1(x)/2}]dx$, and the uniform integrability from Proposition \ref{prop:8} now completes the proof.
  \end{proof}

\subsection{The $\boldsymbol{\gamma}$-LQG metric $\boldsymbol{D_h}$}
\label{sec:lqg-metric}
Constructed in the series of works \cite{DDDF20,DFGPS20,GM19,GM19+,GM20,GM21}, the $\gamma$-LQG metric $D_h$ is a canonical metric on $\HH$ whose ``volume form'' is given \cite{GS22} by the $\gamma$-LQG bulk measure. Formally, the $\gamma$-LQG metric $h\mapsto D_h$ is a measurable map from the space Schwartz distributions on $\HH$ to the space of metrics on $\HH$ that are continuous with respect to the Euclidean metric, and as we shall see shortly, $D_h$ can be defined in terms of certain axioms.

In order to proceed, we first introduce some notation. Given a metric $d$ on a domain $V\subseteq \CC$, we define the length of a continuous curve $\eta\colon [a,b]\rightarrow V$, thereafter referred to as a path, by
\begin{equation}
  \label{eq:171}
  \ell(\eta;d)=\sup_{a=t_0<\dots<t_n=b}\sum_{i=1}^n d(\eta(t_{i-1}),\eta(t_{i})),
\end{equation}
where the supremum is taken over all $n\in \NN$ and all partitions $\{t_0,\dots,t_n\}$.
    Further, for any open set $U\subseteq V$, and points $u,v\in U$, we define the induced metric $d(u,v;U)$ by
    \begin{equation}
      \label{eq:170}
      d(u,v;U)=\inf_{\eta\subseteq U}\ell(\eta;d),
    \end{equation}
    where the infimum is taken over all paths $\eta$ in $U$ from $u$ to $v$. For a continuous function $g\colon V\rightarrow \RR$ and $u,v\in V$, we define
    \begin{equation}
      \label{eq:173}
      (e^g\cdot d)(u,v)=\inf_\eta \int_0^{\ell(\eta;d)}e^{g(\eta(t))}dt,
    \end{equation}
    where the infimum is over all paths $\eta$ in $\HH$ from $u$ to $v$ parametrized at unit $d$ speed, by which we mean that for each $0<s<t<\ell(\eta;d)$, we have $\ell(\eta\lvert_{[s,t]};d)=t-s$. 
    We are now ready to state the axiomatic definition of the $\gamma$-LQG metric, and we state this for a general GFF a plus continuous function $\mathtt{h}$, meaning that $\mathtt{h}=h+\phi$ for a possibly random continuous function $\phi\colon \overline{\HH}\rightarrow \RR$.
\begin{proposition}[{\cite{DDDF20,DFGPS20}}]
  \label{prop:5}
  For a GFF plus a continuous function $\mathtt{h}$, there exists a random metric $D_\mathtt{h}$ on $\HH$ which is measurable with respect to $\mathtt{h}$ and satisfies the following properties almost surely.
  \begin{enumerate}
  \item \textbf{Length space}: Almost surely, for all $u,v\in \HH$, $D_\mathtt{h}(u,v)=\inf_{\eta}\ell(\eta;D_\mathtt{h})$, where the infimum is across paths $\eta$ in $\HH$ connecting $u,v$.
  \item \textbf{Locality}: For any deterministic open set $U\subseteq \HH$, the induced metric $D_\mathtt{h}(\cdot;U)$ is determined by $\mathtt{h}\lvert_U$. 
  \item \textbf{Weyl scaling}: With $\xi=\gamma/d_\gamma$, where $d_\gamma$ denotes the fractal dimension of $\gamma$-LQG as defined in \cite{DG20}, we almost surely have the equality $e^{\xi f}\cdot D_\mathtt{h}=D_{\mathtt{h}+f}$ for every continuous function $f\colon \overline{\HH}\rightarrow \RR$.
  \item \textbf{Coordinate change formula}: With $Q$ defined by $Q = \gamma/2+2/\gamma$ and the field $\mathtt{h}'$ defined by $\mathtt{h}'(\cdot)=\mathtt{h}(r\cdot+x)+Q\log r$, for each fixed deterministic $r>0$ and $x\in \RR$, we almost surely have for all $u,v\in \HH$,
    \begin{equation}
      \label{eq:175}
      D_\mathtt{h}(ru+x,ru+v)=D_{\mathtt{h}'}(u,v).
    \end{equation}
  \end{enumerate}
\end{proposition}

While there is no available statement in the literature arguing that the above axioms determine the $\gamma$-LQG metric uniquely in the above setting of the free boundary GFF, the corresponding statement for whole plane GFF was established in \cite{GM21}. As discussed in \cite{GM19+}, for any open set $U\subseteq \CC$, there is a unique way to defined the $\gamma$-LQG metric for Gaussian free fields on $U$ so as to be compatible with the corresponding definition for the whole plane case. Taking $U=\HH$, this yields a unique choice of the $\gamma$-LQG metric $D_h$, and this is what we shall work with throughout the paper. Also, we note that we will use the notation $Q= \gamma/2+2/\gamma$ and $\xi=\gamma/d_\gamma$. Note that, as is the tradition in the literature, we suppress the dependency of $\gamma$ in the notations $Q$ and $\xi$ but do explicitly retain it for $d_\gamma$. In case we need to use the above for some $\gamma'\neq \gamma$, we will write $Q_{\gamma'}$ and $\xi_{\gamma'}$.

Though we only considered the distances $D_h(u,v)$ for $u,v\in \HH$ in the above discussions, the following result from \cite{HM22} shows that $u,v$ can also be allowed to be taken to be on the boundary. 
\begin{proposition}[{\cite[Proposition 1.7]{HM22}}]
  \label{prop:6}
  The metric $D_h$ on $\HH$ almost surely extends by continuity to a metric on $\overline{\HH}$ that induces the Euclidean topology on $\overline{\HH}$. %
  Further, on the boundary, $D_h$ is locally $2/d_\gamma-$ H\"older with respect to $\nu_h$ in the sense that for every $M>0$ and $\delta>0$, there exists a random positive constant $K$ with a polynomially decaying upper tail such that we have
  \begin{equation}
    \label{eq:229}
    D_h(x,y)\leq K\nu_h([x,y])^{2/d_\gamma -\delta}
  \end{equation}
  for all $x\leq y\in [-M,M]$. Also, $D_h$ is a.s.\ Euclidean bi-H\"older continuous in the sense that there exist deterministic constants $\chi_1,\chi_2>0$ such that, almost surely, for each compact set $K\subseteq \overline{\HH}$, there exists a random $C>0$ such that
  \begin{displaymath}
    C^{-1}|z-w|^{\chi_1}\leq D_h(z,w)\leq C|z-w|^{\chi_2}
  \end{displaymath}
  for all $z,w\in K$.
\end{proposition}
We note that \eqref{eq:229} above already hints at $d_\gamma/2$ being the ``correct'' exponent (Corollary \ref{cor:main}) to consider in order to obtain a non-trivial natural variation process for the distance profile $D_h(z,x)$ for a fixed point $z\in \overline{\HH}$. Indeed, \eqref{eq:229} along with the triangle inequality $|D_h(z,x)-D_h(z,y)|\leq D_h(x,y)$ implies that the distance profile is locally $2/d_{\gamma}-\delta$ H\"older \emph{with respect to $\nu_h$} for any $\delta>0$ and this should be compared to Brownian motion, which is locally $1/2-\delta$ H\"older for any $\delta>0$ and turns out to have finite $2$-variation (quadratic variation).

We will work with the extension from Proposition \ref{prop:6} throughout the paper and will simply use $D_h$ to denote it. %
Since $D_h$ now extends to the boundary, we can define $d(u,v;U)$ for any $\overline{\HH}$-open set $U$ in a manner similar to \eqref{eq:170}. It can also be checked that the following more general analogue of Proposition \ref{prop:5} (2) holds.
\begin{lemma}
  \label{lem:28}
  For any deterministic $\overline{\HH}$-open set $U$, the induced metric $D_h(\cdot;U)$ is determined by $h\lvert_U$.
\end{lemma}
We now state an estimate from \cite{HM22} controlling $D_h$ distances using circle averages.
\begin{proposition}[{\cite[Proposition 3.5]{HM22}}]
  \label{prop:16}
  Fix $a,b\in \RR$ with $|a-b|\leq 2$. For some positive random variable $K$ having superpolynomially decaying upper tails (uniformly in the choice of $a,b$), we have
  \begin{equation}
    \label{eq:213}
    D_h(a,b; \DD_{ |a-b|}( (a+b)/2)) \leq K \int_{\log (|b-a|^{-1})}^\infty [ e^{\xi(h_{e^{-t}}(a)-Q t)} + e^{\xi (h_{e^{-t}}(b)-Q t)}]dt.
  \end{equation}
\end{proposition}
We note that in \cite{HM22}, the above result is stated with $ D_h(a,b; \DD_{ |a-b|}( (a+b)/2))$ replaced by $D_h(a,b)$. However, the path from $a$ to $b$ constructed in the proof provided therein in fact lies within $\DD_{|a-b|}( (a+b)/2)$ (see \cite[Figure 2]{HM22}), thereby yielding the above stronger result. We now state a moment estimate which can be obtained as a consequence of the above.
\begin{lemma}
  \label{prop:9}
  For any $a,b\in \RR$ with $|a-b|\leq 2$ and $p\in (0,Q d_\gamma/\gamma )$, we have
  \begin{equation}
    \label{eq:214}
    \EE D_h(a,b; \DD_{|a-b|}((a+b)/2))^{p}<\infty.
  \end{equation}
\end{lemma}
\begin{proof}
  With Proposition \ref{prop:16} at hand, the proof is almost identical to the proof of \cite[Proposition 3.13]{DFGPS20}. Indeed, the proof is the same almost word for word with just the standard Brownian motion $B_t$ therein replaced by $\sqrt{2}B_t$. The reason for this is that, as discussed earlier, while the process $t\mapsto h_{e^{-t}}(z)$ behaves as a standard Brownian motion at interior points $z\in \HH$, its behavior changes to that of a Brownian motion with diffusivity $2$ at boundary points $z\in \RR$. 
\end{proof}

\subsection{Geodesics and strong confluence}
\label{sec:geod}

Having discussed $D_h$ as a metric on $\overline{\HH}$, a natural next step is to look at geodesics $\Gamma_{u,v}$ between points $u,v\in \overline{\HH}$ which are defined to be paths $\Gamma_{u,v}\subseteq \overline{\HH}$ from $u$ to $v$ for which $\ell(\Gamma_{u,v};D_h)=D_h(u,v)$. By a compactness argument (see \cite[Corollary 2.5.20]{BBI01}), it can be shown that geodesics $\Gamma_{u,v}$ exist for any $u,v\in \overline{\HH}$. We note that in the case of the whole plane GFF, it was established in \cite{MQ18} that there is a.s.\ a unique geodesic between any two fixed points, and a similar but simpler argument for the same also appears as \cite[Lemma 4.2]{DDG21}. Using Proposition \ref{prop:7} and arguing along the same lines as in \cite{DDG21}, it can be shown that there is a.s.\ a unique geodesic $\Gamma_{u,v}$ between two fixed points $u,v\in \overline{\HH}$. %

As mentioned in the introduction, a common but startling property of geodesics observed in many models of random geometry is that of confluence. Geodesic confluence for the LQG metric corresponding to the whole plane GFF in $\gamma$-LQG was established in the work \cite{GM20}, and then later, a stronger version of confluence was established in \cite{GPS20}. While an absolute continuity argument can be used to obtain geodesic confluence for the metric $D_h$ on $\overline{\HH}$ for geodesics and points that stay away from the boundary $\overline{\HH}$, a separate argument is required for geodesics emanating from points on the boundary of $\overline{\HH}$. However, apart from minor differences, these arguments are the same as the corresponding ones in \cite{GM20,GPS20}, and in the appendices, we outline the steps of the argument without complete proofs, but we do attempt to emphasize that the aspects which are somewhat different. We now state the crucial strong confluence result, whose proof outline is deferred to the appendices as described above.

\begin{figure}
  \centering
  \includegraphics[width=0.6\textwidth]{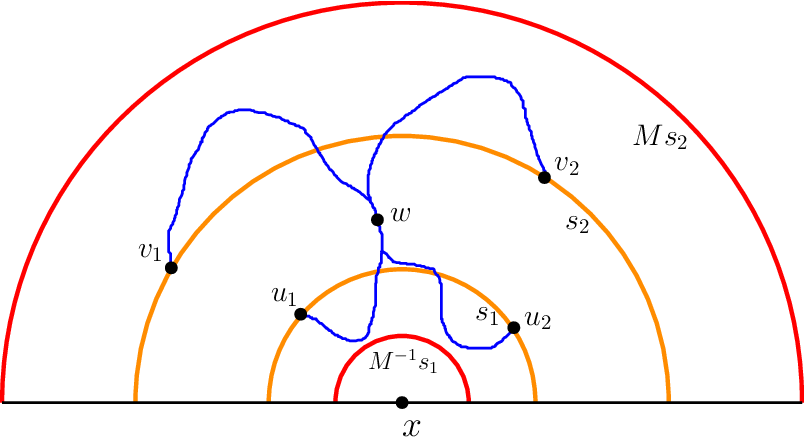}
  \caption{The event $\coal_{s_1,s_2}(x)$: On this event, any two geodesics $\Gamma_{u_1,v_1}, \Gamma_{u_2,v_2}$ for $u_1,u_2\in \TT_{s_1}(x)$ and $v_1,v_2\in \TT_{s_2}(x)$ meet at the coalescence point $w\in \A_{s_1,s_2}(x)$, and further, these geodesics lie inside the set $\A_{M^{-1}s_1,Ms_2}(x)$.}
  \label{fig:coal}
\end{figure}
  \begin{proposition}
    \label{lem:main:20}
     There exists a constant $M>1$, such that for any fixed $0<s_1<s_2$ and any fixed $x\in \RR$, there is an event $\coal_{s_1,s_2}(x)$ and a random point $w\in \A_x(s_1,s_2)$, both measurable with respect to $h\lvert_{\A_{M^{-1}s_1,Ms_2}(x)}$ viewed modulo an additive constant, satisfying the following properties.
    \begin{enumerate}
    \item The probabilities $\PP(\coal_{s_1,s_2}(x))$ %
    are uniformly positive for all $x$ as long as $s_2/s_1\geq M$.
    \item We have $\Gamma_{u,v}\subseteq \A_{M^{-1}s_1,Ms_2}(x)$ for all points $u\in \TT_{s_1}(x),v\in \TT_{s_2}(x)$, and further, all these geodesics pass through the common point $w$.
    \end{enumerate}
  \end{proposition}

\subsection{Geodesics to infinity and Busemann functions}
\label{sec:busem}
As a consequence of Proposition \ref{lem:main:20} along with a standard iteration argument (see Proposition \ref{prop:10}) for the GFF, it can be shown that $\coal_{M^i,M^{i+1}}(0)$ a.s.\ occurs for infinitely many $i\in \NN$. As a consequence, it is not difficult to obtain the following statement on the existence of infinite geodesics.
\begin{lemma}
  \label{prop:17}
 Almost surely, simultaneously from all points $x\in \RR$, there exists an infinite path $\Gamma_x\subseteq \overline{\HH}$ starting at $x$ with the property that every finite segment of $\Gamma_x$ is a $D_h$-geodesic. Further, for any fixed $x\in \RR$, there is almost surely a unique such infinite geodesic $\Gamma_x$.
\end{lemma}
We note that an exact counterpart of the above result for the $\gamma$-LQG corresponding to a whole plane GFF is present in \cite[Proposition 4.4]{GPS20}, and the proof is analogous.

Just as finite geodesics exhibit coalescence, it is true that almost surely, any two infinite geodesics $\Gamma_x,\Gamma_y$  for $x,y\in \RR$ coalesce with each other, in the sense that they agree outside a compact set in $\overline{\HH}$, and this can be seen as a straightforward consequence of the fact that $\coal_{M^i,M^{i+1}}(0)$ occur a.s.\ for infinitely many $i\in \NN$. The above confluence allows us to define Busemann functions which intuitively represent distances to the point at infinity. Busemann functions have their roots in geometry \cite{Bus12} and were introduced to first passage percolation in \cite{New95, Hof05}. For $x,y\in \RR$ and for any choice of the geodesics $\Gamma_x,\Gamma_y$ and $w_{x,y}$ being any point satisfying $w_{x,y}\in \Gamma_x\cap\Gamma_y$, we define the Busemann function $\fB_h(x,y)$ by
\begin{equation}
  \label{eq:177}
  \fB_h(x,y)=D_h(w_{x,y},x)-D_h(w_{x,y},y),
\end{equation}
and it is not difficult to obtain as a consequence of the above-mentioned confluence that $\fB_h(x,y)$ is uniquely defined irrespective of the exact choice of $\Gamma_x,\Gamma_y$ and $w_{x,y}$, and is also a.s.\ continuous in both the variables $x,y$.

Busemann functions are crucial to this paper since they enjoy symmetries which distance profiles do not (see Remark \ref{rem:busem}), and indeed, as mentioned earlier, the constant $\kappa$ in Theorem \ref{thm:2} takes simple form $\kappa=\EE|\fB_h(0,1)|^{\gamma'd_\gamma/(2\gamma)}$. We now state the symmetries enjoyed by Busemann functions as a proposition.
\begin{proposition}
  \label{prop:19}
  For each $x_0\in \RR$, we have the equality $\fB_h(x,y)\stackrel{d}{=}e^{-\xi h_1(x_0)}\fB_h(x+x_0,y+x_0)$ as continuous processes, and further, the process $e^{-\xi h_1(x_0)}\fB_h(x+x_0,y+x_0)$ is independent of $h_1(x_0)$.
  Similarly, for each $r>0$, we have the equality $\fB_h(x,y)\stackrel{d}{=}r^{-\xi Q}e^{-\xi h_r(0)}\fB_h(rx,ry)$ as continuous processes, and further, the process $r^{-\xi Q}e^{-\xi h_r(0)}\fB_h(rx,ry)$ is independent of $h_r(0)$.
\end{proposition}
\begin{proof}
 If we define the field $\mathtt{h}$ by $\mathtt{h}(\cdot)=h(x_0+\cdot)-h_{1}(x_0)$, then almost surely, simultaneously for all $x\in \RR$, if $\Gamma_x$ is an infinite geodesic for $D_h$, then $\Gamma_{x}-x_0$ is an infinite geodesic for $D_{\mathtt{h}}$ and vice-versa. Now, by Weyl scaling and the coordinate change formula (Proposition \ref{prop:5}), for each finite segment $\eta$ of $\Gamma_x-x_0$, we have $\ell(\eta;D_\mathtt{h})=e^{-\xi h_1(x_0)}\ell(\eta+x_0;D_h)$, and this completes the proof of the first distributional equality since we know (Lemma \ref{lem:33}) that $\mathtt{h}\stackrel{d}{=}h$. This also establishes the required independence since we know that $\mathtt{h}$ and $h_1(x_0)$ are independent.

  We now similarly prove the second statement. If we define the field $\mathtt{h}$ by $\mathtt{h}(\cdot)=h(r\cdot)-h_r(0)$, then almost surely, simultaneously for all $x\in \RR$, if $\Gamma_x$ is an infinite geodesic for $D_h$, then $r^{-1}\Gamma_x$ is an infinite geodesic for $D_\mathtt{h}$ and vice-versa. Now, by Weyl scaling and the coordinate change formula (Proposition \ref{prop:5}), for each finite segment $\eta$ of $r^{-1} \Gamma_x$, we have $\ell(\eta;D_\mathtt{h})=r^{-\xi Q} e^{-\xi h_r(0)}\ell(r\eta;D_h)$, and this establishes the needed distributional equality. Again, this also establishes the desired independence since $\mathtt{h}$ and $h_r(0)$ are independent by Lemma \ref{lem:33}.
\end{proof}
The above symmetries can intuitively be considered to be a consequence of the fact that on doing a simple scaling and translation of the upper half plane, the point $\infty$ stays fixed, and we refer the reader to Remark \ref{rem:busem} for a discussion of this point. %
\subsection{Shamov's characterization of Gaussian Multiplicative chaos}
\label{sec:shamov}

Similar to the axiomatic characterization of the $\gamma$-LQG metric in the case of the whole plane GFF, the Gaussian multiplicative chaos measures in general can be characterized as the unique measures which satisfy Weyl scaling. This was shown by Shamov in \cite{Sha16} and we now give a statement adapted to our setting
  \begin{proposition}
    \label{prop:4}
    Let $\lambda_h$ be a random measure on $[-1/2,1/2]$ satisfying the following properties.
    \begin{enumerate}
      \item \textbf{Measurability}: The random measure $\lambda_h$ is measurable with respect to $\sigma(h)$.
    \item \textbf{Equality in mean}: We have $\EE[\lambda_h(S)]=\EE[\nu^{\gamma'}_h(S)]=\int_S \EE[e^{\gamma' h_1(x)/2}]dx$ for every Borel set $S\subseteq [-1/2,1/2]$.
    \item \textbf{$\gamma'$-Weyl scaling}: For each fixed function $\phi\colon \overline{\HH}\rightarrow \RR$ defined as $\phi(v)=\int_\HH G(v,w) \rho(dw)$ for some function $\rho\in \cD(\overline{\HH})$, we a.s.\ have
      \begin{displaymath}
        d\lambda_{h+\phi}=e^{\gamma'\phi/2}d\lambda_h.
      \end{displaymath}
    \end{enumerate}
    Then we have $\lambda_h=\nu^{\gamma'}_h\lvert_{[-1/2,1/2]}$ almost surely.
  \end{proposition}

  The proof of the above is exactly the same as the one presented in Theorem 3.12 and Remark 3.14 in \cite{BP23}, with the only difference being that Proposition \ref{prop:7} is substituted for the corresponding absolute continuity result used therein. Also, we note that in Proposition \ref{prop:4}, there is nothing special about the interval $[-1/2,1/2]$, and we simply state it in this manner since $[-1/2,1/2]$ is the interval with which it will be used later.


\section{The prelimiting measures and their uniform integrability}
\label{sec:no-atoms}
We begin by introducing the prelimiting measures $\mu_{n,h}^{z,U,I}$ which will be used frequently throughout the paper. For any closed interval $I\subseteq \RR$, any fixed $\overline{\HH}$-open set $U$ with $U\cap \RR\supseteq I$, and any fixed point $z\in U\cap \HH$, we define
\begin{equation}
  \label{eq:117}
\mu^{z,U,I}_{n,h}\coloneqq 2^{-n(1-\psi_\gamma(\gamma'/\gamma))} \sum_{u\in \Pi_n\cap I}|D_h(z,u;U)-D_h(z,u^+;U)|^{\gamma' d_\gamma/(2\gamma)}\delta_{u},  
\end{equation}
and as a minor point, we note that in case do not have $u^+\in U\cap \RR$ for the right-most $u\in \Pi_n\cap I$, we skip the corresponding term in the above sum. It is easy to see that the measure considered in Theorem \ref{thm:2} is just $\mu_{n,h}^{z,\overline{\HH}, [-K,K]}$ for some $K>0$. We note that while Theorem \ref{thm:2} works with general intervals $[-K,K]$, by a scale invariance argument, we will later reduce reduce to just the case $K=1/2$, and due to this, from this point onwards, we usually only work with measures on the interval $[-1/2,1/2]$.

In the entire paper, the  measures that we will use most frequently are ``$\mu_{n,h}^{\infty,\overline{\HH},[-1/2,1/2]}$'' which we denote as $\mu_{n,h}$ and formally define using Busemann functions by
\begin{equation}
  \label{eq:230}
 \mu_{n,h}\coloneqq 2^{-n(1-\psi_\gamma(\gamma'/\gamma))}\sum_{u\in \Pi_n\cap [-1/2,1/2]}|\fB_h(u,u^+)|^{\gamma'd_\gamma/(2\gamma)}\delta_{u}.
\end{equation}

Later, we shall take subsequential weak limits (as $n\rightarrow \infty$) of the measures defined above, and for this, it will be crucial that the random variables $\mu_{n,h}([-1/2,1/2])^p$ be uniformly integrable and proving this is the aim of the present section. To achieve the above, we will need to get estimates for the moments $\EE\mu_{n,h}([x,y])^p$ uniformly in $n$, and for this, it will be necessary to control the correlation between $\mu_{n,h}([x_1,y_1])$ and $\mu_{n,h}([x_2,y_2])$ for well-separated intervals $[x_1,y_1],[x_2,y_2]$. This will be done via the Markov property of the GFF, and for this, it will be useful to have a ``local'' proxy $\wmu_{n,h}$ of the measures $\mu_{n,h}$, which we define by
\begin{equation}
  \label{eq:154}
  \wmu_{n,h}=2^{-n(1-\psi_\gamma(\gamma'/\gamma))} \sum_{u\in \Pi_n\cap [-1/2,1/2]}D_h(u,u^+;\mathbb{D}_{2^{-n}}( (u+u^+)/2)) ^{\gamma'd_\gamma/(2\gamma)} \delta_u.
\end{equation}
The main utility of the measures $\wmu_{n,h}$ is due to the following basic result.
\begin{lemma}
  \label{lem:35}
  Almost surely, for any Borel set $A\subseteq [-1/2,1/2]$, we have the inequality $\mu_{n,h}(A)\leq  \widetilde \mu_{n,h}(A)$. Similarly, for any closed interval $I\subseteq [-1/2,1/2]$, any $\overline{\HH}$-open set $U$ with $I\subseteq U\cap \RR$ and any point $z\in U\cap\HH$, we a.s.\ have $\mu_{n,h}(A)\leq  \widetilde \mu_{n,h}(A)$ for all Borel sets $A\subseteq [-1/2,1/2]$.
\end{lemma}
\begin{proof}
Since $\mathbb{D}_{2^{-n}}( (u+u^+)/2))\subseteq \overline{\HH}$, we know that for each $u\in \Pi_n\cap I$, we have
\begin{equation}
  \label{eq:239}
  |\fB_h(u,u^+)|\leq D_h(u,u^+)\leq D_h(u,u^+;\mathbb{D}_{2^{-n}}( (u+u^+)/2)),
\end{equation}
where the first inequality is a consequence of the triangle inequality for $D_h$. Similarly, for all $n$ large enough, we have
\begin{equation}
  \label{eq:259}
  |D_h(z,u;U)-D_h(z,u^+;U)|\leq D_h(u,u^+;U)\leq D_h(u,u^+;\mathbb{D}_{2^{-n}}( (u+u^+)/2)),
\end{equation}
where we have used that $\mathbb{D}_{2^{-n}}( (u+u^+)/2)\subseteq U$ for all $n$ large enough and all $u\in \Pi_n\cap I$. This completes the proof.
\end{proof}
As a result of the above, $\wmu_{n,h}$ can be thought of as a measure which dominates $\mu_{n,h}$, but whose individual summands are additionally locally dependent since $D_h(u,u^+;\mathbb{D}_{2^{-n}}( (u+u^+)/2)) ^{\gamma'd_\gamma/(2\gamma)}$ is determined (Lemma \ref{lem:28}) by $h\lvert_{\mathbb{D}_{2^{-n}}( (u+u^+)/2)}$. %
We now state the main estimate of this section.
\begin{lemma}
  \label{lem:15}
  For any fixed $p\in [1,4/\gamma'^2)$, the quantity $\EE[\wmu_{n,h}([-1/2,1/2])^p]$ is finite and uniformly bounded in $n$. In fact, there exists a constant $C_p$ such that for each $p\in [1,4/\gamma'^2)$, for any interval $I\subset [-1/2,1/2]$ and for all $n$, we have
  \begin{equation}
    \label{eq:205}
    \EE[\wmu_{n,h}(I)^p]\leq C_p|I|^{\psi_{\gamma'}(p)}.
  \end{equation}
\end{lemma}

The aim of the remainder of the section is to prove Lemma \ref{lem:15}. We note that Lemma \ref{lem:15} looks very similar to the corresponding moment estimate (Proposition \ref{prop:8}) enjoyed by $\nu_h$. As a result, it is not surprising that the arguments that we use for establishing Lemma \ref{lem:15} will be a close adaptation of the corresponding arguments used for controlling the moments of Gaussian multiplicative chaos (see \cite[Section 3.9]{BP23}).

Our first task is to handle the simplest case in the above, which is, $p=1$. To do so, we first need the following estimate which is a straightforward consequence of Weyl scaling.
\begin{lemma}
  \label{lem:26}
  For any $x<y\in \RR$, we have
  \begin{equation}
    \label{eq:247}
    \EE[ D_h(x,y;\DD_{|y-x|}((x+y)/2))^{\gamma'd_\gamma/(2\gamma)}]=|y-x|^{\psi_\gamma(\gamma'/\gamma)}\EE[e^{\gamma' h_1(x)/2}]\EE [ D_h(0,1;\DD_1(1/2))^{\gamma' d_\gamma/(2\gamma)}],
  \end{equation}
where the right hand side is finite.
\end{lemma}
\begin{proof}
  By an application of Weyl scaling (Proposition \ref{prop:5}), for any $x<y\in \RR$, we can write
  \begin{align}
    \label{eq:215}
    &\EE[ D_h(x,y;\DD_{|y-x|}((x+y)/2))^{\gamma'd_\gamma/(2\gamma)}]=\EE[e^{\gamma' h_1(x)/2}]\EE[D_h(0,y-x;\DD_{|y-x|}((y-x)/2))^{\gamma' d_\gamma/(2\gamma)}]\nonumber\\
  &=|y-x|^{\gamma' Q/2}\EE[e^{\gamma' h_1(x)/2}] \EE[e^{\gamma' h_{y-x}(0)/2}]\EE [ D_h(0,1;\DD_1(1/2))^{\gamma' d_\gamma/(2\gamma)}].
  \end{align}
  The above equalities are routine but make heavy use of Lemma \ref{lem:33}. Indeed, the first line uses Weyl Scaling, the independence of $h_1(x)$ and $h-h_1(x)$, and the fact that $h(x+\cdot)-h_1(x)\stackrel{d}{=}h(\cdot)$. On the other hand, the second line uses the coordinate change formula (Proposition \ref{prop:5}), Weyl scaling, the fact that $h_{y-x}(0)$ is independent of $h-h_{y-x}(0)$, and $h((y-x)\cdot)-h_{y-x}(0)\stackrel{d}{=}h(\cdot)$.

  Now, since the process $t\mapsto h_{e^{-t}}(0)$ is a Brownian motion with diffusivity $2$, we know that $h_{y-x}(0)$ is exactly a Gaussian of variance $2\log |y-x|^{-1}$ and thus $\EE[e^{\gamma' h_{y-x}(0)/2}]=|y-x|^{-\gamma'^2/4}$. Combined with the above equation, this yields
\begin{equation}
  \label{eq:134.1}
   \EE[ D_h(x,y;\DD_{|x-y|}((x+y)/2))^{\gamma'd_\gamma/(2\gamma)}]=|y-x|^{\gamma'Q/2-\gamma'^2/4}\EE[e^{\gamma' h_1(x)/2}]\EE [ D_h(0,1;\DD_1(1/2))^{\gamma' d_\gamma/(2\gamma)}].
\end{equation}
One can now do a simple algebraic computation to check that $\psi_\gamma(\gamma'/\gamma)=\gamma'Q/2-\gamma'^2/4$, and this finishes the proof of \eqref{eq:247}.
To show that the right hand side of \eqref{eq:247} is finite, we just need to check that $\EE [ D_h(0,1;\DD_1(1/2))^{\gamma' d_\gamma/(2\gamma)}]<\infty$, and this follows by noting that $\gamma' d_\gamma/(2\gamma) < Q d_\gamma/\gamma$ (since $Q=\gamma/2+2/\gamma>2$ and $\gamma'<2$) and subsequently applying Lemma \ref{prop:9}.
\end{proof}
The Weyl-scaling argument used above to relate moments for distances between different points will be used at multiple points in the paper. As an immediate consequence of the above, we can now handle the $p=1$ case of Lemma \ref{lem:15}.
\begin{lemma}
  \label{lem:27}
  There exists a positive constant $C$ such that for any interval $I\subseteq [-1/2,1/2]$, we have %
  \begin{displaymath}
    \EE[\wmu_{n,h}(I)] \leq C |I|.
  \end{displaymath}
\end{lemma}
\begin{proof}
  Define the constant $C$ by $C=2(\sup_{x\in [-1/2,1/2]}\EE[e^{\gamma'h_1(x)/2}])$, and we note that the supremum here is finite due to Lemma \ref{lem:30}. Now, by using the fact that $\#(\Pi_n\cap I)\leq 2^{n+1}|I|$ for any interval $I\subseteq [-1/2,1/2]$, along with Lemma \ref{lem:26}, we obtain that
  \begin{align}
    \label{eq:253}
    \EE[\wmu_{n,h}(I)]&=2^{-n(1-\psi_\gamma(\gamma'/\gamma))} \sum_{u\in \Pi_n\cap I}\EE[D_h(u,u^+;\mathbb{D}_{2^{-n}}( (u+u^+)/2)) ^{\gamma'd_\gamma/(2\gamma)}]\nonumber\\
                      &=2^{-n} \sum_{u\in \Pi_n\cap I} \EE[e^{\gamma'h_1(u)/2}]\EE [ D_h(0,1;\DD_1(1/2))^{\gamma' d_\gamma/(2\gamma)}]\nonumber\\
                      &\leq 2^{-n}\#(\Pi_n\cap I) \EE [ D_h(0,1;\DD_1(1/2))^{\gamma' d_\gamma/(2\gamma)}]\sup_{x\in [-1/2,1/2]}\EE[e^{\gamma'h_1(x)/2}]\nonumber\\
    &=C|I|,
  \end{align}
  where we have used Lemma \ref{lem:30} to obtain that $\sup_{x\in [-1/2,1/2]}\EE[e^{\gamma'h_1(x)/2}]$ is finite.
\end{proof}
 The goal now is to prove Lemma \ref{lem:15} in its entirety, and this will be done by an induction argument. As mentioned earlier, the arguments used here will be an adaption of \cite[Section 3.9]{BP23}.

For $m\in \NN$, we define $\cS_m^1$ be the set of intervals $[j2^{-m},(j+1)2^{-m}]\subseteq [-1/2,1/2]$ with $j$ odd and let $\cS_m^2$ be the corresponding set with $j$ even. The reason for introducing the above notation is that we will shortly split the terms appearing in the moment $\EE[\wmu_{n,h}([-1/2,1/2])^p]$ depending on which interval in $\cS_m^1,\cS_m^2$ they correspond to. Indeed, this is clear in the special case when $p\in \NN$ since, for any $m\in \NN$, we can simply write
\begin{equation}
  \label{eq:202}
  \EE[\wmu_{n,h}([-1/2,1/2])^p]\leq \EE[(\sum_{I\in \cS_m^1\cup \cS_m^2} \wmu_{n,h}(I))^p]=\sum_{I_1,\dots,I_p\in \cS_m^1\cup \cS_m^2} \EE[\Pi_{i=1}^p\wmu_{n,h}(I_i)].
\end{equation}
Usually, we will work in the regime where $m$ is large and $n$ is much larger than $m$. Regarding \eqref{eq:202}, we will shortly obtain an analogous expression even in the case $p\notin \NN$, and the arising diagonal and cross terms shall be separately controlled. The following estimate will used to obtain good control on the diagonal terms.

\begin{lemma}
  \label{lem:16}
  For any $p\in \NN$, there exists a constant $C_p$ such that for all $m\in \NN$, $n\geq m$ and $I\in \cS_m^1\cup\cS_m^2$, we have
  \begin{displaymath}
    \EE[\wmu_{n,h}(I)^p]\leq C_p  |I|^{\psi_{\gamma'}(p)} \EE \wmu_{n-m,h}([-1/2,1/2])^p,
  \end{displaymath}
  with the understanding that either side above might be infinite.
\end{lemma}
\begin{proof}
As mentioned earlier, we use $m(I)$ to denote the midpoint of $I$. Now, by a Weyl scaling argument identical to the one used in Lemma \ref{lem:26}, we have
  \begin{align}
    \label{eq:149}
    \EE[\wmu_{n,h}(I)^p]&=\EE[e^{(p\gamma'/2) h_1(m(I))}] \EE[\wmu_{n,h}( [-|I|/2,|I|/2])^p]\nonumber\\
                                                &= \EE[e^{(p\gamma'/2) h_1(m(I))}] |I|^{\gamma' Q p/2}\EE[e^{(p\gamma'/2) h_{|I|}(0)}]|I|^{p(1-\psi_\gamma(\gamma'/\gamma))}\EE[\wmu_{n-m,h}([-1/2,1/2])^p]\nonumber\\
                                                 &\leq C_p |I|^{\gamma' Q p /2-\gamma'^2p^2/4+p( 1-\psi_\gamma(\gamma'/\gamma))}\EE[\wmu_{n-m,h}([-1/2,1/2])^p]\nonumber\\
                                                 &=C_p |I|^{-p\gamma'^2(p-1)/4+p}\EE[\wmu_{n-m,h}([-1/2,1/2])^p]\nonumber\\
    &=C_p|I|^{\psi_{\gamma'}(p)}\EE[\wmu_{n-m,h}([-1/2,1/2])^p],
  \end{align}
  where to obtain the third line we have used that $\EE[e^{(p\gamma'/2) h_1(m(I))}]$ is bounded above independently of the value of $m(I)\in [-1/2,1/2]$, and this can be seen as a consequence of Lemma \ref{lem:30}. This completes the proof. %
\end{proof}

As an immediate consequence of the above along with the inequalities $\#\cS_m^1,\#\cS_m^2\leq 2^m$, we obtain the following recursive control on the contribution of the diagonal terms.

\begin{lemma}
  \label{lem:18}
  Fix $p\in [1,4/\gamma'^2)$. There is a constant $C_p$ depending on $p$ such that for all $n$, we have the decay estimate
  \begin{displaymath}
    \EE [\sum_{I\in \cS_m^1}\wmu_{n,h}(I)^p]\leq C_p2^{-m(\psi_{\gamma'}(p)-1)}\EE \wmu_{n-m,h}([-1/2,1/2])^p,
  \end{displaymath}
  with the corresponding statement being true as well if $\cS_m^1$ is replaced by $\cS_m^2$.
\end{lemma}
By a simple algebraic computation, it can be checked that, in exponential term in the above lemma, the coefficient $\psi_{\gamma'}(p)-1$ is positive for $p\in (1,4/\gamma'^2)$-- this will be used later. We now define the integer $\beta$ by $\beta=\lceil p\rceil$. The following lemma will shortly be useful in doing an induction argument over $\beta$. %
\begin{figure}
  \centering
  \includegraphics[width=\textwidth]{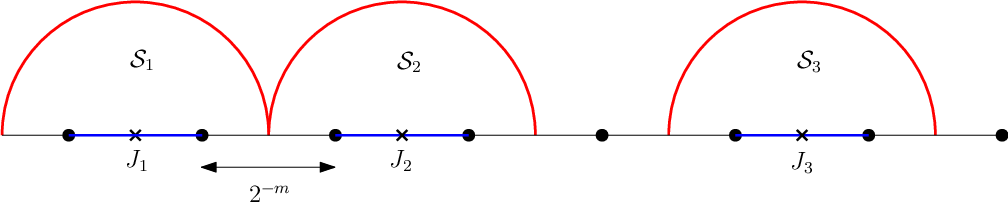}
  \caption{The setting of the proof of Lemma \ref{lem:dec} in the case $k'=3$: To obtain the needed decorrelation estimate between the $\widetilde{\mu}_{n,h}(J_i)^{\beta_i}$ for $i\in \{1,2,3\}$, we use the Markov property with the sets $\cS_i=\DD_{2^{{-m}}}(m(J_i))$.}
  \label{fig:markov}
\end{figure}
\begin{lemma}
  \label{lem:dec}
With $m\in \NN$ and $p>1$, there exists a constant $c_{m,p}$ such that for any $I_1,\dots ,I_\beta\in \cS_m^1$ (or $\cS_m^2$) which are not all equal, we have 
   \begin{equation}
     \label{eq:153}
     \EE[\wmu_{n,h}(I_1)\dots \wmu_{n,h}(I_\beta)]\leq c_{m,p}\max_{k\in [\![0,\beta-1]\!]}\EE [\wmu_{n,h} ([-1/2,1/2])^{k}]^\beta,
  \end{equation}
\end{lemma}
\begin{proof}
  In the proof, we will use the Markov property and for this, the local nature of \eqref{eq:154} will be vital. We begin by writing
  \begin{equation}
    \label{eq:151}
    \EE[\wmu_{n,h}(I_1)\dots \wmu_{n,h}(I_\beta)]=\EE[\Pi_{i=1}^{k'} \wmu_{n,h}(J_i)^{\beta_i}],
  \end{equation}
  where $\beta_1+\dots \beta_{k'}=\beta=\lceil p\rceil \geq 2$, where $\beta_i$ is the number of times the interval $J_i$ appears on the left hand side, and $k'$ denotes the number of distinct intervals appearing on the left hand side. %
  Now, we define the sets $S_i$ by $S_i=\DD_{2^{{-m}}}(m(J_i))$ for $i\in [\![1,k']\!]$ and note that the $S_i$ are pairwise disjoint and satisfy $S_i \subseteq \DD_1(0)$. For an illustration of the basic setting of the proof, we refer the reader to Figure \ref{fig:markov}. Intuitively, we wish to condition on the field $h$ on the semi-circles bounding the $S_i$ in order to get the desired decorrelation, and we use the Markov property (Proposition \ref{prop:18}) for this.

  Indeed, there exist random harmonic functions $\fh_i$ on $S_i$ measurable with respect to $h$ such that the fields
  \begin{equation}
    \label{eq:217}
    h\lvert_{S_i}-\fh_i
  \end{equation}
  are mutually independent Dirichlet-Neumann boundary GFFs on $S_i$ which are in addition independent of $\sigma(\{\fh_i\}_{i\in [\![1,k']\!]})$.

   Further, we can couple i.i.d.\ GFFs $h_i'$ with $h$ such that $h_i\stackrel{d}{=}h$ for all $i$, and if we apply the Markov property for the field $h_i'$ with the domain $S_i$, then for all $i$, we have %
  \begin{equation}
    \label{eq:219}
    h_i'\lvert_{S_i}-\fh_i'=h\lvert_{S_i}-\fh_i,
  \end{equation}
  where $\fh_i'$, the harmonic part of $h_i'$ on $S_i$, is a harmonic function defined on $S_i$ and measurable with respect to $h_i'$. In the above, the $\fh_i'$s can be arranged to be independent of $h$, and we note that the harmonic functions $\fh_i,\fh_i'$ above are all centered Gaussian processes, as discussed just after Proposition \ref{prop:18}.

  Since $\fh_i$ and $\fh_i'$ are centered Gaussian processes and are continuous, if we define $M_i$ to denote the quantity $\max_{z\in \DD_{0.75\times 2^{{-m}}}(m(J_i))}\fh_i(z)$ and $M_i'$ to denote $\min_{z\in \DD_{0.75\times 2^{{-m}}}(m(J_i))}\fh'_i(z)$, then by the Borell-TIS inequality, there exist constants $c_m',C_m'$ such that for all $i$ and all $t>0$, we have
  \begin{align}
    \label{eq:155}
    \PP(M_i\geq t)&\leq C_m'e^{-c_m't^2},\\
    \label{eq:155.1}
     \PP(M_i'\leq -t)&\leq C_m'e^{-c_m't^2}.
  \end{align}
 We are now ready to prove \eqref{eq:153}. Noting that for any $u\in J_i$, we have $\mathbb{D}_{2^{-n}}( (u+u^+)/2)\subseteq \DD_{0.75\times 2^{{-m}}}(m(J_i))$ for all $n$ large enough compared to $m$, we obtain by using the above along with Weyl scaling that for each $i$,
  \begin{equation}
    \label{eq:218}
\Pi_{i=1}^{k'}\wmu_{n,h}(J_i)^{\beta_i}\Pi_{i=1}^{k'} e^{ (\gamma'\beta_i /2) (M_i')}\leq \Pi_{i=1}^{k'} e^{ (\gamma' \beta_i /2) (M_i)}\Pi_{i=1}^{k'}\wmu_{n,h_i'}(J_i)^{\beta_i}.
  \end{equation}
  Now, we know that the $\fh_i'$s and thus the $M_i'$s are independent of $h$. Similarly, since $\sigma(\{\fh_i\}_{i\in [\![1,k']\!]})$ and $\sigma(\{\fh_i',h\lvert_{S_i}-\fh_i\}_{i\in [\![1,k']\!]})$ are independent, we know that the $M_i$s are independent of the $\sigma(\{h_i'\lvert_{S_i}\}_{i\in [\![1,k']\!]})$. Thus, we can take expectations on both sides of \eqref{eq:218} to obtain
  \begin{equation}
    \label{eq:220}
    \EE[\Pi_{i=1}^{k'}\wmu_{n,h}(J_i)^{\beta_i}]\EE[\Pi_{i=1}^{k'} e^{ (\gamma' \beta_i /2) (M_i')}]\leq \EE[\Pi_{i=1}^{k'} e^{ (\gamma' \beta_i /2) (M_i)}]\Pi_{i=1}^{k'}\EE[\wmu_{n,h_i'}(J_i)^{\beta_i}].
  \end{equation}
Now, note that $k'\leq \beta$ and also $1\leq \beta_i\leq \beta-1$ (since not all terms on the left side of \eqref{eq:153} are the same) for all $i$. By using this along with the subgaussian estimate \eqref{eq:155}, we know that $\EE[\Pi_{i=1}^{k'} e^{ (\gamma' \beta_i /2) M_i}]$ (resp.\ $\EE[\Pi_{i=1}^{k'} e^{ (\gamma' \beta_i /2) M_i'}]$) is bounded above (resp.\ below) by a constant depending only on $m,p$. Using this along with the fact that $h_i'\stackrel{d}{=}h$ for all $i$, it follows from \eqref{eq:220} that for some constant $c_{m,p}$, 
  \begin{equation}
    \label{eq:221}
    \EE[\Pi_{i=1}^{k'}\wmu_{n,h}(J_i)^{\beta_i}]\leq c_{m,p} \Pi_{i=1}^{k'}\EE[\wmu_{n,h}(J_i)^{\beta_i}].
  \end{equation}
  Again, by using the inequalities $k'\leq \beta$ and $\beta_i\leq \beta-1$, we obtain
  \begin{equation}
    \label{eq:148}
  \EE[\Pi_{i=1}^{k'} \wmu_{n,h}(J_i)^{\beta_i}]\leq  \max_{k\in [\![0,\beta-1]\!]}\EE [\wmu_{n,h} ([-1/2,1/2])^{k}]^\beta,
  \end{equation}
  and we note that we have additionally used that $J_i\subseteq [-1/2,1/2]$ for all $i$. In conjunction with \eqref{eq:221}, this completes the proof.
\end{proof}
By a simple application of Jensen's inequality, we can now upgrade the above lemma to the following.
\begin{lemma}
  \label{lem:19}
  For any $m\in \NN$ and $p>1$, there exists a constant $C_{m,p}$ such that for all $n$, we have
  \begin{equation}
    \label{eq:152}
    \EE \left[\sum_{I_1,\dots,I_\beta\in \cS_m^1\colon I_1\neq I_2} \wmu_{n,h}(I_1)^{p/\beta}\dots \wmu_{n,h}(I_\beta)^{p/\beta}\right]\leq C_{m,p} \max_{k\in [\![0,\beta-1]\!]}\EE [\wmu_{n,h} ([-1/2,1/2])^{k}]^p.   
  \end{equation}
  The corresponding statement holds for $\cS_m^2$ as well.
\end{lemma}
\begin{proof}
  By Jensen's inequality and the concavity of the function $x\mapsto x^{p/\beta}$, for each individual term, we have
  \begin{equation}
    \label{eq:157}
    \EE [\wmu_{n,h}(I_1)^{p/\beta}\dots \wmu_{n,h}(I_n)^{p/\beta}] \leq \EE[\wmu_{n,h}(I_1)\dots \wmu_{n,h}(I_n)]^{p/\beta}.
  \end{equation}
  Since the number of terms on the left hand side of \eqref{eq:152} depend only on $m,p$, we can simply apply Lemma \ref{lem:dec} and take a sum to obtain the needed result.
\end{proof}
We now finally complete the proof of Lemma \ref{lem:15} by an induction argument.
\begin{proof}[Proof of Lemma \ref{lem:15}]
  We first prove that $\EE[\wmu_{n,h}([-1/2,1/2])^p]$ is finite and uniformly bounded in $n$. To do so, we will induct on the value of $\beta$, where we recall that $\beta$ is simply defined as $\beta=\lceil p\rceil$. We note that base case $\beta=1$ was already covered in Lemma \ref{lem:27}.
  Now, we assume that for some $\beta_0\in \NN$, the needed estimate is true as long as $\beta\leq \beta_0$, and we now show the estimate for $\beta=\beta_0+1$. For some constants $c_p', C_p', C_{m,p}''$, we can now write
  \begin{align}
    \label{eq:158}
    \EE[(\sum_{I\in \cS_m^1}\wmu_{n,h}(I))^p]&= \EE\left[\left(\sum_{I_1,\dots,I_\beta\in \cS_m^1}\wmu_{n,h}(I_1)\dots\wmu_{n,h}(I_\beta)\right)^{p/\beta}\right]\nonumber\\
    &\leq \EE\sum_{I_1,\dots,I_\beta\in \cS_m^1}\wmu_{n,h}(I_1)^{p/\beta}\dots\wmu_{n,h}(I_\beta)^{p/\beta}\nonumber\\
                                        &\leq \EE [\sum_{I\in \cS_m^1}\wmu_{n,h}(I)^p]+c_p'\EE \left[\sum_{I_1,\dots,I_\beta\in \cS_m^1\colon I_1\neq I_2} \wmu_{n,h}(I_1)^{p/\beta}\dots \wmu_{n,h}(I_\beta)^{p/\beta}\right]\nonumber\\
    &\leq C'_{p}2^{-m(\psi_{\gamma'}(p)-1)}\EE \wmu_{n-m,h}([-1/2,1/2])^p + C''_{m,p},
  \end{align}
where to obtain the second line, we have used that since $p/\beta\leq 1$, the inequality
  \begin{equation}
    \label{eq:254}
    (\sum_{I_1,\dots,I_\beta\in \cS_m^1}\wmu_{n,h}(I_1)\dots\wmu_{n,h}(I_\beta))^{p/\beta}\leq \sum_{I_1,\dots,I_\beta\in \cS_m^1}\wmu_{n,h}(I_1)^{p/\beta}\dots\wmu_{n,h}(I_\beta)^{p/\beta}
  \end{equation}
holds deterministically. Regarding the last line of \eqref{eq:158}, to obtain the first term, we have used Lemma \ref{lem:18}, while to obtain the latter term, we have used Lemma \ref{lem:19} along with the fact that the induction hypothesis implies that the term $\max_{k\in [\![0,\beta-1]\!]}\EE [\wmu_{n,h} ([-1/2,1/2])^{k}]^p$ is finite and uniformly bounded in $n$. Now, since $p\in(1,4/\gamma'^2)$, we know that $\psi_{\gamma'}(p)-1>0$, and we now fix $m$ to be large enough such that $C'_p2^{-m(\psi_{\gamma'}(p)-1)}<2^{-p-2}$.

  We note that the discussion above is valid for $\cS_m^1$ replaced by $\cS_m^2$ as well. As a result, we obtain
  \begin{align}
    \label{eq:159}
   \EE[\wmu_{n,h}([-1/2,1/2])^p]\leq  \EE [(\sum_{I\in \cS_m^1\cup \cS_m^2} \wmu_{n,h}(I))^p]&\leq 2^{p}(\EE [(\sum_{I\in \cS_m^1} \wmu_{n,h}(I))^p]+\EE [(\sum_{I\in \cS_m^2} \wmu_{n,h}(I))^p])\nonumber\\
    &\leq \frac{1}{2} \EE \wmu_{n-m,h}([-1/2,1/2])^p + 2^{p+1}C_{m,p}''.
  \end{align}
  As a consequence, for every $N\in \NN$ with $N>m$, we obtain that
  \begin{equation}
    \label{eq:160}
    \max_{n \in [\![1,N]\!]}\EE[\wmu_{n,h}([-1/2,1/2])^p]\leq (1/2) \max_{n \in [\![1,N]\!]}\EE[\wmu_{n,h}([-1/2,1/2])^p]+ \max _{i\in [\![1,m]\!]} \EE[\wmu_{i,h}([-1/2,1/2])^p] +2^{p+1} C_{m,p}'',
  \end{equation}
  and this implies that $\max_{n \in [\![1,N]\!]}\EE[\wmu_{n,h}([-1/2,1/2])^p]\leq 2^{p+2}C_{m,p}''+2\max _{i\in [\![1,m]\!]} \EE[\wmu_{i,h}([-1/2,1/2])^p]$ for every $N>m$. Finally, we take the limit $N\rightarrow \infty$, and this shows that $\EE[\wmu_{n,h}([-1/2,1/2])^p]$ is finite and uniformly bounded in $n$.

 As the final part of the proof, we now show \eqref{eq:205}. First, for some value of $m$ depending on $I$ which we now fix, we can find intervals $I_1,I_2\in \cS_m^1\cup\cS_m^2$ such that $I\subseteq I_1\cup I_2$ and $|I_1|+|I_2|\leq 2|I|$. Thus, we can write
  \begin{equation}
    \label{eq:244}
    \widetilde{\mu}_{n,h}(I)^p\leq 2^{p}(\widetilde{\mu}_{n,h}(I_1)^p+\widetilde{\mu}_{n,h}(I_2)^p),
  \end{equation}
  and by now applying Lemma \ref{lem:16}, we obtain that for some constants $C,C_1$ depending on $p$ and all $n\geq m$,
  \begin{equation}
    \label{eq:245}
    \EE[\widetilde{\mu}_{n,h}(I)^p] \leq C2^p(|I_1|^{\psi_{\gamma'}(p)}+|I_2|^{\psi_{\gamma'}(p)})\EE \wmu_{n-m,h}([-1/2,1/2])^p\leq C_1|I|^{\psi_{\gamma'}(p)}\EE \wmu_{n-m,h}([-1/2,1/2])^p,
  \end{equation}
and we note that $\EE \wmu_{n-m,h}([-1/2,1/2])^p$ is uniformly bounded in $n,m$ as proved in the first part of this proof. This shows \eqref{eq:205}, and completes the proof.
\end{proof}

\section{Intervals without confluence contribute negligibly}
\label{sec:interv}
In the next section, we shall take weak limits of the measures
\begin{displaymath}
  \mu^{z,U,I}_{n,h}=2^{-n(1-\psi_\gamma(\gamma'/\gamma))} \sum_{u\in \Pi_n\cap I}|D_h(z,u;U)-D_h(z,u^+;U)|^{\gamma' d_\gamma/(2\gamma)}\delta_{u}
\end{displaymath}
introduced in \eqref{eq:117}. However, when we take these limits, there will be an a priori dependence of the weak limit on the triple $z,U,I$. To say that consistent weak limits can be chosen without such a dependence (Lemma \ref{lem:11}), we will need to argue that due to the confluence of geodesics, the prelimiting measures $\mu^{z,U,I}_{n,h}$ lose their dependence on the exact choice of $z$ and $U$ asymptotically as $n\rightarrow \infty$, and the aim of this section is to show this. Using the confluence events and the global constant $M$ defined in Proposition \ref{lem:main:20}, for any interval $I\subseteq \RR$, we define the event
\begin{equation}
  \label{eq:26}
  F_{I,i}= \coal_{M^i|I|,M^{i+1}|I|}(m(I)),
\end{equation}
and we recall that $m(I)$ and $|I|$ denote the midpoint and length of $I$ respectively. As an immediate consequence of Proposition \ref{lem:main:20}, we have the following lemma.
\begin{lemma}
  \label{lem:2}
  There exists a constant $p>0$ such that for all $i\in \NN$ and all intervals $I\subseteq \RR$, we have $\PP(F_{I,i})\geq p$.
\end{lemma}

\noindent Next, we iterate the $F_{I,i}$ events across multiple scales in order to boost the probability of their occurrence-- the resulting events will be denoted as $E_I$. To do so, we now introduce two constants $\alpha_1,\alpha_2$ which will remain fixed throughout the paper. \textbf{We now fix a choice of constants $\boldsymbol{\alpha_1,\alpha_2}$ such that $\boldsymbol{0<\alpha_1<\alpha_2<1}$ along with the condition}
\begin{equation}
  \label{eq:206}
(1-\alpha_2)\gamma'd_\gamma/(2\gamma)-\alpha_2\psi_\gamma(\gamma'/\gamma)>0,
\end{equation}
and this can be done as $\psi_\gamma(\gamma'/\gamma)>0$. The reason behind the above choice is to ensure that Lemma \ref{lem:22}, which appears much later, holds. We now define the event $E_{I}$ by
\begin{equation}
  \label{eq:67}
  E_{I}=\bigcup_{i=\alpha_1 \log_M |I|^{-1}+1}^{\alpha_2 \log_M |I|^{-1}-1}F_{I,i}.
\end{equation}
For clarity, we note that the union above is taken over all $i\in [\![\alpha_1 \log_M |I|^{-1}+1,\alpha_2 \log_M |I|^{-1}-1]\!]$. By the measurability statement in Proposition \ref{lem:main:20}, we obtain the following lemma.
\begin{lemma}
  \label{lem:3}
For any interval $I$, the event $E_{I}$ is measurable with respect to $h\lvert_{\A_{|I|^{1-\alpha_1},|I|^{1-\alpha_2}}(m(I))}$ viewed modulo an additive constant.
\end{lemma}
\noindent Now, by a straightforward iteration of scales argument using Proposition \ref{prop:10}, we obtain the following bound on $\PP(E^c_I)$.
\begin{lemma}
  \label{lem:4}
  There exists a constant $\chi>0$ for which we have $\PP(E^c_I)\leq |I|^{\chi}$ for all intervals $I\subseteq [-1/2,1/2]$. %
\end{lemma}

For any interval $I\subseteq [-1/2,1/2]$, we define the set of good points by
    \begin{equation}
      \label{eq:126}
      \good_{n,I}=
      \left\{
        u\in \Pi_n\cap I: E_{[u,u+]} \textrm{ occurs}
      \right\},
    \end{equation}
    and we use $\good_{n,I}^c$ to denote the set $(\Pi_n\cap I)\setminus \good_{n,I}$. %
    At this point, the utility of the above definition \eqref{eq:126} of good points is not clear and explaining this is the goal of the next two lemmas, the latter of which shows that on an appropriate coalescence event, the difference of distances from a reference point forms a quantity which depends neither on the reference point nor the background domain.
    \begin{lemma}
      \label{lem:32}
      Let $U$ be an $\overline{\HH}$-open set such that $U\cap \RR\neq \emptyset$, and let $x\in U\cap \RR$ and $0<s_1<s_2$ be such that $\DD_{Ms_2}(x)\subseteq U$, where $M$ is the constant defined in Proposition \ref{lem:main:20}. Then, on the event $\coal_{s_1,s_2}(x)$, and with the coalescence point $w$ from Proposition \ref{lem:main:20}, almost surely, for all points $z,u'$ such that $z\in U\setminus  \DD_{s_2}(x)$ and $u'\in \DD_{s_1}(x)$, we have
      \begin{equation}
        \label{eq:233}
        D_h(z,u';U)=D_h(z,w;U)+D_h(w,u';U),
      \end{equation}
      \begin{equation}
        \label{eq:238}
        D_h(w,u';U)=D_h(w,u').
      \end{equation}
    \end{lemma}
    \begin{proof}
          To show \eqref{eq:233}, it suffices to show that for any path $\eta\subseteq U$ from $z$ to $u'$, we can create another path $\widetilde \eta\subseteq U$ from $z$ to $u'$ which additionally passes through $w$ and which satisfies $\ell(\widetilde \eta;D_h)\leq \ell(\eta;D_h)$. Using the unit speed parametrization for $\eta$, we can write $\eta\colon [0,\ell(\eta;D_h)]\rightarrow U$. We define the time $t_0$ by
      \begin{equation}
        \label{eq:234}
        t_0=\sup\{t\in [0,\ell(\eta;D_h)]: \eta(t)\in \TT_{s_2}(x)\},
      \end{equation}
      and note that by (2) in Proposition \ref{lem:main:20}, for any geodesic $\Gamma_{\eta(t_0),u'}$, we must have $\Gamma_{\eta(t_0),u'}\subseteq \DD_{Ms_2}(x)$, and since $\DD_{Ms_2}(x)\subseteq U$, $\Gamma_{\eta(t_0),u'}$ must in fact also be a $D_h(\cdot;U)$ geodesic between $u'$ and $\eta(t_0)$. Now, with the above in mind, we can simply define $\widetilde \eta$ by concatenating $\eta\lvert_{[0,t_0]}$ with the path $\Gamma_{\eta(t_0),u'}$, and it is clear that we must have $\ell(\widetilde \eta;D_h)\leq \ell(\eta;D_h)$. Now, by (2) in Proposition \ref{lem:main:20}, we know that $\Gamma_{\eta(t_0),u'}$ passes through the coalescence point $w$, and as a result, we obtain that $\widetilde \eta$ passes through $w$ as well, and this completes the proof of \eqref{eq:233}. In particular, we have shown that $w\in \Gamma_{\eta(t_0),u'}\subseteq U$, and this immediately implies \eqref{eq:238} as well.
    \end{proof}
    \begin{figure}
  \centering
  \includegraphics[width=0.5\textwidth]{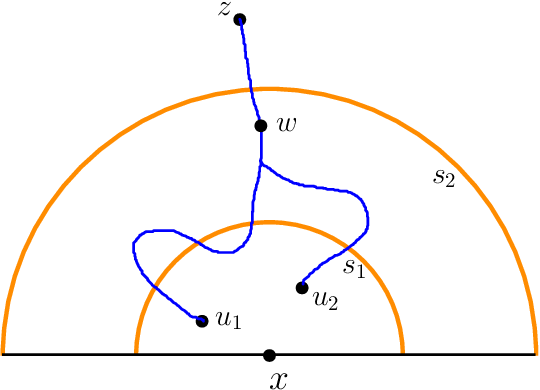}
  \caption{Lemma \ref{lem:29} in the simplest case $U=\overline{\HH}$: On the coalescence event $\coal_{s_1,s_2}(x)$, the geodesics $\Gamma_{u_1,z}$ and $\Gamma_{u_2,z}$ must both go via $w$ and must coalesce before doing so. Thus we have $D_h(z,u_1)-D_h(z,u_2)=D_h(w,u_1)-D_h(w,u_2)=\fB_h(u_1,u_2)$.}
  \label{fig:coal}
\end{figure}
    \begin{lemma}
      \label{lem:29}
      Let $U$ be an $\overline{\HH}$-open set such that $U\cap \RR\neq \emptyset$, and let $x\in U\cap \RR$ and $0<s_1<s_2$ be such that $\DD_{Ms_2}(x)\subseteq U$, where $M$ is the constant defined in Proposition \ref{lem:main:20}. Then, on the event $\coal_{s_1,s_2}(x)$, almost surely, for all points $z$ such that $z\in U\setminus  \DD_{s_2}(x)$ and points $u_1,u_2$ such that $u_1,u_2\in \DD_{s_1}(x)$, we have
      \begin{equation}
        \label{eq:231}
         D_h(z,u_1;U)-D_h(z,u_2;U)=\fB_h(u_1,u_2).
       \end{equation}
       \end{lemma}
    \begin{proof}
     To begin, we refer the reader to Figure \ref{fig:coal} for an illustration of the simplest case $U=\overline{\HH}$. Now, on the event $\coal_{s_1,s_2}(x)$, we know that any infinite geodesics $\Gamma_{u_1},\Gamma_{u_2}$ must both pass through $w\in \A_{s_1,s_2}(x)$, the coalescence point corresponding to the event $\coal_{s_1,s_2}(x)$, as defined in Proposition \ref{lem:main:20}. Thus, it suffices to show that on $\coal_{s_1,s_2}(x)$, we have the a.s.\ equality  %
      \begin{equation}
        \label{eq:235}
        D_h(z,u_1;U)-D_h(z,u_2;U)=D_h(w,u_1)-D_h(w,u_2),
      \end{equation}
      but this follows by applying \eqref{eq:233} and \eqref{eq:238} with $x'=u_1$ and $x'=u_2$, and then subtracting the obtained equations.

    \end{proof}
    Using the above lemmas, we now demonstrate how the definition of good points in \eqref{eq:126} will be useful in switching between different domains and reference points.
    \begin{lemma}
    \label{lem:7}
    For any closed interval $I\subseteq [-1/2,1/2]$ and $u\in \good_{n,I}$ and for any $\overline{\HH}$-open set $U$ with $I\subseteq U\cap \RR$, there exists a random point $w_u\in \DD_{2^{-n(1-\alpha_2)}}((u+u^+)/2)$ such that almost surely, for all points $z\in U\cap \HH$ with $z\notin\DD_{2^{-n(1-\alpha_2)}}((u+u^+)/2)$ and all points $q\in \DD_{2^{-n(1-\alpha_1)}}((u+u^+)/2)$, we have the equality
    \begin{equation}
      \label{eq:246}
      D_h(z,q;U)=D_h(z,w_u;U)+D_h(w_u,q).
    \end{equation}
  Further, almost surely, for all $u\in\good_{n,I}$ and all $q,q'\in \DD_{2^{-n(1-\alpha_1)}}((u+u^+)/2)$, we have the a.s.\ equality
    \begin{equation}
      \label{eq:251}
      D_h(w_u,q;U)-D_h(w_u,q';U)=\fB_h(q,q'),
    \end{equation}
    and as a consequence, we have $\mu_{n,h}^{z,U,I}(\{u\})=\mu_{n,h}(\{u\})$ almost surely.
  \end{lemma}
  \begin{proof}
    The first statement is an immediate consequence of Lemma \ref{lem:32}, while the second follows directly from Lemma \ref{lem:29}. The last equality follows by applying \eqref{eq:246} with $q=u$ and $q=u^+$ and then subtracting the obtained equations.
\end{proof}

We now present the main result of this section which when combined with the above lemma should be interpreted as the measures $\mu_{n,h}^{z,U,I}$ losing their dependence on $z,U$ asymptotically as $n\rightarrow \infty$.
\begin{proposition}
  \label{lem:9}
  The following convergences hold.
  \begin{enumerate}
  \item Almost surely, we have $\widetilde\mu_{n,h}(\good_{n,[-1/2,1/2]}^c)\rightarrow 0$ as $n\rightarrow\infty$.
  \item For any closed interval $I\subseteq [-1/2,1/2]$, $\overline{\HH}$-open set $U$ with $I\subseteq U\cap \RR$ and point $z\in U\cap \HH$, the measure $\mu^{z,U,I}_{n,h}\lvert_{\good^c_{n,I}}$ a.s.\ converges to the zero measure.
  \item The measure $\mu_{n,h}\lvert_{\good_{n,[-1/2,1/2]}^c}$ a.s.\ converges to the zero measure.
  \end{enumerate}
  \end{proposition}
  For the proof of the above, we will require a decorrelation estimate which we now present.
\begin{lemma}
  \label{lem:5}
  There exists a positive constant $C$ such that for any interval $I=[x,y]\subseteq [-1/2,1/2]$, and any event $F$ which is measurable with respect to $h\lvert_{\overline{\HH}\setminus \DD_{2|I|}(m(I))}$ viewed modulo an additive constant, we have
  \begin{equation}
    \label{eq:120}
    \EE
    \left[
      D_h(x,y; \DD_{|I|}( m(I)))^{\gamma'd_\gamma/(2\gamma)} \ind_{F}
    \right]\leq C|I|^{\psi_\gamma(\gamma'/\gamma)}\sqrt{\PP(F)}.
  \end{equation}
\end{lemma}
 \begin{proof}
   Just as the proof of Lemma \ref{lem:dec}, this proof will be based on the Markov property (Proposition \ref{prop:18}) of the GFF. Applying the Markov property with the domain $S=\DD_{2|I|}(m(I))$, we obtain that for a random harmonic function $\fh$ measurable with respect to $h$, $h\lvert_S-\fh$ is a Dirichlet-Neumann GFF that is independent of $\fh$.

   Now, as in Lemma \ref{lem:dec}, we can couple $h$ with a GFF $h'$ such that $h'\stackrel{d}{=}h$ and if we let $\fh'$ denote the harmonic extension obtained by applying the Markov property for $h'$ with respect to $h$ on $S$, then we have $h'\lvert_S-\fh'=h\lvert_S-\fh$. Further, $h'\lvert_{\overline{\HH}\setminus S}$ and $h\lvert_{\overline{\HH}\setminus S}$ are independent, and as a consequence, $\fh'$ is independent of $h$.

   Now, using $M$ to denote $\sup_{z\in \DD_{|I|}(m(I))}\fh(z)$ and $M'$ to denote $\inf_{z\in \DD_{|I|}(m(I))}\fh'(z)$, we can write
   \begin{equation}
     \label{eq:223}
     e^{\gamma' M'/2}D_h(x,y; \DD_{|I|}( m(I)))^{\gamma'd_\gamma/(2\gamma)} \ind_{F}\leq e^{\gamma'M/2} D_{h'}(x,y; \DD_{|I|}( m(I)))^{\gamma'd_\gamma/(2\gamma)}\ind_{F},
   \end{equation}
   and by taking expectations on both sides, and using the independence structure, we obtain 
  \begin{align}
    \label{eq:198}
     &\EE[e^{\gamma' M'/2}]\EE
    \left[
       D_h(x,y; \DD_{|I|}( m(I)))^{\gamma'd_\gamma/(2\gamma)} \ind_{F}\right]\\
           &\leq \EE[ e^{\gamma' M /2}\ind_{F}]\EE[D_{h'}(x,y; \DD_{|I|}( m(I)))^{\gamma'd_\gamma/(2\gamma)}]\nonumber\\
           &= \EE[ e^{\gamma' M/2}\ind_{F}]\EE[D_{h}(x,y; \DD_{|I|}( m(I)))^{\gamma'd_\gamma/(2\gamma)}]\nonumber\\
    &=\EE[e^{\gamma'h_{|I|}(m(I))/2}]\EE[\exp\{\gamma'M/2-\gamma'h_{|I|}(m(I))/2\}\ind_{F}]\EE[D_{h}(x,y; \DD_{|I|}( m(I)))^{\gamma'd_\gamma/(2\gamma)}]
  \end{align}
  where to obtain the second line, we have used that $F\in \sigma(h\lvert_{\overline{\HH}\setminus S})$ which is independent of $h'$. Also, to obtain the last line in \eqref{eq:198}, we have used the fact (see Lemma \ref{lem:33}) that $h_{|I|}(m(I))$ is independent of $M-h_{|I|}(m(I))$ since the former is in fact independent of $h- h_{|I|}(m(I))$. Also, we have used that $E_I$ is independent of $h_{|I|}(m(I))$ and this is true as $F$ depends only on $h$ viewed modulo an additive constant.

  Due to the same reasoning, we can write
  \begin{equation}
    \label{eq:222}
    \EE[e^{\gamma' M'/2}]= \EE[e^{\gamma'h'_{|I|}(m(I))/2}]\EE[\exp\{\gamma'M'/2-\gamma'h'_{|I|}(m(I))/2\}],
  \end{equation}
  and by substituting this in the previous equation and using that $\EE[e^{\gamma'h'_{|I|}(m(I))/2}]=\EE[e^{\gamma'h_{|I|}(m(I))/2}]$, we obtain
  \begin{align}
    \label{eq:225}
    &\EE[\exp\{\gamma'M'/2-\gamma'h'_{|I|}(m(I))/2\}]\EE
    \left[
      D_h(x,y; \DD_{|I|}( m(I)))^{\gamma'd_\gamma/(2\gamma)} \ind_{F}\right]\nonumber\\
    &\leq \EE[\exp\{\gamma'M/2-\gamma'h_{|I|}(m(I))/2\}\ind_{F}]\EE[D_{h}(x,y; \DD_{|I|}( m(I)))^{\gamma'd_\gamma/(2\gamma)}]\nonumber\\
    &\leq \EE[\exp\{\gamma'M-\gamma'h_{|I|}(m(I))\}]^{1/2}\PP(F)^{1/2}\EE[D_{h}(x,y; \DD_{|I|}( m(I)))^{\gamma'd_\gamma/(2\gamma)}],
  \end{align}
  where in the last line, we used the Cauchy-Schwartz inequality. %

  Now, by the scale invariance of the GFF, we know that if we use $\widetilde{\fh}$ to denote the harmonic function obtained in $\DD_1(0)$ by applying the Markov property to $h$, then we have
  $M-h_{|I|}(m(I))\stackrel{d}{=}\sup_{z\in \mathbb{D}_{1}(1/2)}\widetilde{\fh}(z)$ and $M'-h'_{|I|}(m(I))\stackrel{d}{=}\inf_{z\in \mathbb{D}_{1}(1/2)}\widetilde{\fh}'(z)$. By the Borell-TIS inequality, the quantities $\sup_{z\in \mathbb{D}_{1}(1/2)}\widetilde{\fh}(z)$ and $\inf_{z\in \mathbb{D}_{1}(1/2)}\widetilde{\fh}'(z)$ have subgaussian lower and upper tails. As a consequence, we obtain that for some constant $C$ not depending on $I$, we have
  \begin{align}
    \label{eq:227}
    \EE[\exp\{\gamma'M-\gamma'h_{|I|}(m(I))\}]^{1/2}&\leq C,\nonumber\\
    \EE[\exp\{\gamma'M'/2-\gamma'h'_{|I|}(m(I))/2\}]&\geq C^{-1}.
  \end{align}
  We now combine the above with \eqref{eq:225} to obtain that
  \begin{equation}
    \label{eq:228}
    \EE
    \left[
      D_h(x,y; \DD_{|I|}( m(I)))^{\gamma'd_\gamma/(2\gamma)} \ind_{F}\right]\leq C^2 \PP(F)^{1/2}\EE[D_{h}(x,y; \DD_{|I|}( m(I)))^{\gamma'd_\gamma/(2\gamma)}],
  \end{equation}
  and this completes the proof.

\end{proof}
Recall the events $E_I$ from \eqref{eq:67}. By applying the above with $F=E_I^c$, we immediately obtain the following result.
\begin{lemma}
  \label{lem:36}
  There exist positive constants $C,c$ such that for all intervals $I=[x,y]\subseteq [-1/2,1/2]$, we have
  \begin{equation}
    \label{eq:261}
     \EE
    \left[
      D_h(x,y; \DD_{|I|}( m(I)))^{\gamma'd_\gamma/(2\gamma)} \ind_{E^c_I}
    \right]\leq C|I|^{\psi_\gamma(\gamma'/\gamma)+c}.
  \end{equation}
\end{lemma}
\begin{proof}
  By Lemma \ref{lem:3}, $E_I^c$ is determined by $ h\lvert_{\A_{|I|^{1-\alpha_1},|I|^{1-\alpha_2}}(m(I))}$ viewed modulo an additive constant, and is thus, in particular, determined by $h\lvert_{\overline{\HH}\setminus \DD_{2|I|}(m(I))}$ viewed modulo an additive constant as long as $|I|$ is small enough. Lemma \ref{lem:4} now implies that $\PP(E_I^c)\leq |I|^\chi$ for some $\chi>0$ and combining this with Lemma \ref{lem:5} completes the proof.
\end{proof}

  With the above estimate at hand, we are now ready to complete the proof of Proposition \ref{lem:9}.
  \begin{proof}[Proof of Proposition \ref{lem:9}]
   By using Lemma \ref{lem:35}, we know that both $\mu_{n,h}^{z,U,I}(\good^c_{n,I}), \mu_{n,h}(\good_{n,[-1/2,1/2]}^c)$ are smaller than $\widetilde\mu_{n,h}(\good_{n,[-1/2,1/2]}^c)$ for all $n$ large enough. Thus, to complete the proof, it suffices to show that $\widetilde\mu_{n,h}(\good_{n,[-1/2,1/2]}^c)$ converges to $0$ almost surely as $n\rightarrow \infty$, and this is what we show now.

    To do so, we simply compute the first moment. %
    Indeed, with the constant $c$ coming from Lemma \ref{lem:5}, by using that $\#(\Pi_n \cap[-1/2,1/2])=2^n+1\leq 2^{n+1}$, we have
    \begin{align}
      \label{eq:211}
      \EE[\widetilde\mu_{n,h}(\good_{n,[-1/2,1/2]}^c)]&=\EE\left[\sum_{u\in \Pi_n\cap [-1/2,1/2]} D_h(u,u^+;\DD_{2^{-n}}( (u+u^+)/2))^{\gamma'd_\gamma/(2\gamma)}\ind_{E^c_{[u,u+]}}\right]\nonumber\\
      &\leq 2^{n+1}\times 2^{-n(\psi_\gamma(\gamma'/\gamma)+c)}. 
    \end{align}
    As a result, we immediately obtain that for all $n$,
    \begin{equation}
      \label{eq:212}
      \EE\left[2^{-n(1-\psi_\gamma(\gamma'/\gamma))}\sum_{u\in \Pi_n\cap [-1/2,1/2]} D_h(u,u^+;\DD_{2^{-n}}( (u+u^+)/2))^{\gamma'd_\gamma/(2\gamma)}\ind_{E^c_{[u,u+]}}\right]\leq 2^{1-c n},
    \end{equation}
    and since the right hand side is summable in $n$, an application of the Borel-Cantelli lemma completes the proof.
  \end{proof}

\section{Tightness of the prelimiting measures and the expectation of subsequential limits}
\label{sec:tightness}
For the rest of the paper, we will often work with a closed interval $I\subseteq [-1/2,1/2]$, an $\overline{\HH}$-open set $U$ such that $I\subseteq U\cap \RR$ and a point $z\in U\cap \HH$, and whenever we write $z,U,I$, it will be understood that we are working with a fixed choice of the above. To begin, we first use the uniform integrability from Section \ref{sec:no-atoms} to obtain the existence of subsequential limits of the measures $\mu_{n,h}^{z,U,I}$ defined in \eqref{eq:117}.
\begin{proposition}
    \label{prop:main:1}
    Both the sequences of pairs $(h,\mu^{z,U,I}_{n,h}), (h,\mu_{n,h})$ are tight in $n$ %
 and thus admit subsequential limits which we respectively denote as $(h,\mu^{z,U,I}_h)$ and $(h,\mu_h)$.   %
  \end{proposition}
  \begin{proof}
We only write the proof of the tightness of $(h,\mu_{n,h}^{z,U,I})$; the proof of the tightness of $(h,\mu_{n,h})$ is the same with $\mu_{n,h}^{z,U,I}$ replaced by $\mu_{n,h}$. It is easy to see that we need only prove the tightness of the sequence $\mu_{n,h}^{z,U,I}$ instead of the pair $(h,\mu_{n,h}^{z,U,I})$. We recall that for any $M>0$, Borel measures on $[-1/2,1/2]$ with total mass $\leq M$ form a compact set in the topology of distributional convergence. Thus, by Prokhorov's theorem, it suffices to show that for each $\varepsilon>0$, there exists an $M$ such that $\PP \left(\mu^{z,U,I}_{n,h}([-1/2,1/2])> M\right)\leq \varepsilon$ for all $n$. As a result, it suffices to show that the sequence $\mu^{z,U,I}_{n,h}([-1/2,1/2])$ is uniformly integrable in $n$. Recall the measures $\widetilde \mu_{n,h}$ defined in Section \ref{sec:no-atoms}, and note that, by Lemma \ref{lem:15}, we already know that the random variables $\widetilde \mu_{n,h}([-1/2,1/2])$ are uniformly integrable. By Lemma \ref{lem:35}, the sequence $\mu_{n,h}^{z,U,I}([-1/2,1/2])$ must be uniformly integrable as well, and this completes the proof. %
\end{proof}
The remaining results of this section will hold for any choice of the subsequential limit $(h,\mu^{z,U,I}_{h})$ and $(h,\mu_h)$ as above. Note that even though we use the notation $\mu_h^{z,U,I}$ (resp.\ $\mu_h$) in the above results, it is a priori not clear if the random measure $\mu_h^{z,U,I}$ (resp.\ $\mu_h$) is measurable with respect to $h$. This will established later in Section \ref{sec:meas}. Now, we state a result computing the expectation of the above subsequential limits-- this will later be important in justifying the mean condition in Shamov's characterization of Gaussian multiplicative chaos (Proposition \ref{prop:4} (2)). %
  \begin{proposition}
    \label{lem:14}
With $\kappa$ denoting the constant $\EE|\fB_h(0,1)|^{\gamma' d_\gamma/(2\gamma)}$, for any subsequential weak limits $(h,\mu_h^{z,U,I})$ and $(h,\mu_h)$, and for any interval $J\subseteq I\subseteq [-1/2,1/2]$, we have
  \begin{equation}
    \label{eq:133}
    \EE[\mu^{z,U,I}_h(J)]=  \EE[\mu_h(J)]=\kappa\EE [\nu^{\gamma'}_h(J)].
  \end{equation}
\end{proposition}
The goal of the remainder of this section is to prove the above proposition. Before moving on to the proof of the above, we set up some basic useful results. The following result shows that the subsequential limiting measures $\mu_h,\mu_h^{z,U,I}$ a.s.\ have no atoms.
\begin{proposition}
  \label{prop:3}
  For any subsequential weak limit $(h,\mu_h)$ (resp.\ $(h,\mu_h^{z,U,I})$), the measure $\mu_h$ (resp.\ $\mu_h^{z,U,I}$) a.s.\ has no atoms. In fact, for any fixed $s\in(0, \gamma'^2/4+1-\gamma')$, there exists a random $\delta$ such that almost surely, $\mu_h ( [x,y])\leq (y-x) ^s$ for all intervals $[x,y]\subseteq [-1/2,1/2]$ (resp.\ $[x,y]\subseteq I$) satisfying $(y-x)\leq \delta$. 
\end{proposition}
\begin{proof}
We just write the proof for $\mu_h$ and the same works for the measures $\mu_h^{z,U,I}$ as well. Note that as a consequence of Lemma \ref{lem:35} and Lemma \ref{lem:15}, for any interval $[x,x+\varepsilon]\subseteq [-1/2,1/2]$ and for any $p\in [0,4/\gamma'^2)$, the sequence $\mu_{n,h}([x,x+\varepsilon])^p$ is uniformly integrable in $n$. By using this and \eqref{eq:205}, we obtain that for some constant $C$ depending only on $p,\gamma,\gamma'$, we have
  \begin{equation}
    \label{eq:144}
    \EE[ \mu_h([x,x+\varepsilon])^p]\leq C \varepsilon^{\psi_{\gamma'}(p)},
  \end{equation}
 Now, for any $x,\varepsilon$ as above and any $s>0$, we can write
  \begin{equation}
    \label{eq:145}
    \PP
    \left(
      \mu_h( [x,x+\varepsilon])>\varepsilon^s
    \right)\leq C\varepsilon^{\psi_{\gamma'}(p)-ps}.
  \end{equation}
Now, to maximize the above exponent, we choose $p=(\gamma'^2/2-2s+2)/\gamma'^2$, and we note that with this value of $p$, $\psi_{\gamma'}(p)-ps= (\gamma'^2/2- 2s+2)^2/(4\gamma'^2)$. %
Thus, as long as $p=(\gamma'^2/2-2s+2)/\gamma'^2\in [0,4/\gamma'^2)$, or equivalently $s\in (\gamma'^2/4-1,\gamma'^2/4+1]$, we have
  \begin{equation}
    \label{eq:146}
    \PP
    \left(
      \mu_h( [x,x+\varepsilon])>\varepsilon^s
    \right)\leq C\varepsilon^{ (\gamma'^2/2- 2s+2)^2/(4\gamma'^2)}.
  \end{equation}
 Now, as in the statement of the proposition, we fix $s\in (0,  \gamma'^2/4+1-\gamma')\subseteq (\gamma'^2/4-1,\gamma'^2/4+1]$, and we note that $\gamma'^2/4+1-\gamma'>0$ since $\gamma'\in (0,2)$. For any such $s$, it is easy to see that the exponent appearing in \eqref{eq:146} is strictly greater than $1$. %
 By taking a union bound over the $\varepsilon^{-1}$ many values of $x$, we obtain
  \begin{equation}
    \label{eq:147}
    \PP
    \left(
      \textrm{there exists } x\in \varepsilon \ZZ\cap [-1/2,1/2) ~\mathrm{with} ~\mu_h( [x,x+\varepsilon])>\varepsilon^s
    \right)\leq C\varepsilon^{ (\gamma'^2/2- 2s+2)^2/(4\gamma'^2)-1}.
  \end{equation}
  We can finally consider all small dyadic values of $\varepsilon$ and the result then follows by a Borel-Cantelli argument.
\end{proof}
Now, we present a result showing that whenever the measures $\mu_h^{z,U,I},\mu_h$ can be naturally coupled together, they must satisfy a basic compatibility relation.
 \begin{lemma}
    \label{lem:11}
    For any subsequential weak limit $(h,\mu^{z,U,I}_{h},\mu_{h})$ of $(h,\mu^{z,U,I}_{n,h},\mu_{n,h})$, we have the a.s.\ equality $\mu_h^{z,U,I}=\mu_h\lvert_I$. 
  \end{lemma}
  \begin{proof}
Let $\{n_i\}$ denote the subsequence along which we have the weak convergence of $(h,\mu^{z,U,I}_{n,h},\mu_{n,h})$ to $(h,\mu^{z,U,I}_{h},\mu_{h})$. Since the measures $\mu_h^{z,U,I}$ and $\mu_h$ both a.s.\ have no atoms (Proposition \ref{prop:3}), by the Portmanteau lemma, it follows that along $\{n_i\}$, we also have the weak convergence of $(h,\mu^{z,U,I}_{n,h},\mu_{n,h}\lvert_I)$ to $(h,\mu^{z,U,I}_{h},\mu_{h}\lvert_I)$. With the good set $\good_{n,I}$ from \eqref{eq:126}, we know by Lemma \ref{lem:7} that
  \begin{equation}
    \label{eq:136}
    \mu^{z,U,I}_{n,h}\lvert_{\good_{n,I}}=\mu_{n,h}\lvert_{\good_{n,I}}.
  \end{equation}
  By Proposition \ref{lem:9}, both the measures $\mu^{z,U,I}_{n,h}\lvert_{\good^c_{n,I}}$ and $\mu_{n,h}\lvert_{\good^c_{n,I}}$ a.s.\ converge to the zero measure. This completes the proof.
\end{proof}
  We now move on to the proof of Proposition \ref{prop:main:1}. %
  The first step is to use the invariance properties of the GFF and Busemann functions to obtain the following lemma.
  \begin{lemma}
    \label{lem:21}
    With $\kappa$ denoting the $\gamma,\gamma'$ dependent constant $\EE | \fB_h(1,0)|^{\gamma' d_\gamma/(2\gamma)}$, for any $x<y\in \RR$, we have
    \begin{displaymath}
      \EE\left[\left|\fB_h(x,y)\right|^{\gamma'd_\gamma/(2\gamma)}\right]=\kappa |y-x|^{\psi_\gamma(\gamma'/\gamma)}\EE\left[e^{\gamma' h_1(x)/2}\right].
    \end{displaymath}
  \end{lemma}
  \begin{proof}

  By using the symmetries of Busemann functions (Proposition \ref{prop:19}) and arguing as in the proof of Lemma \ref{lem:26}, we have %
\begin{align}
  \label{eq:22}
 \EE\left[\left|\fB_h(x,y)\right|^{\gamma'd_\gamma/(2\gamma)}\right]&=\EE\left[e^{\gamma' h_1(x)/2}\right]\EE\left[\left|\fB_h(y-x,0)\right|^{\gamma' d_\gamma/(2\gamma)}\right]\\
                                                                    &=|y-x|^{\gamma' Q/2}\EE\left[e^{\gamma' h_1(x)/2}\right] \EE[e^{\gamma' h_{y-x}(0)/2}]\EE [ \fB_h(1,0)|^{\gamma' d_\gamma/(2\gamma)}]\nonumber\\
  &=|y-x|^{\gamma'Q/2-\gamma'^2/4}\EE\left[e^{\gamma' h_1(x)/2}\right]\EE [ |\fB_h(1,0)|^{\gamma' d_\gamma/(2\gamma)}],
\end{align}
and the proof is now completed by noting that $\psi_\gamma(\gamma'/\gamma)=\gamma'Q/2-\gamma'^2/4$.
\end{proof}
\begin{remark}
  \label{rem:busem}
  We note that in Lemma \ref{lem:21}, it is important to use Busemann functions, which, as mentioned earlier, are intuitively thought of as formalizing the notation of distances from $\infty$. Indeed, Lemma \ref{lem:21} can be thought of as the informal statement
  \begin{displaymath}
    \EE[|D_h(\infty,x)-D_h(\infty,y)|^{\gamma'd_\gamma/(2\gamma)}]= \EE[|D_h(\infty,1)-D_h(\infty,0)|^{\gamma'd_\gamma/(2\gamma)}]|y-x|^{\psi_\gamma(\gamma'/\gamma)}\EE\left[e^{\gamma' h_1(x)/2}\right].
  \end{displaymath}
  The crucial special property of $\infty$ that is used is its invariance under any translation and scaling of the upper half plane. Indeed, if we replace $\infty$ by a fixed point $z\in \HH$ and attempt to use Weyl scaling to obtain a statement of the above form, then we only obtain that $\EE[|D_h(z,x)-D_h(z,y)|^{\gamma'd_\gamma/(2\gamma)}]$ is equal to
  \begin{displaymath}
  \EE[|D_h((z-x)/(y-x),1)-D_h((z-x)/(y-x),0)|^{\gamma'd_\gamma/(2\gamma)}]|y-x|^{\psi_\gamma(\gamma'/\gamma)}\EE\left[e^{\gamma' h_1(x)/2}\right],
\end{displaymath}
and this is not useful since we need distances to be measured from the same reference point $z$ on both sides of the equation.
\end{remark}
In the following lemma, we use Lemma \ref{lem:21} to compute $\lim_{n\rightarrow \infty}\EE\mu_{n,h}^{z,U,I}([-1/2,1/2])$.
  \begin{lemma}
    \label{lem:8}
With $\kappa$ denoting the $\gamma,\gamma'$ dependent constant $\EE |\fB_h(1,0)|^{\gamma'd_\gamma/(2\gamma)}$, for any subsequential weak limits $(h,\mu_h^{z,U,I})$ and $(h,\mu_h)$, and for any fixed closed interval $J\subseteq I\subseteq [0,1]$, we have
  \begin{equation}
    \label{eq:18.1}
    \lim_{n\rightarrow \infty}\EE[\mu^{z,U,I}_{n,h}(J)]=\lim_{n\rightarrow \infty}\EE[\mu_{n,h}(J)]=\kappa\int_J \EE\left[e^{\gamma' h_1(x)/2} \right]dx=\kappa\EE \nu^{\gamma'}_h(J).
  \end{equation}    
  \end{lemma}
  \begin{proof}
  We begin by noting that $\EE |\fB_h(1,0)|^{\gamma'd_\gamma/(2\gamma)}$ is finite as a consequence of Lemma \ref{prop:9} and the fact that $|\fB_h(1,0)|\leq D_h(1,0)$ and $\gamma'd_\gamma/(2\gamma)< Q d_\gamma/\gamma$. Now, by Lemma \ref{lem:31}, we know that $\EE\nu^{\gamma'}_h(J)= \int_J \EE\left[e^{\gamma' h_1(x)/2} \right]dx$, and it thus suffices to show the first two equalities in \eqref{eq:18.1}. Recall the local proxy measures $\widetilde \mu_{n,h}$ introduced in \eqref{eq:154}. By Lemma \ref{lem:35}, we know that for all $n$ large enough, and all Borel sets $A\subseteq [-1/2,1/2]$, we have $\mu_{n,h}^{z,U,I}(A)\leq \widetilde \mu_{n,h}(A)$ and $\mu_{n,h}(A)\leq \widetilde \mu_{n,h}(A)$, and thus, by Lemma \ref{lem:15}, the sequences $\mu_{n,h}^{z,U,I}([-1/2,1/2])$ and $\mu_{n,h}([-1/2,1/2]$ are both uniformly integrable. Using this along with the almost sure convergence from Proposition \ref{lem:9}, we know that as $n\rightarrow \infty$, we have the convergences
    \begin{equation}
      \label{eq:240}
      \EE[\mu_{n,h}^{z,U,I}(\good_{n,J}^c)]\rightarrow 0, \EE[\mu_{n,h}(\good_{n,J}^c)]  \rightarrow 0.
    \end{equation}
    As a result of this, to complete the proof, it suffices to show that
      \begin{equation}
      \label{eq:130}
      \lim_{n\rightarrow \infty}  \EE\left[\sum_{u\in \good_{n,J}}\mu_{n,h}^{z,U,I}(\{u\})\right]=\lim_{n\rightarrow \infty}  \EE\left[\sum_{u\in \good_{n,J}}\mu_{n,h}(\{u\})\right]=\kappa\int_J \EE\left[e^{\gamma' h_1(x)/2} \right]dx.
    \end{equation}
   Now by Lemma \ref{lem:7}, we obtain that 
    \begin{equation}
      \label{eq:132}
      \EE\left[\sum_{u\in \good_{n,J}}\mu_{n,h}^{z,U,I}(\{u\})\right]=\EE\left[\sum_{u\in \good_{n,J}}\mu_{n,h}(\{u\})\right].
    \end{equation}
   Thus, by using the second convergence in \eqref{eq:240} along with the above, we need only establish the convergence $\lim_{n\rightarrow \infty} \EE[\mu_{n,h}(J)]=\kappa\int_J \EE[e^{\gamma' h_1(x)/2}]dx$. However, since $|u^+-u|=2^{-n}$, we can use Lemma \ref{lem:21} to obtain that for each $u\in \Pi_n\cap J$, we have
    \begin{equation}
      \label{eq:204}
      \EE \mu_{n,h}(\{u\})= 2^{-n(1-\psi_\gamma(\gamma'/\gamma))}\EE\left|\fB_h(u,u^+)\right|^{\gamma'd_\gamma/(2\gamma)}=\kappa 2^{-n}\EE\left[e^{\gamma' h_1(u)/2}\right].
    \end{equation}
    By Lemma \ref{lem:30}, we know that $x\mapsto \EE e^{\gamma' h_1(x)/2}$ is continuous in $x$. %
 By using this along with the definition of the Riemann integral, we obtain that
    \begin{equation}
      \label{eq:242}
      \EE[\mu_{n,h}(J)]= \kappa 2^{-n}\sum_{u\in \Pi_n\cap J}\EE[e^{\gamma' h_1(u)/2}]\rightarrow \kappa \int_{J} \EE[e^{\gamma' h_1(x)/2}]dx
    \end{equation}
    as $n\rightarrow \infty$, and this completes the proof.
\end{proof}
We emphasize that right hand side of $\eqref{eq:18.1}$ does not depend on $z,U,I$. We are now ready to complete the proof of Proposition \ref{lem:14}. %
\begin{proof}[Proof of Proposition \ref{lem:14}]
We only write down the proof of $\EE[\mu_h(J)]=\kappa\EE \nu^{\gamma'}_h(J)$, and the corresponding statement for $\mu_{h}^{z,U,I}$ has the same proof. Since $\mu_{h}$ a.s.\ has no atoms (Proposition \ref{prop:3}), we know that $\mu_{n,h}(J)\stackrel{d}{\rightarrow} \mu_h(J)$ along the sequence $\{n_i\}$. Now, by Lemma \ref{lem:35} and Lemma \ref{lem:15}, we know that the random variables $\mu_{n,h}(J)$ are uniformly integrable in $n$. This implies that we have $\EE[\mu_h(J)]=\lim_{n\rightarrow \infty} \EE[\mu_{n,h}(J)]$, and an application of Lemma \ref{lem:8} now completes the proof.
  
\end{proof}

\section{The measure $\mu_h$ satisfies $\gamma'$-Weyl scaling}
\label{sec:weyl}
The main goal of this section is to prove the following Weyl scaling result, and the proof of this is the heart of this paper.
\begin{proposition}
  \label{lem::3}
Fix a closed interval $I\subseteq [-1/2,1/2]$, an $\overline{\HH}$-open set $U$ with $I\subseteq U\cap \RR$, and a random differentiable function $\zeta$ on $U$ coupled to $h$ such that the field $h+\zeta$ is mutually absolutely continuous to $h$.
  Consider a subsequence $\{n_i\}$ along which $(h,\mu_{n,h})$ converge weakly and denote the subsequential limit by $(h,\mu_h)$. Then, along the same subsequence, the tuple $(h,\mu_{n,h},\mu_{n,h+\zeta}\lvert_I)$ converges weakly to a limit $(h,\mu_h,\widetilde{\mu}_{h+\zeta,I})$ which additionally a.s.\ satisfies %
  \begin{equation}
    \label{eq:256}
    d\widetilde{\mu}_{h+\zeta,I}=e^{\gamma' \zeta /2}d(\mu_{h}\lvert_I).
  \end{equation}
\end{proposition}
The goal of the remainder of this section is to provide the proof of Proposition \ref{lem::3}. However, we first state a consequence of Proposition \ref{lem::3}, and this will be used directly to verify the Weyl scaling condition in Shamov's characterization of Gaussian multiplicative chaos (Proposition \ref{prop:4} (3)).
\begin{lemma}
  \label{lem:38}
Fix a function $\phi$ on $\overline{\HH}$ defined as $\phi(v)=\int_\HH G(v,w)\rho(w)dw$ for some function $\rho\in \cD(\overline{\HH})$. %
  Consider a subsequence $\{n_i\}$ along which $(h,\mu_{n,h})$ converges weakly and denote the subsequential limit by $(h,\mu_h)$. Then, along the same subsequence, the tuple $(h,\mu_{n,h},\mu_{n,h+\phi})$ converges weakly to a limit $(h,\mu_h,\mu_{h+\phi})$ which additionally a.s.\ satisfies %
  \begin{equation}
    \label{eq:256}
    d\mu_{h+\phi}=e^{\gamma' \phi /2}d\mu_{h}.
  \end{equation}  
\end{lemma}
\begin{proof}[Proof assuming Proposition \ref{lem::3}]
We simply set $\zeta=\phi$, $I=[-1/2,1/2]$ and $U=\overline{\HH}$ and invoke Proposition \ref{lem::3}. Note that $h+\phi$ and $h$ are mutually absolutely continuous (Proposition \ref{prop:7}).  
\end{proof}
To set up the stage for the proof of Proposition \ref{lem::3}, we now introduce some definitions and some technical results, and for the remainder of this section, we work in the setting of Proposition \ref{lem::3}. To begin, we define the set of $\zeta$-good points $\good_{n,I,\zeta}$ for a fixed closed interval $I\subseteq [-1/2,1/2]$ as
\begin{equation}
  \label{eq:137}
  \good_{n,I,\zeta}=\{u\in \Pi_n\cap I: E_I \textrm{ occurs both for } h,h+\zeta\}.
\end{equation}
and use $\good^c_{n,I,\zeta}$ to denote the set $(\Pi_n\cap I)\setminus \good_{n,I,\zeta}$. Analogous to Proposition \ref{lem:9}, we have the following result about the set $\good^c_{n,I,\zeta}$.

\begin{lemma}
  \label{lem:37}
  Both the measures $\mu_{n,h}\lvert_{\good_{n,I,\zeta}}$ and $\mu_{n,h+\zeta}\lvert_{\good_{n,I,\zeta}}$ a.s.\ converge to the zero measure.
\end{lemma}
\begin{proof}
  We first prove the a.s.\ convergence of $\mu_{n,h+\zeta}\lvert_{\good_{n,I,\zeta}}$ to the zero measure. We define the event
  \begin{equation}
    \label{eq:263}
    H_I=\{E_I \textrm{ occurs for the field }h+\zeta\},
  \end{equation}
  and let $\widetilde{\good}_{n,I}$ denote the set of $u\in \Pi_n\cap I$ for which $H_I$ occurs. As usual, we define $\widetilde{\good}_{n,I}^c=(\Pi_n\cap I)\setminus \widetilde{\good}_{n,I}$. By using the absolute continuity of $h+\zeta$ with respect to $h$ along with Proposition \ref{prop:9}, we know that $\mu_{n,h+\zeta}\lvert_{\widetilde{\good}^c_{n,I}}$ converges a.s.\ to the zero measure. Noting that $\good_{n,I,\zeta}=\good_{n,I}\cap \widetilde{\good}_{n,I}$, it suffices to establish the a.s.\ convergence of $\mu_{n,h+\zeta}\lvert_{\good_{n,I}^c}$ to the zero measure. Using the proxy measures $\widetilde{\mu}_{n,h}$ from \eqref{eq:154}, as in the proof of Proposition \ref{prop:9}, we can reduce to showing the a.s.\ convergence $\widetilde{\mu}_{n,h+\zeta}(\good_{n,I}^c)\rightarrow 0$, %
  Now fix an $\overline{\HH}$-open set $V$ such that $I\subseteq V\cap \RR$ and $\overline{V}\subseteq U$. By a simple application of Weyl scaling (Proposition \ref{prop:5}) and the fact that $\DD_{2^{-n}}( (u+u^+)/2)\subseteq V$ for all $u\in \Pi_n\cap I$ an all large $n$, it is easy to see that for any Borel set $A\subseteq I$, we have
  \begin{equation}
    \label{eq:266}
    \widetilde{\mu}_{n,h+\zeta}(A)\leq \sup_{v\in V}e^{\gamma'\zeta(v)/2} \widetilde{\mu}_{n,h}(A)
  \end{equation}
for all $n$ large enough, and we note that $\sup_{v\in V}e^{\gamma'\zeta(v)/2}$ is a.s.\ finite since $\zeta$ is continuous on $U$ and $\overline{V}\subseteq U$. By Proposition \ref{lem:9}, we already know that $\mu_{n,h}(\good^c_{n,I})$ a.s.\ converges to $0$, and by using \eqref{eq:266}, we obtain that $\mu_{n,h+\zeta}(\good^c_{n,I})$ a.s.\ converges to $0$ as well.

  We now use the above along with an absolute continuity trick to also obtain the a.s.\ convergence of $\mu_{n,h}\lvert_{\good^c_{n,I,\zeta}}$ to the zero measure. Indeed, by writing $h= (h+\zeta) + (-\zeta)$ and using the result of the previous paragraph along with the mutual absolute continuity between $h$ and $h+\zeta$, we obtain the needed convergence. %
\end{proof}

For a given $u\in \Pi_n$, we define the semicircle  
\begin{equation}
  \label{eq:138}
  \cS_u=\TT_{2^{-n(1-\alpha_2)}}((u+u^+)/2).
\end{equation}
Soon, we shall work with distances $D_h(K,q)$ for a point $q\in \overline{\HH}$ and a compact set $K\subseteq \overline{\HH}$ which we define simply by
\begin{equation}
  \label{eq:237}
  D_h(K,q)=\min_{z\in K}D_h(z,q).
\end{equation}
We will usually consider the distances $D_h(\cS_u,u)$ for some $u\in \Pi_n$. In fact, we can also associate geodesics to such distances which we now introduce. By a basic compactness argument, for any point $q\in \DD_{2^{-n(1-\alpha_1)}}((u+u^+)/2)$, it can be seen that there exists a geodesic $\Gamma^h_{\cS_u,q}$ which is a path going from some point $p^q_u$ in $\cS_u$ to $q$ which has the property that $\ell(\Gamma^h_{\cS_u,q};D_h)=D_h(\cS_u,q)$. Since the path $\Gamma^h_{\cS_u,q}$ is a geodesic, it is easy to see that $\Gamma^h_{\cS_u,q}\cap \cS_u=\{p^q_u\}$, and as a consequence, we obtain that
  \begin{equation}
    \label{eq:232}
    \Gamma^h_{\cS_u,q}\subseteq \overline{\DD}_{2^{-n(1-\alpha_2)}}((u+u^+)/2),
  \end{equation}
  and this will be important to us. By using these geodesics along with Lemma \ref{lem:7}, we obtain the following result.
  \begin{lemma}
    \label{lem:34}
    For any fixed closed interval $I\subseteq [-1/2,1/2]$, almost surely, for all $u\in \good_{n,I}$ and $q,q'\in \DD_{2^{-n(1-\alpha_2)}}((u+u^+)/2)$, we have the a.s.\ equality
  \begin{equation}
    \label{eq:250}
    D_h(\cS_u,q)-D_h(\cS_u,q')=\fB_h(q,q').
  \end{equation}
  \end{lemma}
  \begin{proof}
  Since $\Gamma^h_{\cS_u,q}$ is also a geodesic from some point $p_u^q\in \cS_u$ to $q$, by using $w_u$ to denote the random point defined in Lemma \ref{lem:7}, \eqref{eq:246} leads to the equality
  \begin{equation}
    \label{eq:248}
    D_h(\cS_u,q)=D_h(p^q_u,q)=D_h(p^q_u,w_u)+D_h(w_u,q).
  \end{equation}
  In other words, we can choose the geodesic $\Gamma^h_{\cS_u,q}$ to additionally pass through the point $w_u$, and with this in mind, we also know that $D_h(p^q_u,w_u)=D_h(\cS_u,w_u)$, and we thus have
  \begin{equation}
    \label{eq:249}
    D_h(\cS_u,q)=D_h(\cS_u,w_u)+D_h(w_u,q).
  \end{equation}
  Now, Lemma \ref{lem:7} implies that $D_h(w_u,q)-D_h(w_u,q')=\fB_h(q,q')$, and thus on applying \eqref{eq:249} with $q$ and $q'$ and subtracting the obtained equations, we obtain the needed result.   
  \end{proof}
 For context, we note that the above will soon be used with $q=u$ and $q'=u'$.
We now present a technical result which will shortly be used to show that certain error terms arising in the proof of Proposition \ref{lem::3} are indeed small. We note that the precise choice of $\alpha_2$ made in \eqref{eq:206} earlier is important for this result. 
\begin{lemma}
  \label{lem:22}
  Almost surely, we have the convergence
  \begin{displaymath}
    2^{-n(1-\psi_\gamma(\gamma'/\gamma))}\sum_{u\in \Pi_n\cap [-1/2,1/2]} (2^{-n(1-\alpha_2)}(D_h(\cS_u,u) +D_h(\cS_u,u^+)))^{\gamma'd_\gamma/(2\gamma)}\rightarrow 0
  \end{displaymath}
  as $n\rightarrow \infty$.
\end{lemma}
\begin{proof}
  We first note that
  \begin{equation}
    \label{eq:185}
    (D_h(\cS_u,u)+D_h(\cS_u,u^+))^{\gamma'd_\gamma/(2\gamma)}\leq 2^{\gamma'd_\gamma/(2\gamma)}(D_h(\cS_u,u)^{\gamma'd_\gamma/(2\gamma)}+D_h(\cS_u,u^+)^{\gamma'd_\gamma/(2\gamma)}).
  \end{equation}
 We now show that almost surely
  \begin{equation}
    \label{eq:191}
     2^{-n(1-\psi_\gamma(\gamma'/\gamma))}\sum_{u\in \Pi_n\cap [-1/2,1/2]} 2^{-n(1-\alpha_2)\gamma'd_\gamma/(2\gamma)}(D_h(\cS_u,u)^{\gamma'd_\gamma/(2\gamma)}+D_h(\cS_u,u^+)^{\gamma'd_\gamma/(2\gamma)}) \rightarrow 0,
   \end{equation}
as $n\rightarrow \infty$, and this would complete the proof.
  To show \eqref{eq:191}, we simply compute the expectation of the term therein and apply the Borel-Cantelli lemma. %
  
  By using Weyl scaling in the same manner as in the proof of Lemma \ref{lem:26}, it follows that
  \begin{align}
    \label{eq:180}
    &\EE(D_h(\cS_u,u)^{\gamma'd_\gamma/(2\gamma)}+D_h(\cS_u,u^+)^{\gamma'd_\gamma/(2\gamma)})\nonumber\\&=(2^{-n(1-\alpha_2)})^{\psi_\gamma(\gamma'/\gamma)}\EE[e^{\gamma' h_1(u)/2}]\left(\EE D_h(\mathbb{T}_1(0),0)^{\gamma'd_\gamma/(2\gamma)}+  \EE D_h(\mathbb{T}_1(0),2^{-n\alpha_2})^{\gamma'd_\gamma/(2\gamma)}\right),
  \end{align}
  and we note that $\EE D_h(\mathbb{T}_1(0),0)^{\gamma'd_\gamma/(2\gamma)}\leq \EE D_h(1,0)^{\gamma'd_\gamma/(2\gamma)}<\infty$ by an application of Proposition \ref{prop:9} along with the fact that $\gamma'd_\gamma/(2\gamma)<Q d_\gamma/\gamma$ which holds since $Q>2$ and $\gamma'<2$. Further, it is easy to see that $\EE D_h(\mathbb{T}_1(0),2^{-n\alpha_2})^{\gamma'd_\gamma/(2\gamma)}\rightarrow \EE D_h(\mathbb{T}_1(0),0)^{\gamma'd_\gamma/(2\gamma)}$ as $n\rightarrow \infty$, and as a result, we obtain that for some constant $C$ not depending on the choice of $u$, and for all $n$, we have
  \begin{equation}
    \label{eq:260}
    \EE(D_h(\cS_u,u)^{\gamma'd_\gamma/(2\gamma)}+D_h(\cS_u,u^+)^{\gamma'd_\gamma/(2\gamma)})\leq C(2^{-n(1-\alpha_2)})^{\psi_\gamma(\gamma'/\gamma)}\EE[e^{\gamma' h_1(u)/2}].
  \end{equation}
  Using the above along with the fact (see Lemma \ref{lem:30}) that $\EE[e^{\gamma' h_1(u)/2}]\leq C_1$ for some constant $C_1$ not depending on the choice of $u\in \Pi_n \cap [-1/2,1/2]$, we obtain that
  \begin{align}
    &2^{-n(1-\psi_\gamma(\gamma'/\gamma))}\sum_{u\in \Pi_n\cap [-1/2,1/2]} 2^{-n(1-\alpha_2)\gamma'd_\gamma/(2\gamma)}\left(\EE D_h(\cS_u,u)^{\gamma'd_\gamma/(2\gamma)}+\EE D_h(\cS_u,u^+)^{\gamma'd_\gamma/(2\gamma)}\right)\nonumber\\
    &\leq CC_1 (2\times 2^n)\times 2^{-n(1-\psi_\gamma(\gamma'/\gamma))}\times 2^{-n(1-\alpha_2)\gamma'd_\gamma/(2\gamma)}\times 2^{-n(1-\alpha_2)\psi_\gamma(\gamma'/\gamma)}\nonumber\\
    &=(2CC_1)2^{n\alpha_2\psi_\gamma(\gamma'/\gamma)}\times 2^{-n(1-\alpha_2)\gamma'd_\gamma/(2\gamma)}\label{eq:201},
  \end{align}
  where in the second line, we have used that $\#(\Pi_n\cap [-1/2,1/2])=2^n+1\leq 2\times 2^n$. Now, we note that the above expression decays exponentially in $n$ since the choice of $\alpha_2$, made in \eqref{eq:206}, ensures that $(1-\alpha_2)\gamma'd_\gamma/(2\gamma)-\alpha_2\psi_\gamma(\gamma'/\gamma)>0$. The proof can now be completed by applying the Borel-Cantelli lemma and using the summability in $n$ of the expression in \eqref{eq:201}.
\end{proof}

We are now ready for the proof of Proposition \ref{lem::3}.
\begin{proof}[Proof of Proposition \ref{lem::3}]
 We define the measure $\widetilde{\mu}_{h+\zeta,I}$ on $I$ by the relation $d\widetilde{\mu}_{h+\zeta,I}= e^{\gamma'\zeta/2}d(\mu_h\lvert_I)$, and the goal of this proof to show that along the subsequence $\{n_i\}$, we have the weak convergence of the measures $\mu_{n,h+\zeta}\lvert_I$ to the measure $\widetilde{\mu}_{h+\zeta,I}$; it is not difficult to see that this would also imply the weak convergence of $(h,\mu_{n,h},\mu_{n,h+\zeta}\lvert_I)$ to $(h,\mu_h,\widetilde{\mu}_{h+\zeta,I})$ along $\{n_i\}$. Now, by standard arguments involving the weak convergence of random measures (see e.g.\ \cite[Theorem 4.11]{Kal17}), it suffices to fix a bounded continuous function $f$ on $I$ and just establish the weak convergence
  \begin{equation}
    \label{eq:257}
    \int_Ifd\mu_{n,h+\zeta}\stackrel{d}{\rightarrow} \int_I fd\widetilde{\mu}_{h+\zeta,I}=\int_I fe^{\gamma'\zeta/2}d\mu_h,
  \end{equation}
 along $\{n_i\}$, and the rest of the proof is focused on establishing this. We note that we will often work with the good points $\good_{n,I,\zeta}$ defined in \eqref{eq:137}. Also, since all the convergences in this proof will be along the subsequence $\{n_i\}$, we do not mention this explicitly from now on.

  To keep the notation simple in the upcoming computations, we introduce some shorthand. For $u\in \good_{n,I,\zeta}$, we define $a_u,b_u$ by
  \begin{align}
    \label{eq:183}
    a_u&\coloneqq e^{\xi \zeta(u)} (D_{h}(\cS_u,u)-D_{h}(\cS_u,u^+))=e^{\xi \zeta(u)} \fB_h(u,u^+)\\
   \label{eq:183.1} b_u&\coloneqq (D_{h+\zeta}(\cS_u,u)-D_{h+\zeta}(\cS_u,u^+))- e^{\xi \zeta(u)}(D_{h}(\cS_u,u)-D_{h}(\cS_u,u^+)),
  \end{align}
  where we note that the second equality in the first line holds because of Lemma \ref{lem:34} and the fact $\good_{n,I,\zeta}\subseteq \good_{n,I}$. Similarly, we also know that for all $u\in \good_{n,I,\zeta}$,
  \begin{equation}
    \label{eq:190}
    a_u+b_u=D_{h+\zeta}(\cS_u,u)-D_{h+\zeta}(\cS_u,u^+)=\fB_{h+\zeta}(u,u^+).
  \end{equation}
We note that intuitively, in the above expression, the term $a_u$ should be interpreted as the ``main'' term, and the term $b_u$ as the ``error'' term, and indeed, the goal of this proof is to formalize this. As a consequence of Lemma \ref{lem:37}, both $\mu_{n,h}\lvert_{\good_{n,I,\zeta}^c}$ and $\mu_{n,h+\zeta}\lvert_{\good_{n,I,\zeta}^c}$ converge a.s.\ to the zero measure. 
  As a result, to show \eqref{eq:257}, we need only establish that 
  \begin{equation}
    \label{eq:162}
    2^{-n(1-\psi_\gamma(\gamma'/\gamma))}\sum_{u\in \good_{n,I,\zeta}}f(u)|a_u+b_u|^{\gamma' d_\gamma/(2\gamma) }\stackrel{d}{\rightarrow}\int_I fe^{\gamma'\zeta/2}d\mu_h,
  \end{equation}
as $n\rightarrow \infty$, and this is what we shall prove. We now fix an $\overline{\HH}$-open set $V$ such that $I\subseteq V\cap \RR$ and $\overline{V}\subseteq U$. Since $\zeta$ is differentiable on $U$ and has bounded derivative on $V$, there exists a random constant $C$ such that for all $n$ large enough, and for all points $v\in \overline{\DD}_{2^{-n(1-\alpha_2)}}((u+u^+)/2)$ and $u\in \good_{n,I,\zeta}$, we have
  \begin{equation}
    \label{eq:207}
    |\zeta(v)-\zeta(u)|\leq C2^{-n(1-\alpha_2)}.
  \end{equation}
Now, by using that (see \eqref{eq:232}) all the geodesics $\Gamma^h_{\cS_u,u},\Gamma^h_{\cS_u,u^+}, \Gamma^{h+\zeta}_{\cS_u,u},\Gamma^{h+\zeta}_{\cS_u,u^+}$ lie inside $\overline{\DD}_{2^{-n(1-\alpha_2)}}((u+u^+)/2)$, along with the Weyl scaling of the LQG metric ((2) in Proposition \ref{prop:5}), we conclude that for each $u\in \good_{n,I,\zeta}$, we have
  \begin{align}
    \label{eq:209}
    &\exp(\xi \zeta(u)-C\xi 2^{-n(1-\alpha_2)})D_h(\cS_u,u)\leq D_{h+\zeta}(\cS_u,u)\leq \exp(\xi \zeta(u)+C\xi 2^{-n(1-\alpha_2)})D_h(\cS_u,u),\nonumber\\
    &\exp(\xi \zeta(u)-C\xi 2^{-n(1-\alpha_2)})D_h(\cS_u,u^+)\leq D_{h+\zeta}(\cS_u,u^+)\leq \exp(\xi \zeta(u)+C\xi 2^{-n(1-\alpha_2)})D_h(\cS_u,u^+).
  \end{align}
 Now, by using the basic fact $1-2x\leq e^x\leq 1+ 2x$ for all small $x$, we obtain that with $C_1=2\xi C$,
  \begin{equation}
    \label{eq:182}
     b_u\in e^{\xi\zeta(u)}(D_h(\cS_u,u)+D_h(\cS_u,u^+))[-C_12^{-n(1-\alpha_2)},C_12^{-n(1-\alpha_2)}]
  \end{equation}
for each $u\in \good_{n,I,\zeta}$ and all large $n$. Further, we note that for each $\theta\in (0,1)$, and any $x,y\in \RR$, we have the deterministic inequality $|x+y|\leq |x|+|y|\leq \max( (1+\theta) |x|, (1+\theta^{-1}) |y|)$ and thus 
  \begin{equation}
    \label{eq:187}
 |x+y|^{\gamma'd_\gamma/(2\gamma)}\leq (1+\theta)^{\gamma' d_\gamma/(2\gamma)} |x| ^{\gamma' d_\gamma/(2\gamma)} + (1+\theta^{-1})^{\gamma' d_\gamma/(2\gamma)} |y|^{\gamma' d_\gamma/(2\gamma)}.
\end{equation}
 Using the above with $x=a_u$ and $y=b_u$, we obtain
  \begin{equation}
    \label{eq:186}
    |a_u+b_u|^{\gamma' d_\gamma/(2\gamma)} \leq (1+\theta)^{\gamma' d_\gamma/(2\gamma)} | a_u| ^{\gamma' d_\gamma/(2\gamma)} + (1+\theta^{-1})^{\gamma' d_\gamma/(2\gamma)} |b_u|^{\gamma' d_\gamma/(2\gamma)}.
  \end{equation}
  Similarly, on using \eqref{eq:187} with $x=a_u+b_u$ and $y=-b_u$, we obtain
  \begin{equation}
    \label{eq:188}
    |a_u+b_u|^{\gamma' d_\gamma/(2\gamma)}\geq (1+\theta)^{-\gamma' d_\gamma/(2\gamma)}|a_u|^{\gamma' d_\gamma/(2\gamma)} - \left(\frac{1+\theta^{-1}}{1+\theta}\right)^{\gamma' d_\gamma/(2\gamma)}|b_u|^{\gamma' d_\gamma/(2\gamma)}.
  \end{equation}
  Now, as a consequence of Lemma \ref{lem:22}, \eqref{eq:182} and the fact that $\sup_{x\in I} \zeta(x)<\infty$, we know that $2^{-n(1-\psi_\gamma(\gamma'/\gamma))}\sum_{u\in \good_{n,I,\zeta}}|b_u|^{\gamma' d_\gamma/(2\gamma)}\rightarrow 0$ almost surely as $n\rightarrow \infty$. As a consequence of this and the boundedness of $f$, for each $\theta\in (0,1)$ and $\varepsilon >0$, we a.s.\ have for all $n$ large enough,
  \begin{align}
    \label{eq:189}
    &2^{-n(1-\psi_\gamma(\gamma'/\gamma))}(1+\theta)^{-\gamma' d_\gamma/(2\gamma)}\sum_{u\in \good_{n,I,\zeta}} f(u)|a_u|^{\gamma' d_\gamma/(2\gamma)}-\varepsilon\nonumber\\
    &\leq 2^{-n(1-\psi_\gamma(\gamma'/\gamma))}\sum_{u\in \good_{n,I,\zeta}}f(u)|a_u+b_u|^{\gamma' d_\gamma/(2\gamma)}\nonumber\\
    &\leq 2^{-n(1-\psi_\gamma(\gamma'/\gamma))}(1+\theta)^{\gamma' d_\gamma/(2\gamma)}\sum_{u\in \good_{n,I,\zeta}} f(u)|a_u|^{\gamma' d_\gamma/(2\gamma)}+\varepsilon
  \end{align}
Since we have the weak convergence of $\mu_{n,h}$ to $\mu_h$ and the latter measure a.s.\ has no atoms (Proposition \ref{prop:3}), by the Portmanteau lemma, we obtain that as $n\rightarrow \infty$,
  \begin{equation}
    \label{eq:258}
    2^{-n(1-\psi_\gamma(\gamma'/\gamma))}\sum_{u\in \good_{n,I,\zeta}}|\fB_h(u,u^+)|^{\gamma' d_\gamma/(2\gamma)}\delta_u\stackrel{d}{\rightarrow} \mu_h\lvert_I,
  \end{equation}
 and we note that the above also uses the a.s.\ convergence of $\mu_{n,h}\lvert_{\good_{n,I,\zeta}^c}$ to the zero measure. %
Since $f\exp(\gamma'\zeta/2)$ is a bounded continuous function on $I$, by using the definition of weak convergence, we obtain that as $n\rightarrow \infty$,
  \begin{align}
    \label{eq:165}
    &2^{-n(1-\psi_\gamma(\gamma'/\gamma))}\sum_{u\in \good_{n,I,\zeta}}f(u)|a_u|^{\gamma' d_\gamma/(2\gamma)}\nonumber\\
    &= 2^{-n(1-\psi_\gamma(\gamma'/\gamma))}\sum_{u\in \good_{n,I,\zeta}}f(u)e^{\gamma' \zeta(u)/2}\left|  \fB_h(u,u^+)\right|^{\gamma' d_\gamma/(2\gamma)}
     \stackrel{d}{\rightarrow} \int_I fe^{\gamma'\zeta /2} d\mu_h,
   \end{align}
   where we have used \eqref{eq:183}. Now, by taking $\theta$ to be small followed by taking $\varepsilon$ to be small, we obtain by using \eqref{eq:189} along with \eqref{eq:165} that
    \begin{equation}
     \label{eq:184}
    2^{-n(1-\psi_\gamma(\gamma'/\gamma))}\sum_{u\in \good_{n,I,\zeta}}f(u)|a_u+b_u|^{\gamma' d_\gamma/(2\gamma)}\stackrel{d}{\rightarrow} \int_I fe^{\gamma'\zeta /2} d\mu_h
  \end{equation}
as $n\rightarrow \infty$, and this establishes \eqref{eq:162}, thereby completing the proof.

 \end{proof}
\section{Measurability of $\mu_h$ with respect to $h$}
\label{sec:meas}
The aim of this section is to prove the following result which will be used in the next section to verify the measurability condition in Shamov's characterization of Gaussian multiplicative chaos (Proposition \ref{prop:4} (1)).
 \begin{proposition}
    \label{prop:1}
    For any subsequential weak limit $(h,\mu_h^{z,U,I})$, the measure $\mu_h^{z,U,I}$ is determined by $h$. Further, for any deterministic subsequence $\{n_i\}$ for which $(h,\mu^{z,U,I}_{n_i,h})\stackrel{d}{\rightarrow }(h,\mu_h^{z,U,I})$ as $i\rightarrow \infty$, the convergence $(h,\mu^{z,U,I}_{n_i,h})\rightarrow (h,\mu_h^{z,U,I})$ also holds in probability.

  \end{proposition}
  We note that in the works constructing the LQG metric, a result (\cite[Section 2.6]{DFGPS20}, \cite[Section 5]{GM19}) analogous to the above  appears for ``weak LQG metrics'', and the proof strategy for Proposition \ref{prop:1} is along the same lines. Indeed, we first use an independence argument to establish (Lemma \ref{lem::37}) that conditional on $h$, the measures $\mu_h\lvert_{I_i}$ for disjoint open intervals $I_i$ are independent, and then use an Efron-Stein argument to show that in fact, $\mu_h$ must be determined by $h$.

 Now, for an interval $I\subseteq \RR$, we introduce the notation $U_{I,\varepsilon}$ to denote the half disk $\mathbb{D}_{|I|/2+\varepsilon}(m(I))$ %
and $z_I$ to denote the point $m(I)+m(I) i/2$, where $i$ here denotes $\sqrt{-1}$. We now have the following simple lemma. %
\begin{lemma}
  \label{lem::27}
  Let $\mathcal{W}$ denote the countable set of all closed intervals $I\subseteq [-1/2,1/2]$ with dyadic endpoints and let $\cD$ denote the set of positive dyadic rationals. Then the sequence
  \begin{displaymath}
    \left(h,\mu_{n,h},\{\mu^{z_I,U_{I,\varepsilon},I}_{n,h}\}_{I\in \mathcal{W},\varepsilon \in \cD}\right)
  \end{displaymath}
  is tight in $n$, and further, for any subsequential weak limit which we denote by $\left(h,\mu_h,\{\mu_{h,I,\varepsilon}\}_{I\in \mathcal{W},\varepsilon \in \cD}\right)$, we have the a.s.\ equality $\mu_h\lvert_{I}=\mu_{h,I,\varepsilon}$ for all $I\in \mathcal{W}$ and $\varepsilon \in \cD$.%
\end{lemma}
\begin{proof}
  The convergence follows by invoking Proposition \ref{prop:main:1} and taking a diagonal subsequence, and the final equality follows by Lemma \ref{lem:11}.  %
\end{proof}

Using the independence structure in the prelimit, we obtain the following conditional independence result for the measures $\mu_h$.
\begin{lemma}
  \label{lem::23}
 Fix a subsequential weak limit $(h,\mu_h)$, and let $I\subset [-1/2,1/2]$ be an open interval and define $U=\mathbb{D}_{|I|/2}(m(I)),V=\overline{\HH}\setminus U$.  Then the measure $\mu_h\lvert_{[-1/2,1/2]\setminus I}$ is conditionally independent of $(\mu_h\lvert_I,h\lvert_U)$ given $h\lvert_{V}$.
\end{lemma}

\begin{proof}
  Though we have started with only the coupling $(h,\mu_h)$, by applying Lemma \ref{lem::27} and considering a further subsequence, we can in fact define the coupling
  \begin{equation}
    \label{eq:255}
    \left(h,\mu_h,\{\mu_{h,I,\varepsilon}\}_{I\in \mathcal{W},\varepsilon \in \cD}\right).
  \end{equation}
 We begin by noting that since $\mu_h$ has no atoms (Proposition \ref{prop:3}), the measures $\mu_h\lvert_{J}$ converge a.s.\ to $\mu_h\lvert_I$ if the closed intervals $J$ are chosen to be dyadic and increasing to $I$. Similarly, by writing the set $[-1/2,1/2]\setminus I$ as an increasing union of two sequences of disjoint closed intervals with dyadic endpoints, one concludes that in order to prove the lemma, it suffices to show that for any closed intervals $J,K\subseteq [-1/2,1/2]$ with dyadic endpoints satisfying $J\subseteq I$ and $K\cap I =\emptyset$, the measure $\mu_h\lvert_{K}$ is conditionally independent of $(\mu_h\lvert_{J},h\lvert_U)$ given $h\lvert_V$. For the rest of the proof, we fix such intervals $J,K$ and an $\varepsilon\in \cD$ small enough such that $U_{J,\varepsilon}\cap V=\emptyset$ and $U_{K,\varepsilon}\subseteq V$. %

  By using the Markov property (Proposition \ref{prop:18}), we obtain the independent decomposition
  \begin{equation}
    \label{eq:87}
    h\lvert_U=\fh+\tth
  \end{equation}
  where $\fh$ is harmonic in $U$ with Neumann boundary conditions at $I$ and is measurable with respect to $\sigma(h\lvert_{V})$. On the other hand, $\tth$ is a Dirichlet-Neumann GFF in $U$. As in Lemma \ref{lem:5}, we can couple $h$ with a GFF $h'$ satisfying $h'\stackrel{d}{=}h$ and such that if $\fh'$ denotes the harmonic extension obtained by applying the Markov property to $h'$ with the domain $U$, then we have
  \begin{equation}
    \label{eq:1}
    h'\lvert_U-\fh'=\mathtt{h},
  \end{equation}
  and further, $\fh'$ is independent of $h$.
  %
  Now, by the locality of the LQG metric, $\mu_{n,h'}^{z_J,U_{J,\varepsilon},J}$ %
   is measurable with respect to $\sigma(\tth\lvert_{U_{J,\varepsilon}},\fh')\subseteq \sigma(\tth,\fh')$ %
   and similarly, $\mu^{z_K,U_{K,\varepsilon},K}_{n,h}$ is measurable with respect to $\sigma(h\lvert_{U_{K,\varepsilon}})\subseteq \sigma(h\lvert_V)$. %
   As a consequence, we obtain that
\begin{equation}
    \label{eq:98}
    (\tth,\fh',\mu^{z_J,U_{J,\varepsilon},J}_{n,h'}) \text{ and } (h\lvert_V, \mu^{z_K,U_{K,\varepsilon},K}_{n,h} ) 
  \end{equation}
  are independent. %
  By using Lemma \ref{lem::27} with $h$ and $h'$, we obtain that by possibly passing through a further subsequence, we have the convergence %
  \begin{equation}
    \label{eq:178}
    (h,h\lvert_V,h',\fh,\fh',\tth,\mu^{z_J,U_{J,\varepsilon},J}_{n,h'}, \mu^{z_K,U_{K,\varepsilon},K}_{n,h},\mu_{n,h},\mu_{n,h'})\stackrel{d}{\rightarrow}(h,h\lvert_V,h',\fh,\fh',\tth,\mu_{h',J,\varepsilon},\mu_{h,K,\varepsilon},\mu_{h},\mu_{h'}),
  \end{equation}
  and we note that this in particular defines a coupling between $\mu_{h',J,\varepsilon},\mu_{h,K,\varepsilon}$. By using the above along with the fact that independence in preserved under weak limits, \eqref{eq:98} yields that
  \begin{equation}
    \label{eq:100}
    (\tth,\fh', \mu_{h',J,\varepsilon})\text{ and } (h\lvert_V,  \mu_{h,K,\varepsilon})
  \end{equation}
  are independent. Now, by the last part of Lemma \ref{lem::27}, we know that $\mu_{h'}\lvert_J=\mu_{h',J,\varepsilon}$ and $\mu_{h,K,\varepsilon}=\mu_h\lvert_K$ almost surely, and thus the above is equivalent to the statement that
  \begin{equation}
    \label{eq:179}
  (\tth,\fh', \mu_{h'}\lvert_J) ~\mathrm{and}~ (h\lvert_V,\mu_h\lvert_K)  
\end{equation}
are independent.
We now finally show that $\mu_{h}\lvert_J$ is a measurable function of $h\lvert_V$ and $(\fh',\mu_{h'}\lvert_J)$. To see this, we first note that since $\mu_h$ a.s.\ has no atoms, the weak convergence of $\mu_{n,h}$ to $\mu_h$ immediately implies the weak convergence of $\mu_{n,h}\lvert_J$ to $\mu_h\lvert_J$ along the same subsequence. Now, we invoke Proposition \ref{lem::3} and use the fact that the function $\fh-\fh'$ is differentiable in an $\overline{\HH}$-neighbourhood of $J$ to obtain that
  \begin{equation}
    \label{eq:102}
      d(\mu_h\lvert_J)=e^{\gamma'(\fh-\fh')/2}d(\mu_{h'}\lvert_J)
  \end{equation}
  which implies that $\mu_h\lvert_J$ is a measurable function of $\fh-\fh'\in\sigma(h\lvert_{V},\fh')$ and $\mu_{h'}\lvert_J$. Also, recall that $h\lvert_U=\fh+\tth$ and thus $h\lvert_U$ is a measurable function of $h\lvert_{V}$ and $\tth$.
Thus in view of \eqref{eq:179}, \eqref{eq:102} implies that $\mu_h\lvert_{K}$ is conditionally independent of $(\mu_h\lvert_{J},h\lvert_U)$ given $h\lvert_V$. 
\end{proof}

Using the above lemma, we now obtain the main ingredient required to prove that $\mu_h$ is determined by $h$.
\begin{lemma}
  \label{lem::37}
 Fix a subsequential weak limit $(h,\mu_h)$. For some $m\in \NN$ and $i\in [\![1,m]\!]$, let $I_i\subseteq [-1/2,1/2]$ be disjoint open intervals. Then the measures $\{\mu_h\lvert_{I_i}\}_{i\in [\![1,m]\!]}$ are conditionally independent given $h$.
\end{lemma}
\begin{proof}
  By taking $I=I_i$ in Lemma \ref{lem::23} and further conditioning on $\sigma(h\lvert_U)$ with $U$ defined therein, we obtain that for any $i$, $\mu_h\lvert_{[0,1]\setminus I_i}$ is conditionally independent of $\mu_h\lvert_{I_i}$ given $h$. In particular, this implies that for all $i$, $\mu_h\lvert_{\cup_{j\neq i} I_j}$ is conditionally independent of $I_i$ given $h$ which is the same as $\sigma(\mu_h\lvert_{I_j}:j\neq i)$ and $\sigma(\mu_h\lvert_{I_i})$ being conditionally independent given $h$ for all $i$. It is straightforward to check that the above is equivalent to the measures $\{\mu_h\lvert_{I_i}\}_{i\in [\![1,m]\!]}$ being conditionally independent given $h$. 
\end{proof}
We are now ready to complete the proof of Proposition \ref{prop:1}.

\begin{proof}[Proof of Proposition \ref{prop:1}]
By choosing a further sequence, we can work with a coupling $(h,\mu_h^{z,U,I},\mu_h)$ instead of just the coupling $(h,\mu_h^{z,U,I})$. Now, as a consequence of Lemma \ref{lem:11}, it suffices to just show that $\mu_h$ is determined by $h$. Since $\mu_h$ is determined by the values $\mu_h(a,b)$ for all $a<b\in \mathbb{Q}\cap (0,1)$, it suffices to show that $\mu_h(a,b)$ is measurable with respect to $h$ for fixed $a<b\in (0,1)$. Take an $\varepsilon>0$ which will be sent to $0$ at the end of the argument. Since $\mu_h$ has no atoms by Proposition \ref{prop:3}, we can write $\mu_h(a,b)=\sum_{i}\mu_h(a_i,a_{i+1})$ where $a_i=a+i\varepsilon$ and $i$ ranges from $0$ to $(b-a)/\varepsilon$.

As a consequence of Lemma \ref{lem::37}, for any fixed $K\in \NN$, we can use the Efron-Stein inequality to obtain
  \begin{align}
    \label{eq:80}
    \Var(\mu_h(a,b)\wedge K\vert h)&\leq  \EE\left[\sum_i (\mu_h(a_i,a_{i+1})\wedge K)^2\vert h\right]\\\nonumber
    &\leq \EE
      \left[
      \max_i (\mu_h(a_i,a_{i+1})\wedge K)\mu_h(a,b)\vert h
      \right]
  \end{align}
 We now note that since $\mu_h$ a.s.\ has no atoms as a consequence of Proposition \ref{prop:3}, we have the a.s.\ convergence $\max_i (\mu_h(a_i,a_{i+1})\wedge K)\mu_h(a,b)\rightarrow 0$ as $\varepsilon\rightarrow 0$. Further, $\max_i (\mu_h(a_i,a_{i+1})\wedge K)\mu_h(a,b)\leq K \mu_h(a,b)$ and, by Proposition \ref{lem:14}, we know that $\EE\left[\mu_h(a,b)\right]<\infty$. As a result, the conditional version of the dominated convergence theorem applies, and this implies that the right hand side in \eqref{eq:80} converges to $0$ a.s.\ as $\varepsilon\rightarrow 0$. This shows that for each fixed $K>0$, $\Var(\mu_h(a,b)\wedge K\vert h)=0$ almost surely, which implies that there is a measurable function $f_K$ such that $\mu_h(a,b)\wedge K=f_K(h)$ almost surely. Note that since $\mu_h(a,b)<\infty$ a.s., we have $\mu_h(a,b)\wedge K\rightarrow \mu_h(a,b)$ as $K\rightarrow \infty$ almost surely. Thus by defining the measurable function $f$ by $f(h)=\limsup_{K\rightarrow \infty}f_{K}(h)$, we have $\mu_h(a,b)=f(h)$ almost surely and this completes the proof of the measurability of $\mu_h^{z,U,I}$ with respect to $\sigma(h)$.

Thus, we are now in the setting where, along a subsequence, we have the convergence $(h,\mu_{n_i,h}^{z,U,I})\stackrel{d}{\rightarrow}(h,\mu_h^{z,U,I})$ as $i\rightarrow \infty$ and further, $\mu_h^{z,U,I}$ is determined by $h$. It is a general fact that in such a setting (see \cite[Lemma 1.3]{DFGPS20}, \cite[Lemma 4.5]{SS13}), the convergence $(h,\mu_{n_i,h}^{z,U,I})\rightarrow (h,\mu_h^{z,U,I})$ in fact holds in probability as $i\rightarrow \infty$. This completes the proof.

\end{proof}

\section{The proof of Theorem \ref{thm:2} via Shamov's theorem}
\label{sec:endshamov}

We now combine the results of the previous sections to prove Theorem \ref{thm:2}. This will be done via Shamov's axiomatic characterization of Gaussian multiplicative chaos (Proposition \ref{prop:4}). 

\begin{proof}[Proof of Theorem \ref{thm:2}]
  Let $K=2^k$ for some $k\in \NN$, and use $\pi_K$ to denote the scaling map from $[-K,K]$ to $[-1/2,1/2]$. Define the field $\mathtt{h}$ defined as $\mathtt{h}(\cdot)=h(2K\cdot)-h_{2K}(0)$ and note that by Lemma \ref{lem:33}, we have $\mathtt{h}\stackrel{d}{=}h$. Recall from \eqref{eq:117} that the measure $\mu_{n,h}^{z,\overline{\HH},[-K,K]}$ is defined as 
  \begin{equation}
    \label{eq:15}
    \mu_{n,h}^{z,\overline{\HH},[-K,K]}=2^{-n(1-\psi_\gamma(\gamma'/\gamma))}\sum_{u\in \Pi_n\cap [-K,K]}|D^{\gamma}_h(z,u)-D^{\gamma}_h(z,u^+)|^{\gamma'd_\gamma/(2\gamma)}\delta_{u},
  \end{equation}
  and note that this is the same as the measure appearing in \eqref{eq:main}. Now, it is easy to see by using Weyl scaling and the coordinate change formula (Proposition \ref{prop:5}) that, almost surely,
  \begin{align}
    \label{eq:6}
    \pi_K^*(\mu_{n,h}^{z,\overline{\HH},[-K,K]})    &= 2^{(k+1)(1-\psi_\gamma(\gamma'/\gamma))} (2K)^{\gamma'Q/2}e^{\gamma'h_{2K}(0)/2}\mu_{n+k+1,\mathtt{h}}^{z/(2K),\overline{\HH},[-1/2,1/2]}\nonumber\\
                                   &= (2K)^{\gamma'Q_{\gamma'}/2}e^{\gamma'h_{2K}(0)/2}\mu_{n+k+1,\mathtt{h}}^{z/(2K),\overline{\HH},[-1/2,1/2]}.
  \end{align}
  where $\pi_K^*$ denotes the push-forward by the map $\pi_K$. To get the last line above, we have used the deterministic equality
  \begin{equation}
    \label{eq:14}
    2^{(k+1)(1-\psi_\gamma(\gamma'/\gamma))} (2K)^{\gamma'Q/2}= (2K)^{\gamma'Q_{\gamma'}/2},
  \end{equation}
  which can be checked to hold by using the relation $1-\psi_{\gamma}(\gamma'/\gamma)=\gamma'(Q_{\gamma'}-Q)/2$ and $K=2^k$. Similarly, by using the coordinate change rule (see e.g.\ \cite[Theorem 4.3]{SW17}) for $\nu^{\gamma'}_h$, we obtain that, almost surely,
  \begin{equation}
    \label{eq:13}
    \pi_K^*(\nu_h^{\gamma'}\lvert_{[-K,K]})= (2K)^{\gamma'Q_{\gamma'}/2}e^{\gamma'h_{2K}(0)/2} \nu^{\gamma'}_{\mathtt{h}}\lvert_{[-1/2,1/2]}.
  \end{equation}
  Since for any fixed $K=2^k$, the point $z/(2K)\in \HH$ is fixed, \eqref{eq:6} and \eqref{eq:13} imply that that to prove Theorem \ref{thm:2}, we can assume without loss of generality that $K=1/2$ and this is the setting that we shall now work in.
  
 First, we show that $\mu_h=\kappa \nu^{\gamma'}_h\lvert_{[-1/2,1/2]}$ almost surely, where $\kappa$ denotes the constant $\EE|\fB_h(0,1)|^{\gamma'd_\gamma/(2\gamma)}$. To see this, we just check that $\kappa^{-1}\mu_h$ satisfies the conditions of Proposition \ref{prop:4}. Firstly, we know that the random measure $\mu_h$ is measurable with respect to $h$ as a consequence of Proposition \ref{prop:1}. Further, as a consequence of Lemma \ref{lem:14}, we know that for any Borel set $S\subseteq [-1/2,1/2]$, we have $\EE \mu_h(S)=\kappa \EE\nu_h(S)$. Finally, it remains to check that the measure $\mu_h$ satisfies Weyl scaling corresponding to the parameter $\gamma'$, but this was checked in Lemma \ref{lem:38}. Thus, by Proposition \ref{prop:4}, we obtain that $\kappa^{-1}\mu_h=\nu^{\gamma'}_h\lvert_{[-1/2,1/2]}$ almost surely, and this shows that $\mu_h=\kappa \nu^{\gamma'}_h\lvert_{[-1/2,1/2]}$ almost surely.

  By using the above along with Lemma \ref{lem:11}, we obtain that for any subsequential weak limit $(h,\mu_h^{z,\overline{\HH},[-1/2,1/2]},\mu_h)$ of the tight (Proposition \ref{prop:main:1}) sequence $(h,\mu_{n,h}^{z,\overline{\HH},[-1/2,1/2]},\mu_{n,h})$, we in fact have $\mu_h^{z,\overline{\HH},[-1/2,1/2]}=\mu_h=\kappa \nu^{\gamma'}_h\lvert_{[-1/2,1/2]}$ almost surely. As a consequence of the above-mentioned characterization of subsequential limits, we have $(h,\mu_{n,h}^{z,\overline{\HH},[-1/2,1/2]},\mu_{n,h})\stackrel{d}{\rightarrow} (h,\mu_h^{z,\overline{\HH},[-1/2,1/2]},\mu_h)$ as $n\rightarrow \infty$, without the need to consider a subsequence. Finally, by applying Proposition \ref{prop:1}, the above weak convergence is upgraded to a convergence in probability. This completes the proof.
\end{proof}

\section{Appendix 1: Confluence in the half plane setting}
\label{sec:append}
\subsection{One point confluence}
\label{subsec:onept}
The goal of these appendices is to outline the structure of the arguments required to prove Proposition \ref{lem:main:20}. These arguments are analogous to the works \cite{GM20,GPS20} and thus we state the analogues of the corresponding results therein, adapted to the current setting, and emphasize the slight differences in the arguments if present.

We begin with some notation. For compact sets $K_1,K_2\subseteq \overline{\HH}$, we will use $D_h(K_1,K_2)$ to denote $\min_{z_1\in K_1,z_2\in K_2}D_h(z_1,z_2)$. Also, we generalize the notation $\DD_r(z)$ for $r>0$ and points $z\in \overline{\HH}$ by defining $\DD_r(A)=\bigcup_{z\in A}\DD_r(z)$ for any bounded set $A\subseteq \overline{\HH}$. We will often use $\diam K$ to denote the Euclidean diameter of a set $K\subseteq \overline{\HH}$. Regarding LQG metric balls, Given a point $z\in \overline{\HH}$, we use $\cB_s(z)\subseteq \overline{\HH}$ to denote the open metric ball of radius $s$ with respect to $D_h$. Similarly, we use $\cB_s^\bullet(z)$ to denote the filled metric ball targeted at $\infty$, by which we mean the complement in $\overline{\HH}$ of the unique unbounded component of $\overline{\cB_s(z)}^c$. We emphasize here that both $\cB_s(z)$ and $\cB_s^\bullet(z)$ will also contain points from the boundary $\RR$, and in this section, we will use $\partial \cB_s(z)$ and $\partial \cB_s^\bullet(z)$ to denote their topological boundaries as subsets of $\overline{\HH}$ equipped with the Euclidean metric. If we need to consider boundaries when viewed as subsets of $\CC$, we will use the symbol $\partial^\CC$ instead of $\partial$. To give an example of this notation, we have $\partial \overline{\HH}=\emptyset$ but $\partial^{\CC}\overline{\HH}=\RR$.

As discussed in the introduction, there is almost surely a unique $D_h$ geodesic between fixed points $z_1,z_2\in \overline{\HH}$. However, there still do exist exceptional pairs $(z_1,z_2)$ admitting multiple geodesics $\Gamma_{z_1,z_2}$. Nevertheless, it can be shown that there always exists a unique left-most such geodesic $\underline{\Gamma}_{z_1,z_2}$ and right-most such geodesic $\overline{\Gamma}_{z_1,z_2}$ in the sense that these geodesics lie weakly to the left (resp.\ right) of all choices of geodesics $\Gamma_{z_1,z_2}$. This is an analogous statement to \cite[Lemma 2.4]{GM20} and can be shown by the same argument. We now state the half plane analogues of the main results in the work \cite{GM20}.

\begin{proposition}[{\cite[Theorem 1.4]{GM20}}]
  \label{prop:main:4}
  Fix $z\in \overline{\HH}$. Almost surely, for each $0<t<s<\infty$, there is a finite set of geodesics from $z$ to $\partial \cB_t^\bullet (z)$ such that every left-most geodesic from $z$ to $\partial \cB_s^\bullet(z)$ coincides with one of the above finitely many geodesics on the time interval $[0,t]$. In particular, almost surely, there are only finitely many points of $\partial \cB_t^\bullet(z)$ that are hit by left-most geodesics from $z$ to $\partial \cB_s^\bullet(z)$. %
\end{proposition}

\begin{proposition}[{\cite[Theorem 1.3]{GM20}}]
  \label{prop:main:3}
 Fix $z\in \overline{\HH}$. Almost surely, for each radius $s>0$, there exists a radius $t\in (0,s)$ such that any two geodesics from $z$ to points $z'\notin \cB_s(z)$ coincide on the time interval $[0,t]$.
\end{proposition}

As in the above two propositions, we will work with a fixed $z\in \overline{\HH}$ in both this section and the next, and we will not mention this again. The proof of the above two propositions follows the arguments in \cite{GM20}, and for the reader already acquainted with the above work, the proofs of the above proceed via the same barrier argument. The only minor difference in the half plane setting is that the barrier events $E_r(w)$ defined for $w\in \overline{\HH}$ and measurable with respect to $\sigma(h\lvert_{\DD_{5r}(w)})$ now do not have the same probability for all $w\in \overline{\HH}$ and $r>0$ but instead have probability bounded away from zero (Lemma \ref{lem::41}) for all $w,r$ as above. 

Now, we give a road map to the proofs of the above two propositions by stating analogous versions of lemmas from \cite{GM20}, while emphasizing the above-mentioned difference about the probability of $E_r(w)$. For most of the lemmas, we skip the proof if the proof is along the same lines as the corresponding statement from \cite{GM20}. To begin, we state an elementary useful topological lemma for the metric $D_h$.

\begin{lemma}
  \label{lem:top}
  Almost surely, for all $s>0$, the set $(\CC\setminus \partial^{\CC} \cB^\bullet_s(z)) \cup \{\infty\}$ is simply connected as a subset of the Riemann sphere. Also, almost surely, for all $s>0$, the set $\partial \cB_s^\bullet(z)$ does not contain any interval of the real line. 
\end{lemma}
\begin{proof}
  The first statement follows immediately by the definition of a filled metric ball. We argue the second statement by contradiction. Suppose that for some $s>0$ and some interval $J\subseteq \RR$, we have $J\subseteq \partial \cB_s^\bullet(z)$. Since any geodesic $\Gamma_{z,v}$ for points $v\in J$ must satisfy $\Gamma_{z,v}\subseteq \cB_s^\bullet(z)$, it can be seen that there must exist distinct points $v_1,v_2\in J$ such that for any geodesic $\Gamma_{z,v_2}$, we have $v_1\in \Gamma_{z,v_2}$ and in fact, $[v_1,v_2]\subseteq \Gamma_{z,v_2}$, where $[v_1,v_2]$ refers to the closed interval connecting the points $v_1,v_2$. However, since $J\subseteq \partial \cB^\bullet_s(z)$, we also know that $D_h(z,u)=s$ for all $u\in J$. However, this implies that $\ell([v_1,v_2];D_h)=0$ and thereby $D_h(v_1,v_2)=0$ which contradicts the fact (see Proposition \ref{prop:6}) that $D_h$ induces the Euclidean topology on $\overline{\HH}$.
\end{proof}
Due to the above lemma, we can think of $\partial ^{\CC} \cB_s^\bullet(z)$ as a collection of prime ends, and this perspective will be used throughout this section. A crucial ingredient enabling one to efficiently place barriers is a deterministic uniform estimate on the number of bottlenecks present in the boundary of a planar set, and the following is a version of this, adapted to suit the setting of this paper.
\begin{lemma}[{\cite[Lemma 2.15]{GM20}}] 
  \label{lem::38}
There is a constant $A>0$ such that the following is true. Consider $z\in \overline{\HH}$ and any compact connected set $K\subseteq \overline{\HH}$ containing a neighbourhood of $z$ with the additional property that $(\CC\setminus K) \cup \{\infty\}$ is simply connected as a subset of the Riemann sphere. Viewing $\partial^\CC K$ as a collection of prime ends, for $n\in \NN$, let $\mathcal{I}$ be a collection of $n$ arcs of $\partial K\subseteq \partial^{\CC}K$ which intersect only at their endpoints. Then for $C>0$, the number of arcs of $I\in \mathcal{I}$ which can be disconnected from $\infty$ in $\CC$ (or $\overline{\HH}\setminus K$) by a path in $\CC$ (or $\overline{\HH}\setminus K$) of Euclidean diameter at most $Cn^{-1/2}$ is at least
  \begin{equation}
    \label{eq:103}
    \left(1-\frac{A}{C^2}\frac{\mathrm{outrad}(K)^4}{\mathrm{inrad}(K)^2}\right)n,
  \end{equation}
  where we note that $\mathrm{inrad}(K)= \sup\{r>0\colon \DD_r(z)\subseteq K\}$ and $\mathrm{outrad}(K)=\inf\{ r>0:K\subseteq \DD_r(z)\}$.
\end{lemma}
\begin{proof}
  With $K^*$ denoting the reflection of the set $K$ about the real line, consider the set $K\cup K^*$ and apply \cite[Lemma 2.15]{GM20} with this set and with $0$ therein replaced by the point $z$. The required result now follows by using the above along with the fact that $\partial K \cup (\partial K)^*=\partial^{\CC}(K\cup K^*)$.
\end{proof}
The above lemma will be used with the set $K$ being a filled metric ball $\cB_t^\bullet(z)$ and yields the existence of sufficiently many bottlenecks on $\partial \cB_t^\bullet(z)$ around which one can place barriers. In order to place barriers with positive probability, the harmonic part of the field needs to be controlled around the bottlenecks, and this is done via the Markov property of the GFF. In fact, apart from deterministic sets, the Markov property of the GFF is also true for a certain class of sets called local sets which depend on the GFF only in a ``local'' manner (see for e.g.\ \cite[Section 1.3.3]{WP21}).
\begin{definition}
  A closed set coupled with $h$ is said to be a local set if for each $\overline{\HH}$-open set $U$, the event $\{A\subseteq U\}$ is conditionally independent of $h\lvert_{\overline{\HH}\setminus U}$ given $h\lvert_U$. 
\end{definition}
Now, given an $s>0$, we can define the $\sigma$-algebras $\sigma(h\lvert_{\cB^\bullet_s(z)})= \sigma((h,\phi):\phi\in \cD(\overline{\HH}), \mathrm{supp}(\phi)\subseteq \cB^\bullet_s(z))$ and $\sigma(h\lvert_{\overline{\cB_s(z)}})$ as $\bigcap_{\varepsilon>0} \sigma((h,\phi):\phi \in \cD(\overline{\HH}), \mathrm{supp}(\phi)\subseteq \DD_\varepsilon(\cB_s(z)))$, and with this notation, we have the following lemma.
\begin{lemma}[{\cite[Lemma 2.1]{GM20}}]
  \label{lem::39}
If $\tau$ is a stopping time for the filtration generated by $(\overline{\cB_s(z)},h\lvert_{\overline{\cB_s(z)}})$, then $\overline{\cB_\tau(z)}$ is a local set for $h$. The same is true with $\cB_s^\bullet (z)$ in place of $\overline{\cB_s(z)}$.
\end{lemma}
In fact, for a local set $A$, one can legitimately define (see \cite[Section 3.3]{SS13}) the $\sigma$-algebra $\sigma(h\lvert_A)$, and further, the Markov property holds in general for local sets (see \cite[Proposition 4.11]{WP21}). Though we do not give formal statements, the above allows us to apply the Markov property with the random set $\overline{\HH}\setminus \overline{\cB_\tau(z)}$, and thereby decompose $h\lvert_{\overline{\HH}\setminus \overline{\cB_\tau(z)}}$ as an independent sum of a harmonic function and another GFF. Before going into the details of the barrier construction involved in the proof of Proposition \ref{prop:main:4}, we first state a quantitative version of Proposition \ref{prop:main:4}. Note that we will use the notation $\tau_r(z)$ for $r>0$ to denote the following stopping time with respect to the filtration generated by $\{(\cB_s^\bullet(z),h\lvert_{\cB_s^\bullet(z)})\}_{s\geq 0}$:
\begin{equation}
  \label{eq:104}
  \tau_r=\tau_r(z)=\inf
  \left\{
    s>0:\cB_s^\bullet(z)\not \subseteq \overline{\DD}_r(z)
  \right\}.
\end{equation}
\begin{lemma}[{\cite[Theorem 3.1]{GM20}}]
  \label{lem::40}
 For each $t>0,p\in (0,1),\rr>0$, there exists $N=N(t,p)\in \NN$ such that the following holds for each stopping time $\tau$ for $\{(\cB_s^\bullet(z),h\lvert_{\cB_s^\bullet(z)})\}_{s\geq 0}$ satisfying $\tau\in [\tau_\rr,\tau_{2\rr}]$ almost surely. With probability at least $p$, there are at most $N$ points of $\partial \cB_\tau^\bullet(z)$ which lie on left-most geodesics from $z$ to $\partial \cB^\bullet_{\tau+t \rr^{\xi Q}e^{\xi h_\rr(z)}}(z)$.
\end{lemma}
The Euclidean scale parameter $\rr>0$ in the above statement will be present throughout as in the work \cite{GM20}. We note that it is technically possible to not have the parameter $\rr$ present, since the scale invariance of the LQG metric is now established, as opposed to the setting of ``weak LQG metrics'' in \cite{GM20}. Nevertheless, we choose to keep the parameter $\rr$ since the statements of the lemmas can then be formulated more closely to their counterpart lemmas from \cite{GM20}, thereby making it easier to compare the statements.

We note that the proof of Proposition \ref{prop:main:4} assuming Lemma \ref{lem::40} proceeds identically to the proof of the corresponding statement in \cite{GM20}; we now move on to the barrier construction.
\subsection{Shield events}
\label{subsec:shield}
With $\widetilde \cB_s(z),\widetilde \cB_s^\bullet(z)$ denoting unfilled and filled metric balls for the LQG metric corresponding to a whole plane GFF, the proof of confluence in \cite{GM20} proceeds by exposing the filled metric ball $\widetilde \cB_s^\bullet(z)$ for $z\in \CC$ grown till a particular time $s$ and considering any collection of $N$ disjoint arcs on the boundary of the above ball. The approach is to use that the filled metric ball is a local set to build ``shields'' around the $N$ arcs such that with good probability, a constant fraction of the $N$ arcs have shields around them. The important property is that if a shield is present around an arc $I\subseteq \partial \widetilde \cB_s^\bullet(z)$, then no geodesic $\eta$ from $0$ to a point outside $\widetilde \cB_s^\bullet(z)$ can pass through $I$. The same approach works in the setting of the free boundary GFF. However, there are some minor technical differences in defining the corresponding shield events $E_r^U(w),E_r(w)$ from \cite[Section 3.2]{GM20} in this setting since unlike the setting of the whole plane GFF which, when viewed modulo an additive constant, is invariant under all translations in $\CC$, this is not so for the free boundary GFF. To be specific, for a point $w\in \overline{\HH}$ close enough to the boundary $\RR$, the Euclidean disk of radius $r$ around $w$ might not even be a subset of $\overline{\HH}$, and for this reason, unlike the setting of \cite{GM20}, we cannot ensure that the shield events $E_r(w)$ here have the same probability for every $w$, but instead, we require that their probabilities be bounded below uniformly in $w\in \overline{\HH}$.

Since our definitions of disks $\DD_s(w)$ and annuli $\A_{s,t}(w)$ for $s<t$ already include an intersection with $\overline{\HH}$, the definition of the shield events $E_r(w),E_r^U(w)$ appear verbatim the same as in \cite[Section 3.2]{GM20}. The definition proceeds as follows. For $\varepsilon>0,w\in \overline{\HH}$ and a set $V\subseteq \overline{\HH}$, define
\begin{equation}
  \label{eq:105}
  \cS_\varepsilon^w(V)=
  \left\{
    [x,x+\varepsilon]\times [y,y+\varepsilon]:(x,y)\in \varepsilon\ZZ^2+w,
    \left(
      [x,x+\varepsilon]\times [y,y+\varepsilon]
    \right)\cap V\neq \emptyset
  \right\}.
\end{equation}
For $w\in \overline{\HH}$ and $\delta\in (0,1)$, we define $\cU_r(w)=\cU_r(w;\delta)$ to be the set $\overline{\HH}$-open subsets $U$ of $\A_{3r,4r}(w)$ such that $\A_{3r,4r}(w)\setminus U$ is a finite union of sets of the form $S\cap \A_{3r,4r}(w)$ for $S\in \cS^w_{\delta r}(\A_{3r,4r}(w))$. For $U\in \cU_r(w;\delta)$ and $\varepsilon>0$, define
\begin{equation}
  \label{eq:96}
  U_\varepsilon=
  \left\{
    u\in U: \dist(w,\partial U)>\varepsilon
  \right\}.
\end{equation}
Now, for $w\in \overline{\HH},r>0,\delta,c\in (0,1),A>0$ and $U\in \cU_r(w;\delta)$, define the event $E_r^U(w)=E_r^U(w;c,\delta,A)$ to be the event such that the following occur.
\begin{enumerate}
\item $D_h(\TT_{2r}(w),\TT_{3r}(w))\geq cr^{\xi Q} e^{\xi h_r(w)}$.
\item $\max_{S\in \cS_{\delta r}^w(\A_{3r,4r}(w))}\sup_{u,v\in S}D_h(u,v;\A_{2r,5r}(w))\leq \frac{c}{100}r^{\xi Q}e^{\xi h_r(w)}$.
  \item Let $\fh^U$ denote the harmonic part of $h\lvert_U$ obtained by using the Markov property (see the next paragraph) to decompose $h$. Then $\sup_{u\in U_{\delta r/4}}|\fh^U(u)-h_r(w)|\leq A$.%
\end{enumerate}
Finally, we define $E_r(w)=E_r(w;c,d,A)=\bigcap_{U\in \cU_r(w;\delta)}E_r^U(w)$. We note that the Markov property used in (3) above does not exactly fit into the statement from Proposition \ref{prop:18}, and we now briefly discuss how this can be handled. First, since (3) depends only on $h$ viewed modulo an additive constant, we can assume that $h$ is normalized so that $h_{r_0}(0)=0$ for some $r_0$ large enough such that $U\subseteq \DD_{r_0}(0)$; this is done just to avoid complicating the Markov property due to the semi-circle with average $0$ potentially intersecting non-trivially with $U$. Now, there are two cases, with the first one being the one where $\partial U\cap \RR=\emptyset$. In this case, instead of using the Markov property from Proposition \ref{prop:18}, we use the usual Markov property decomposing a GFF into a Dirichlet GFF plus a harmonic extension (see e.g.\ \cite[Lemma 2.2]{GMS18} for the whole plane analogue).

In the second case, we have $\partial U\cap \RR\neq \emptyset$, and in this case, we decompose $h\lvert_U$ as a harmonic extension plus an independent Dirichet-Neumann part. However, the domain $U$ might not be a half-disk as assumed in Proposition \ref{prop:18}. In fact, the definition of the Dirichlet-Neumann GFF given earlier can be generalized to more general domains $U$ (see e.g.\ \cite{IK13,QW18}). One way to define the above is to consider the set $U^*$ defined as $\{\overline{z}:z\in U\}$ and define the Dirichlet-Neumann GFF on $U$ to be the ``even'' part of the Dirichlet GFF on $U\cup U^*$, and this yields a field with Neumann boundary conditions on $\partial U\cap \RR$ and Dirichlet boundary conditions on other boundaries of $U$. We do not give more details, but refer the reader to \cite[Definition 6.31]{BP23} for an example of this approach. %

Now, with the definition of the barrier events $E_r(w)$ at hand, we uniformly lower bound their probability.

  \begin{lemma}[{\cite[Lemma 3.2]{GM20}}]
    \label{lem::41}
    For each $p\in (0,1)$, we can find parameters $c,\delta\in (0,1)$ and $A>0$ such that we have $\PP(E_r(w))\geq p$ uniformly in $w\in \overline{\HH}$ and $r>0$.
  \end{lemma}
  \begin{proof}
    The uniformity in $w\in \overline{\HH}$ is the only aspect that is different in the current setting compared to \cite{GM20}. We only describe how to handle it for the case of condition $(1)$ in the definition of $E_r(w)$, and (2), (3) involve similar ideas. That is, we show that for any $p\in (0,1)$, we can choose $c$ small enough such that
    \begin{equation}
      \label{eq:252}
      \PP
    \left(
      D_h(\TT_{2r}(w), \TT_{3r}(w))\geq cr^{\xi Q} e^{\xi h_r(w)}
    \right)\geq p
    \end{equation}
    uniformly in $w\in \overline{\HH}$ and $r>0$.  First note that by an application of Weyl scaling along with the translational and rescaling symmetries from Lemma \ref{lem:33}, it suffices to show that
    \begin{equation}
      \label{eq:115}
    \PP
    \left(
      e^{-\xi h_1(w)}D_h
      \left(
        \TT_{2}(w),\TT_{3}(w)
      \right)\geq c
    \right)\geq p  
    \end{equation}
    uniformly in $w\in \overline{\HH}$ with $\mathrm{Re}(w)=0$. Now, with $h_\mathrm{wp}$ denoting a whole plane GFF (see for e.g.\ \cite[Section 6.4.1]{BP23}) normalized to have average zero on unit circle in $\CC$, it can be shown that, when viewed modulo additive constants, $h\lvert_{\mathbb{D}_3(w)}$ is mutually absolutely continuous to $h_{\mathrm{wp}}\lvert_{\mathbb{D}_3(w)}$ for all $w$ with $\mathrm{Im}(w)>4$, and as $\mathrm{Im}(w)\rightarrow \infty$, the total variation distance between the two goes uniformly converges to $0$ (see e.g.\ \cite[Theorem 6.26]{BP23}). As a result of this, by using the whole plane version of \eqref{eq:115} proved in \cite{GM20}, we can restrict to showing \eqref{eq:115} for the case when $\mathrm{Im}(w)\in [0,4]$.

To do so, we simply note that
    \begin{equation}
      \label{eq:116}
      \inf_{|w|\leq 4,w\in \overline{\HH},\mathrm{Re}(w)=0} e^{-\xi h_1(w)}D_h
      \left(
        \TT_{2}(w),\TT_{3}(w)
      \right)>0
    \end{equation}
    almost surely simply because $D_h$ is a continuous metric (Proposition \ref{prop:5}) and the circle averages $h_1(w)$ are a.s.\ continuous in $w\in \overline{\HH}$, as was discussed in Lemma \ref{lem:cty}. Combining this with the previous paragraph and choosing $c$ to be small enough yields the uniformity required in \eqref{eq:115}.
  \end{proof}
  We note that in \cite{GM20}, the analogous events $E_r(w)$ had the same probability for all $r>0$ and $w\in \CC$ due to the scaling and translational symmetries of the whole plane GFF. In contrast, for the free boundary GFF, not all points $w\in \overline{\HH}$ are equivalent in this regard, and in particular, boundary points cannot be mapped to bulk points via translations/scalings.
  
The following result states that with positive probability, components of $U$ in $\A_{2r,5r}(w)$ have small $D_h$ diameter. The proof is exactly the same as the corresponding result in \cite{GM20}.
  \begin{lemma}[{\cite[Lemma 3.3]{GM20}}]
    \label{lem::42}
    For any choice of the parameters $c,\delta,A$, there is a constant $\fp=\fp(c,\delta,A)>0$ such that the following holds. Let $r>0$, $w\in \overline{\HH}$ and $U\in \cU_r(w)=\cU_r(w;\delta)$. Let $\cV(U)\subseteq \cU_r(w)$ be the set of connected components of $U$. Then almost surely,
    \begin{equation}
      \label{eq:106}
      \PP
      \left(
        \max_{V\in \cV(U)}\sup_{u,v\in V}D_h
        \left(
          u,v;\A_{2r,5r}(w)
        \right)\leq \frac{c}{2}r^{\xi Q}e^{\xi h_r(w)}\vert h\lvert_{\overline{\HH}\setminus U}, E_r^U(w)
      \right)\geq \fp
    \end{equation}
  \end{lemma}
  \subsection{Cutting off geodesics from arcs on boundaries}
Since there is no understanding of the field off a filled metric ball in the case $\gamma\neq \sqrt{8/3}$, the barrier events $E_r(w)$ are in fact placed globally at a dyadic set of points $w$, and we now discuss this. For a given $\rr>0$, let $\rho_\rr^0(w)=\rr$ and inductively define
  \begin{equation}
    \label{eq:107}
    \rho_\rr^i(w):=\inf
    \left\{
      r\geq 6\rho_\rr^{i-1}(w):r=r^k\rr \text{ for some } k\in \ZZ, \textrm{ the event }E_r(w) \text{ occurs}
    \right\},
  \end{equation}
  and we now have the following lemma which can be proved by using Lemma \ref{lem::41} and iterating across annuli (Proposition \ref{prop:1}) followed by a union bound.
  \begin{lemma}[{\cite[Lemma 3.5]{GM20}}]
    \label{lem::43}
    There exists a choice of parameters $c,\delta\in (0,1)$ and $A>0$ and another parameter $\eta>0$ such that for each compact set $K\subseteq \overline{\HH}$, it holds with probability $1-O_\varepsilon(\varepsilon^2)$ (at a rate depending on $K$) that $\rho_{\varepsilon\rr}^{\eta \log\varepsilon^{-1}}(w)\leq \varepsilon^{1/2}\rr$ for all $w\in 
    \left(
      \frac{\varepsilon\rr}{4}\ZZ^2 \cap \DD_{\varepsilon\rr}(\rr K)
    \right)$.
  \end{lemma}
   The parameters $c,\delta,A,\eta$ are henceforth fixed to satisfy Lemma \ref{lem::43}. For $\varepsilon>0$ and a compact set $K\subseteq \overline{\HH}$, define
  \begin{equation}
    \label{eq:108}
    R_\rr^\varepsilon(K)=6\sup
    \left\{
      \rho_{\varepsilon\rr }^{\eta \log \varepsilon^{-1}}(w):w\in 
      \left(
        \frac{\varepsilon\rr }{4}\ZZ^2\cap \DD_{\varepsilon\rr }( K)
      \right)
    \right\}+\varepsilon\rr .
  \end{equation}
By Lemma \ref{lem::43}, for each fixed choice of $K$, $\PP
\left(
  R_\rr^\varepsilon(K)\leq (6\varepsilon^{1/2}+\varepsilon)\rr
\right)\rightarrow 1$ as $\varepsilon \rightarrow 0$. For $s>0$, define
\begin{equation}
  \label{eq:109}
  \sigma_{s,\rr}^\varepsilon= \inf
  \left\{
    s'>s:\DD_{R_\rr^\varepsilon(\cB_s^\bullet(z))}(\cB_s^\bullet(z))\subseteq \cB_{s'}^\bullet(z)
  \right\}.
\end{equation}
The following lemma shows that with good conditional probability, barrier events do in fact occur in the neighbourhood of any chosen point $x$ on the boundary of a filled metric ball.
\begin{lemma}[{\cite[Lemma 3.6]{GM20}}]
  \label{lem::44}
  There exist constants $\alpha>0, C>1$ such that the following holds. Let $\rr>0$ and let $\tau$ be a stopping time for the filtration $
  \left\{
    (\cB_s^\bullet(z),h\lvert_{\cB_s^\bullet(z)})
  \right\}_{s\geq 0}$ and let $w\in \partial \cB_\tau ^\bullet(z)$ and $\varepsilon \in (0,1)$ be chosen in a way depending only on $(B_\tau^\bullet(z), h\lvert_{\cB_\tau^\bullet(z)})$. There is an event $G_w^\varepsilon \in \sigma
  \left(
    \cB^\bullet_{\sigma_{\tau,\rr}^\varepsilon}(z),h\lvert_{\cB_{\sigma_{\tau,\rr}^\varepsilon}^\bullet(z)}
  \right)$ with the following properties.
  \begin{enumerate}
  \item If $R_\rr^\varepsilon(\cB_\tau^\bullet(z))\leq \diam\cB_\tau^\bullet(z)$ and $G_w^\varepsilon$ occurs, then no geodesic from $z$ to a point in $\overline{\HH}\setminus \DD_{R_\rr^\varepsilon(\cB_\tau^\bullet(z))}(\cB_\tau^\bullet(z))$ can enter $\DD_{\varepsilon}(w)\setminus \cB^\bullet_\tau(z)$. %
  \item Almost surely, $\PP
    \left(
      G_w^\varepsilon\vert \cB_\tau^\bullet(z), h\lvert_{\cB_{\tau}^\bullet(z)}
    \right)\geq 1-C_0\varepsilon^{\alpha}$.
  \end{enumerate}
\end{lemma}
The proof of the above lemma is the same as that of the corresponding statement in \cite{GM20} and uses Lemma \ref{lem::42} along with Lemma \ref{lem::39} and the Markov property of the GFF. We now state an upgraded version of the above lemma which disconnects whole arcs on a filled metric ball boundary instead of just points. This uses Lemma \ref{lem::38} and proceeds analogously to its counterpart statement in \cite{GM20}.
\begin{lemma}[{\cite[Lemma 3.7]{GM20}}]
  \label{lem::45}
  There exist constants $\alpha>0, C>1$ such that the following holds. Let $\rr>0$ and let $\tau$ be a stopping time for the filtration $
  \left\{
    (\cB_s^\bullet(z),h\lvert_{\cB_s^\bullet(z)})
  \right\}_{s\geq 0}$. Viewing $\partial^{\CC} \cB_\tau ^\bullet(z)$ as a collection of prime ends, let $I\subseteq \partial \cB_\tau ^\bullet(z)\subseteq \partial^{\CC}\cB_{\tau}^\bullet(z)$ be an arc and $\varepsilon \in (0,1)$ be chosen in a way depending only on $(B_\tau^\bullet(z), h\lvert_{\cB_\tau^\bullet(z)})$, such that $I$ can be disconnected from $\infty$ in $\overline{\HH}\setminus \cB_\tau^\bullet(z)$ by a set of Euclidean diameter at most $\varepsilon$. There is an event $G_I \in \sigma
  \left(
    \cB_{\sigma_{\tau,\rr}^\varepsilon}(z),h\lvert_{\cB_{\sigma_{\tau,\rr}^\varepsilon}^\bullet(z)}
  \right)$ with the following properties.
  \begin{enumerate}
  \item If $R_\rr^\varepsilon(\cB_\tau^\bullet(z))\leq \diam\cB_\tau^\bullet(z)$ and $G_I$ occurs, then no geodesic from $0$ to a point in $\overline{\HH}\setminus \DD_{R_\rr^\varepsilon(\cB_\tau^\bullet(z))}(\cB_\tau^\bullet(z))$ can pass through $I$. %
  \item Almost surely, $\PP
    \left(
      G_I\vert \cB_\tau^\bullet(z), h\lvert_{\cB_{\tau}^\bullet(z)}
    \right)\geq 1-C_0\varepsilon^{\alpha}$.
  \end{enumerate}
\end{lemma}
The above lemma will be used to obtain confluence on a high probability regularity event $\cE_\rr(a)$ which is defined for a fixed choice of $z\in \overline{\HH}$, $\rr>0$ and $a\in (0,1)$ to satisfy the following conditions.
\begin{enumerate}
\item $\DD_{a\rr}(z)\subseteq \cB^\bullet_{\tau_\rr}(z)$.
\item $\tau_{3\rr}(z)-\tau_{2\rr}(z)\geq a\rr^{\xi Q}e^{\xi h_1(z)}$.
\item $\rr^{-\xi Q} e^{-\xi h_\rr(z)}D_h(u,v)\leq (|u-v|/\rr)^{\chi_2}$ for each $u,v\in \DD_{4\rr}(z)$ with $|u-v|\leq a\rr$, where $\chi_2$ is the constant from Proposition \ref{prop:6}.
\item $\rho_{\varepsilon\rr}^{\lfloor \eta \log \varepsilon^{-1}\rfloor} (w)\leq \varepsilon^{1/2}\rr$ for each $w \in (\varepsilon\rr/4) \ZZ^2 \cap \DD_{4\rr}(z)$ and each dyadic $\varepsilon \in (0,a]$.
\end{enumerate}
We now have the following lemma regarding the probability $\PP(\cE_{\rr}(a))$.
\begin{lemma}[{\cite[Lemma 3.8]{GM21}}]
  \label{lem:23}
  For each $p\in (0,1)$, there exists $a=a_p>0$ such that $\PP(\cE_\rr(a))\geq p$ for every $\rr>0$. 
\end{lemma}
We now state a further quantitative confluence statement which is stronger than Lemma \ref{lem::40}.
\begin{proposition}[{\cite[Theorem 3.9]{GM20}}]
  \label{prop:main:5}
  For every $a\in (0,1)$, there is a constant $b_0>0$ depending only on $a$ and constants $b_1,\beta>0$ such that the following holds. For each $\rr>0$, $N\in \NN$ and each stopping time $\tau $ for $
  \left\{
    (\cB_s^\bullet(z),h\lvert_{\cB_s^\bullet(z)})
  \right\}_{s\geq 0}$ with $\tau \in [\tau_\rr,\tau_{2\rr}]$ a.s., the probability that $\cE_\rr(a)$ occurs and that there are more than $N$ points of $\partial \cB_\tau^\bullet(z)$ which are hit by left-most geodesics from $0$ to $\partial \cB_{\tau +N^{-\beta} \rr^{\xi Q}e^{\xi h_\rr(0)}}^\bullet(z)$ is at most $b_0e^{-b_1N^{\beta}}$. 
\end{proposition}
The proof of the above uses that on the event $\cE_\rr(a)$, the outradius and inradius of the relevant filled metric balls around $z$ are approximately comparable-- this allows for an application of Lemma \ref{lem::38}. After this, the events $G_I$ are considered for each arc $I\in \cI$ coming from Lemma \ref{lem::38} and then an iteration argument depending on whether $G_I$ occurs for a given arc or not. We do not give the full details which are exactly analogous to the proof of \cite[Theorem 3.9]{GM20}. However, we do state a couple of lemmas to give a flavour of the above-mentioned iteration argument.

We first introduce some notation. Let $\cI_0$ be a collection of disjoint boundary arcs of $\cB_\tau^\bullet(z)$ chosen in a way depending only on $(\cB_\tau^\bullet(z), h\lvert_{\cB_\tau^\bullet(z)})$. Now inductively define for each $k\in \NN$ a radius $s_k\geq s_{k-1}$ and a finite set of disjoint arcs of $\partial \cB_{s_k}^\bullet(z)$ chosen in a way depending only on $(\cB_{s_k}^\bullet(z), h\lvert_{\cB_{s_k}^\bullet(z)})$ and satisfying $\#\cI_k\leq \#\cI_{k-1}$ as follows. Start with $s_0=\tau$ and $\cI_0$ as defined above. Now assuming that $s_{k-1}$ and $\cI_{k-1}$ have been defined, let $\varepsilon_{k-1}$ be the smallest dyadic number satisfying $\varepsilon_{k-1}\geq (\# \cI_{k-1})^{-1/4}$ and define $s_k=\sigma_{s_{k-1},\rr}^{\varepsilon_{k-1}}$. For $I\in \cI_{k-1}$, let $I'$ be the set of points $w\in \partial \cB_{s_k}^\bullet(z)$ for which the left-most geodesic from $0$ to $w$ passes via $I$ and define $\cI_k=
\left\{
  I':I\in \cI_{k-1},I'\neq \phi
\right\}$. For $N\in \NN$, define $K_N=\min 
\left\{
  k\in \NN\cup \{0\}:s_k>\tau_{3\rr} \text{ or } n_k<N
\right\}$. We now have the following lemma.
\begin{lemma}[{\cite[Lemma 3.10]{GM20}}]
  \label{lem::46}
  Let $k_0=0$ and for $j\in \NN$, inductively define $k_j$ as the smallest $k\geq k_{j-1}$ for which $n_k\leq \frac{1}{2}n_{k_{j-1}}$. There exist constants $C>1$ and $\tilde{\beta}>0$ such that if $N_0$ is chosen to be sufficiently large depending only on $a$, then for $j\in \NN\cup\{0\}$ and $M>1$,
  \begin{equation}
    \label{eq:110}
    \PP
    \left(
      k_{j+1}-k_j>M,k_{j+1}\leq K_{N_0},\cE_\rr(a)\vert \cB_{s_{k_j}}^\bullet(z), h\lvert_{\cB_{s_{k_j}}^\bullet(z)}
    \right)\leq (n_{k_j}/C)^{-\tilde{\beta}M}.
  \end{equation}
\end{lemma}
The above can now be iterated to obtain the following result.
\begin{lemma}[{\cite[Lemma 3.11]{GM20}}]
  \label{lem::47}
  For each $a\in (0,1)$, there are constants $b_0,b_1,\beta>0$ as in the statement of Proposition \ref{prop:main:5} such that for each $\rr>0$ and $N\in \NN$,
  \begin{equation}
    \label{eq:111}
    \PP
    \left(
      s_{K_N}>\tau +N^{-\beta}\rr^{\xi Q} e^{\xi h_\rr (0)}, \cE_\rr (a)
    \right)\leq b_0e^{-b_1N^\beta}.
  \end{equation}
\end{lemma}
This completes the proof outline of Proposition \ref{prop:main:5} and thus also for Lemma \ref{lem::40} and Proposition \ref{prop:main:4}. We now state a few lemmas which are used in the proof of Proposition \ref{prop:main:3}. The proof strategy is to again use a barrier argument to show that conditionally on the metric ball, with positive probability, it can be ensured that out of any collection of given arcs, only one arc survives. Then on using the finiteness coming from Proposition \ref{prop:main:4}, one obtains that Proposition \ref{prop:main:3} holds with positive probability instead of almost surely. Finally, a zero-one law argument involving the triviality of the tail sigma algebra of the GFF at a point upgrades the positive probability statement to an almost sure statement. We now introduce some notation and then simply state the lemmas that can be used to execute the above strategy. Exactly as in (2.18) of \cite{GM20}, given a simply connected domain $U\subseteq \CC$, we consider $\partial U$ as a set of prime ends and define for $X\subset U$, the prime end closure $\mathrm{Cl}'(X)$ of $X$ with respect to $U$ as the set of points in $v\in U\cup \partial U$ with the property that if $\phi\colon U\rightarrow \DD$ is a conformal map, then $\phi(v)$ lies in $\overline{\phi(X)}$. Further, for $u,v\in U\cup \partial U$, we define the metric
\begin{equation}
  \label{eq:194}
  d^U(u,v)=\inf\{\mathrm{diam} (X): X \subseteq {U} \textrm{ is connected and }u,v\in \mathrm{Cl}'(X)\}.
\end{equation}
We use $B_\varepsilon(u;d^U)$ to denote the $\varepsilon$ metric ball around $u$ with respect to the above metric. We note that the above discussion can be repeated for domains $U\subseteq \CC$ with the property that $\CC\setminus U$ is compact and simply-connected since we can consider $U\cup \{\infty\}$ to be a simply connected subset of the Riemann sphere. With the above definitions at hand, we have the following lemma.
\begin{lemma}[{\cite[Lemma 4.1]{GM20}}]
  \label{lem::48}
  For every $A>1,\rr>0,\varepsilon\in (0,(A-1)/100)$ and $p\in (0,1)$, there exists $\fp=\fp(A,\varepsilon,p)>0$ such that the following holds. Let $I\subseteq \partial \cB_{\tau_\rr}^\bullet(z)\subseteq \partial^\CC\cB_{\tau_\rr}^\bullet(z)$ be a closed arc chosen in a way depending only on $(\cB_{\tau_\rr}^\bullet(z), h\lvert_{\cB_{\tau_\rr}^\bullet(z)})$, with the property that $B_{\varepsilon}
  \left(
    \partial \cB_{\tau_\rr}^\bullet(z) \setminus I; d^{\CC\setminus \cB_{\tau_\rr}^\bullet(z)}\right)$ does not disconnect $I$ from $\infty$ in $\overline{\HH}\setminus \cB_{\tau_\rr}^\bullet(z)$. With probability at least $p$, it holds with conditional probability at least $\fp$ given $(\cB_{\tau_\rr}^\bullet(z),h\lvert_{\cB_{\tau_\rr}^\bullet(z)})$ that every geodesic from $z$ to a point of $\overline{\HH}\setminus \DD_{A\rr}(z)$ passes via $I$.
\end{lemma}
We do not provide more details on the proof of the above lemma, but its proof is almost the same as in \cite{GM20}, just like the all the other results in this appendix. By using the above lemma, it can be shown that the required coalescence at $z$ occurs for left-most geodesics with positive probability.
\begin{lemma}[{\cite[Lemma 4.3]{GM20}}]
  \label{lem::50}
  There exists $M>1$ and $\fq>0$, such that for each $\rr>0$, it holds with probability at least $\fq$ that there is only one point $x\in \partial \cB_{\tau_\rr}^\bullet(z)$ which is hit by every left-most geodesic from $0$ to $\partial \cB_{\tau_{M\rr}}^\bullet(z)$.
\end{lemma}
In fact, the above result can be upgraded to all geodesics emanating from $z$ instead of just left-most ones.
\begin{lemma}[{\cite[Lemma 4.4]{GM20}}]
  \label{lem::51}
  There exists $M>1$ and $\fq>0$ such that for each $\rr>0$, it holds with probability at least $\fq$ that any two geodesics from $0$ to a point of $\overline{\HH}\setminus \cB_{\tau_{M\rr}}^\bullet(z)$ coincide on the time interval $[0,\tau_\rr]$. 
\end{lemma}
Finally, by using the tail triviality of the GFF around fixed points, the above can be upgraded to yield Proposition \ref{prop:main:3}.

\subsection{Confluence for filled metric balls targeted at finite points}
\label{sec:confl-fill-metr}
We now give a short discussion of confluence for filled metric balls targeted at points of $\overline{\HH}$ as opposed to the boundary point $\infty$. For $w\in \overline{\HH}\cup\{\infty\}$, we define the set $\cB_t^{w,\bullet}(z)$ to be $\overline{\HH}$ in the case $t\geq D_h(z,w)$ and equal to complement of the connected component of $w$ in the set $\cB_s(z)^c$ in the case $t<D_h(z,w)$. We note that $\cB_t^{\infty,\bullet}(z)=\cB_t^{\bullet}(z)$. We now first state the confluence result corresponding to the sets $\cB_t^{w,\bullet}(z)$ for points $w\in \RR$ and then later extend it to all points $w\in \overline{\HH}$.
\begin{proposition}[{\cite[Proposition 3.6]{GPS20}}]
  \label{prop:12}
  Almost surely, for each $w\in \RR\setminus \{z\}$ and each $0<t<s<D_h(z,w)$, we have the following.
  \begin{enumerate}
  \item There is a finite set of point $\cX^w_{t,s}(z)\subseteq \partial \cB_t^{w,\bullet}(z)$ such that every left-most $D_h$-geodesic from $z$ to a point of $\partial \cB_s^{w,\bullet}(z)$ passes through some $x\in \cX_{t,s}^w(z)$.
  \item There is a unique $D_h$-geodesic from $z$ to $x$ for each $x\in \cX_{t,s}^w(z)$.
  \item For $x\in \cX_{t,s}^w(z)$, let $I_x$ be the set of $y\in \cB_s^{w,\bullet}(z)$ such that the left-most $D_h$-geodesic from $z$ to $y$ passes through $x$. Each $I_x$ for $x\in \cX_{t,s}^w(z)$ is a connected arc of $\partial \cB_s^{w,\bullet}(z)$ and $\partial \cB_s^{w,\bullet}(z)$ is the disjoint union of the arcs $I_x$ for $x\in \cX_{t,s}^w(z)$.
  \end{enumerate}
\end{proposition}
The proof of the above is analogous to the one in \cite{GPS20}. The idea is that due to the conformal covariance of the LQG metric, any boundary point of $\overline{\HH}$ should be ``equivalent'' to $\infty$ after the addition of an appropriate $\log$ singularity. Indeed, if we define the map $\phi\colon \overline{\HH}\cup \{\infty\}\rightarrow \overline{\HH}\cup \{\infty\}$ by $\phi(w)=zw/(w-z)$, then we have $\phi(0)=0$ and further, $\phi$ interchanges $z$ and $\infty$. Further, by the conformal covariance of the LQG metric (see \cite[Theorem 1.3]{GM19+}), if we consider the field $\widetilde{h}$ defined by $\widetilde{h}=h\circ \phi^{-1}+Q \log | (\phi^{-1})'|$, then almost surely, $D_{\widetilde{h}}(\phi(u),\phi(v))=D_h(u,v)$ for all $u,v\in \overline{\HH}$. After this point, an absolute continuity argument can be used to obtain Proposition \ref{prop:12} from the corresponding statement for $w=\infty$.

As an ingredient in the proof of Proposition \ref{lem:main:20}, we shall require Proposition \ref{prop:12} to hold for all values of $w\in \overline{\HH}\setminus\{z\}$ instead of just $w\in \RR\setminus \{z\}$ and we now state this.

\begin{proposition}
  \label{prop:13}
  The statement of Proposition \ref{prop:12} in fact holds for all $w\in \overline{\HH}\setminus \{z\}$.
\end{proposition}
Unfortunately, since the point $\infty$ lies on the boundary of $\HH$ in the Riemann sphere, it cannot be mapped to a bulk point $w\in \HH$ by using a conformal map and thus in order to obtain the above-mentioned upgraded version of Proposition \ref{prop:12}, the arguments in this section need to be repeated with slight modifications. We note that this complication was not present in the corresponding statement \cite{GPS20} for the whole plane case since if we view $\CC\cup \{\infty\}$ as the Riemann sphere, then there exist conformal automorphisms taking any point to any other point. We now do a very quick discussion of the basic set-up in the proof of Proposition \ref{prop:13} but refrain from providing additional details.

Since the case of $w\in \RR\setminus \{z\}$ is covered in Proposition \ref{prop:12}, only the case of bulk points $w\in \HH$ needs to considered in Proposition \ref{prop:13}. Since there exist conformal automorphisms of $\HH$ taking any given bulk point to any other bulk point, it suffices to show the statement for a fixed bulk point. By conformal invariance, we can reduce to the case where the domain is $\{w\in \CC:|z|>1\}$, the bulk point $w$ is $\infty$, and the field is a Neumann GFF $h^{\mathrm{ne}}$ plus an additional log singularity at $\infty$. In fact, by using an absolute continuity argument, we can simply work with the Neumann GFF $h^{\mathrm{ne}}$ without any additional log singularity. By Proposition \ref{prop:6} and the conformal invariance of the Neumann GFF, the LQG metric $D_{h^{\mathrm{ne}}}$ can be extended continuously to the boundary $\{w\in \CC: |w|=1\}$. We now note that the setting now is very similar to the one in \cite{GM20} with the difference being that the space is now $\{w\in \CC:|w|>1\}$ instead of $\CC$. The same proof strategy from \cite{GM20}, as discussed in Section \ref{sec:append}, works and yields Proposition \ref{prop:13}.

Before finally closing this section, we state a result which yields confluence for all geodesics starting from $z$ instead of just left-most or right-most geodesics. The proof is analogous to that of the corresponding statement in \cite{GPS20}.
\begin{proposition}[{\cite[Proposition 3.7]{GPS20}}]
  \label{prop:15}
  Fix $0<t<s$ and $w\in \overline{\HH}\setminus \{z\}$. Consider the set of confluence points $\cX_{t,s}^w(z)$ and the associated arcs as defined in Propositions \ref{prop:13}, \ref{prop:12}. Almost surely, on the event $\{s<D_h(z,w)\}$, the following holds. For every $D_h$-geodesic $P$ from $0$ to a point in $\overline{\HH}\setminus \cB_s^{w,\bullet}(z)$, there is a point $x\in \cX_{t,s}^w(z)$ such that $P(t)=x$ and $P(s)$ is a point of the arc $I_x$ which is not one of the endpoints of $I_x$.
\end{proposition}

\section{Appendix 2: Strong confluence in the half plane setting}
\label{sec:appendix:strong}
In this short appendix, we discuss how the arguments from \cite{GPS20} can be adapted to obtain Proposition \ref{lem:main:20} from Propositions \ref{prop:main:4},\ \ref{prop:main:3}. In \cite{GPS20}, the confluence statements from \cite{GM20}, which are about geodesics started from a fixed point $z$, are extended to confluence statements from a neighbourhood of a fixed point $z$. We now state the half plane analogue of the above ``strong'' confluence statement from \cite{GPS20}. 
\begin{proposition}[{\cite[Theorem 1.2]{GPS20}}]
  \label{prop:14}
Almost surely, for each $\overline{\HH}$-neighbourhood $U$ containing $z$, there is an $\overline{\HH}$-neighbourhood $U'\subseteq U$ of $z$ and a point $w\in U\setminus U'$ such that every $D_h$-geodesic from a point in $U'$ to a point in $\overline{\HH}\setminus U$ passes through $w$.
\end{proposition}
We note that in the above, as usual, $U$ being an $\overline{\HH}$-neighbourhood of $z$ means that $U$ is $\overline{\HH}$-open and $z\in U$. As in \cite{GPS20}, the proof strategy of the above proposition involves a deterministic topological argument which goes along the same lines as the proof in \cite[Proposition 12]{AKM17} for the corresponding statement in the Brownian map. Even in the half plane setting considered in this paper, the details of the proof stay almost the same, with a very slight modification needed to handle the potential intersections of geodesics with the boundary (see Lemma \ref{lem::54}).

We now state the main lemmas involved in the proof and mention the corresponding lemma from \cite{GPS20} in each statement. The following two results follow by using Proposition \ref{prop:13} and Proposition \ref{prop:15} in the exact same manner as in the corresponding proofs in \cite{GPS20}.

\begin{lemma}[{\cite[Lemma 3.8]{GPS20}}]
  \label{lem::52}
 Almost surely, the following holds for each $u\in \overline{\HH}$ such that the geodesic from $z$ to $u$ is unique. For each $\overline{\HH}$-open set $U$ containing $u$, there exists an $\overline{\HH}$-open set $U'\subseteq U$ containing $u$ such that each geodesic from $z$ to a point in $U'$ coincides with the geodesic from $z$ to $u$ outside of $U$.
\end{lemma}

\begin{lemma}[{\cite[Lemma 3.9]{GPS20}}]
  \label{lem::53}
 Almost surely, the following holds for each $v\in \overline{\HH}$ and each geodesic $\Gamma_{z,v}\colon [0,D_h(z,v)]\rightarrow \overline{\HH}$ from $z$ to $v$. For each $0<r<D_h(z,v)$, the segment $\Gamma_{z,v}\lvert_{[0,r]}$ is the only geodesic from $z$ to $\Gamma_{z,v}(r)$.
\end{lemma}
By a rational approximation argument, one can obtain the following lemma.
\begin{lemma}[{\cite[Lemma 3.10]{GPS20}}]
  \label{lem::54}
 Almost surely, for all $v\in \overline{\HH}$, and all $0<t<s<D_h(z,v)$ and for each $\delta>0$, there are points $z_+,z_-\in \QQ^2\cap \cB_\delta (\Gamma_{z,v}(t))$ and $v_+,v_-\in \QQ^2\cap \cB_\delta(\Gamma_{z,v}(s))$ satisfying the following properties.
  \begin{enumerate}
  \item We have
  \begin{equation}
    \label{eq:113}
    \Gamma_{z',v'}\setminus \Gamma_{z,v}\subseteq 
      \cB_\delta(\Gamma_{z,v}(t))\cup \cB_\delta(\Gamma_{z,v}(s))
  \end{equation}
  for all $z'\in \{z_+,z_-\},v'\in\{v_+,v_-\}$.
  \item If $\eta$ is any path from $\cB_\delta( \Gamma_{z,v}(t))$ to $\cB_\delta( \Gamma_{z,v}(s))$ such that the Hausdorff distance between $\eta$ and $\Gamma_{z,v}\lvert_{[t,s]}$ is at most $\delta$, then both $\eta\cap \Gamma_{z',v'}\cap \cB_\delta( (\Gamma_{z,v}(t))$ and $\eta\cap \Gamma_{z',v'}\cap \cB_\delta( (\Gamma_{z,v}(s))$ must be non-empty for at least one $z'\in \{z_+,z_-\}$ and $v'\in \{v_+,v_-\}$.
  \end{enumerate}
\end{lemma}
\begin{proof}
  If $\Gamma_{z,v}(t)\in \RR$, we choose the points $z_-,z_+$ such that $\Gamma_{z,v}(t)\in (z_-,z_+)$ and if $\Gamma_{z,v}(t)\notin \RR$, we choose $z_-,z_+$ on different ``sides'' of $\Gamma_{z,v}$ as in \cite{GPS20}. We define $v_-,v_+$ in a similar manner. To obtain (1), we follow the proof in \cite{GPS20}. Finally, (2) can be seen to hold by planarity.
\end{proof}
 We note that though the proof goes along the same lines, the formulation of the above lemma is slightly different from \cite{GPS20}-- indeed, we consider geodesics $\Gamma_{z',v'}$ for all four pairs $(z',v')$ instead of just two geodesics $\Gamma_{z_+,v_+},\Gamma_{z_-,v_-}$. The reason for the above is that a priori, the geodesic $\Gamma_{z,v}$ above might have intersections with $\RR$ and thus the set $\overline{\HH}\setminus \Gamma_{z,v}$ might have a beaded topology as opposed to the case from \cite{GPS20}, where $\Gamma_{z,v}$ is a simple curve in the plane. This makes it inconvenient to define the two ``sides'' of the curve $\Gamma_{z,v}$ and thus we work with the above alternate lemma. %
 The following lemma is now a straightforward consequence of Lemma \ref{lem::54}.

\begin{lemma}[{\cite[Lemma 3.11]{GPS20}}]
  \label{lem::55}
 Almost surely, the following is true simultaneously for each $u\in \overline{\HH}$, each $\overline{\HH}$-neighbourhood $U$ of $u$ and each $\overline{\HH}$-neighbourhood $V$ of the fixed point $z$. Let $P$ be a geodesic from $z$ to $u$ and let $\{P_n\}_{n\in \NN}$ be a sequence of geodesics which converge uniformly to $P$. Then $P_n\setminus (U\cup V)\subseteq P$ for large enough $n\in \NN$. 
\end{lemma}
By an Arzela-Ascoli argument, it is not difficult to obtain the following from the above lemma.
\begin{lemma}[{\cite[Lemma 3.12]{GPS20}}]
  \label{lem::56}
Almost surely, for each $u\in \overline{\HH}$, each $\overline{\HH}$-neighbourhood $U$ of $u$ and each $\overline{\HH}$-neighbourhood $V$ of the fixed point $z$, there are $\overline{\HH}$-open sets $U',V'$ with $u\in U'\subseteq U$ and $z\in V'\subseteq V$ such that every geodesic from a point of $U'$ to a point of $V'$ coincides with a geodesic from $u$ to $z$ outside of $U\cup V$.
\end{lemma}
Finally, by a compactness argument as in \cite{GPS20}, the above lemma can be used to prove Proposition \ref{prop:14}. By using Proposition \ref{prop:14} with $U=\DD_R(z)$ for a large value of $R$, one can obtain Proposition \ref{lem:main:20} by arguing as in \cite[Lemma 4.8]{GPS20}.

\printbibliography
\end{document}